\documentclass[12pt]{amsart}
\usepackage{etex}
\usepackage{amssymb}
\usepackage{amsxtra}
\usepackage{mathrsfs}
\usepackage{graphics}
\usepackage{latexsym}
\usepackage{xcolor}
\usepackage{diagrams}
\usepackage{amsmath}
\usepackage{amssymb,amsthm,amsfonts}
\usepackage{mathtools}
\usepackage{amscd}
\usepackage[arrow, matrix, curve]{xy}
\usepackage{syntonly}

\usepackage{tikz}
\usepackage{tikz-cd}
\usepackage{amsmath}
\usepackage{amsfonts}
\usepackage{braket}
\usepackage{adjustbox}
\usetikzlibrary{matrix, calc, arrows}

\title[]{Semistable Higgs bundles, periodic Higgs bundles and representations of algebraic fundamental groups}
\author[Guitang Lan]{Guitang Lan}
\email{lan@uni-mainz.de; zuok@uni-mainz.de}
\address{Institut f\"{u}r  Mathematik, Universit\"{a}t
	Mainz, Mainz, 55099, Germany}

\author[Mao Sheng]{Mao Sheng}
\email{msheng@ustc.edu.cn}
\address{School of Mathematical Sciences,
	University of Science and Technology of China, Hefei, 230026, China}

\author[Kang Zuo]{Kang Zuo}

\begin{document}
\theoremstyle{plain}
\newtheorem{thm}{Theorem}[section]
\newtheorem{theorem}[thm]{Theorem}
\newtheorem{lemma}[thm]{Lemma}
\newtheorem{corollary}[thm]{Corollary}
\newtheorem{proposition}[thm]{Proposition}
\newtheorem{addendum}[thm]{Addendum}
\newtheorem{variant}[thm]{Variant}
\theoremstyle{definition}
\newtheorem{lemma and definition}[thm]{Lemma and Definition}
\newtheorem{construction}[thm]{Construction}
\newtheorem{notations}[thm]{Notations}
\newtheorem{question}[thm]{Question}
\newtheorem{problem}[thm]{Problem}
\newtheorem{remark}[thm]{Remark}
\newtheorem{remarks}[thm]{Remarks}
\newtheorem{definition}[thm]{Definition}
\newtheorem{statement}[thm]{Statement}
\newtheorem{claim}[thm]{Claim}
\newtheorem{assumption}[thm]{Assumption}
\newtheorem{assumptions}[thm]{Assumptions}
\newtheorem{properties}[thm]{Properties}
\newtheorem{example}[thm]{Example}
\newtheorem{conjecture}[thm]{Conjecture}
\newtheorem{proposition and definition}[thm]{Proposition and Definition}
\numberwithin{equation}{thm}

\newcommand{\pP}{{\mathfrak p}}
\newcommand{\sA}{{\mathcal A}}
\newcommand{\sB}{{\mathcal B}}
\newcommand{\sC}{{\mathcal C}}
\newcommand{\sD}{{\mathcal D}}
\newcommand{\sE}{{\mathcal E}}
\newcommand{\sF}{{\mathcal F}}
\newcommand{\sG}{{\mathcal G}}
\newcommand{\sH}{{\mathcal H}}
\newcommand{\sI}{{\mathcal I}}
\newcommand{\sJ}{{\mathcal J}}
\newcommand{\sK}{{\mathcal K}}
\newcommand{\sL}{{\mathcal L}}
\newcommand{\sM}{{\mathcal M}}
\newcommand{\sN}{{\mathcal N}}
\newcommand{\sO}{{\mathcal O}}
\newcommand{\sP}{{\mathcal P}}
\newcommand{\sQ}{{\mathcal Q}}
\newcommand{\sR}{{\mathcal R}}
\newcommand{\sS}{{\mathcal S}}
\newcommand{\sT}{{\mathcal T}}
\newcommand{\sU}{{\mathcal U}}
\newcommand{\sV}{{\mathcal V}}
\newcommand{\sW}{{\mathcal W}}
\newcommand{\sX}{{\mathcal X}}
\newcommand{\sY}{{\mathcal Y}}
\newcommand{\sZ}{{\mathcal Z}}
\newcommand{\A}{{\mathbb A}}
\newcommand{\B}{{\mathbb B}}
\newcommand{\C}{{\mathbb C}}
\newcommand{\D}{{\mathbb D}}
\newcommand{\E}{{\mathbb E}}
\newcommand{\F}{{\mathbb F}}
\newcommand{\G}{{\mathbb G}}
\renewcommand{\H}{{\mathbb H}}
\newcommand{\I}{{\mathbb I}}
\newcommand{\J}{{\mathbb J}}
\renewcommand{\L}{{\mathbb L}}
\newcommand{\M}{{\mathbb M}}
\newcommand{\N}{{\mathbb N}}
\renewcommand{\P}{{\mathbb P}}
\newcommand{\Q}{{\mathbb Q}}
\newcommand{\Qbar}{\overline{\Q}}
\newcommand{\R}{{\mathbb R}}
\newcommand{\SSS}{{\mathbb S}}
\newcommand{\T}{{\mathbb T}}
\newcommand{\U}{{\mathbb U}}
\newcommand{\V}{{\mathbb V}}
\newcommand{\W}{{\mathbb W}}
\newcommand{\Z}{{\mathbb Z}}
\newcommand{\g}{{\gamma}}
\newcommand{\bb}{{\beta}}
\newcommand{\as}{{\alpha}}
\newcommand{\Id}{{\rm Id}}
\newcommand{\rk}{{\rm rank}}
\newcommand{\END}{{\mathbb E}{\rm nd}}
\newcommand{\End}{{\rm End}}
\newcommand{\Hom}{{\rm Hom}}
\newcommand{\Hg}{{\rm Hg}}
\newcommand{\tr}{{\rm tr}}
\newcommand{\Sl}{{\rm Sl}}
\newcommand{\Gl}{{\rm Gl}}
\newcommand{\Cor}{{\rm Cor}}

\newcommand{\SO}{{\rm SO}}
\newcommand{\OO}{{\rm O}}
\newcommand{\SP}{{\rm SP}}
\newcommand{\Sp}{{\rm Sp}}
\newcommand{\UU}{{\rm U}}
\newcommand{\SU}{{\rm SU}}
\newcommand{\SL}{{\rm SL}}
\newcommand{\Spec}{\mathrm{Spec}}
\newcommand{\Spf}{\mathrm{Spf}}
\newcommand{\ra}{\rightarrow}
\newcommand{\xra}{\xrightarrow}
\newcommand{\la}{\leftarrow}
\newcommand{\Nm}{\mathrm{Nm}}
\newcommand{\Gal}{\mathrm{Gal}}
\newcommand{\Res}{\mathrm{Res}}
\newcommand{\GL}{\mathrm{GL}}

\newcommand{\GSp}{\mathrm{GSp}}
\newcommand{\Tr}{\mathrm{Tr}}

\newcommand{\bA}{\mathbf{A}}
\newcommand{\bK}{\mathbf{K}}
\newcommand{\bM}{\mathbf{M}} 
\newcommand{\bP}{\mathbf{P}}
\newcommand{\bC}{\mathbf{C}}
\newcommand{\HIG}{\mathrm{HIG}}
\newcommand{\MIC}{\mathrm{MIC}}
\newcommand{\Aut}{\mathrm{Aut}}


\thanks{The first and third named authors are supported by the SFB/TR 45 `Periods, Moduli Spaces and Arithmetic of Algebraic Varieties' of the DFG. The second named author is supported by National Natural Science Foundation of China (Grant No. 11622109, No. 11626253) and the Fundamental Research Funds for the Central Universities.\\
1. Higgs-de Rham flow was called Higgs-de Rham sequence in \cite{LSZ2012}. The change has been made because of its analogue to the notion of Yang-Mills-Higgs flow over $\C$.}

\begin{abstract}
Let $k $ be the algebraic closure of a finite field of odd
characteristic $p$ and  $X$ a smooth projective scheme over the
Witt ring $W(k)$ which is geometrically connected in characteristic zero. 
We introduce the notion of \emph{Higgs-de Rham
flow}\footnote{A Higgs-de Rham flow was previously called a Higgs-de
Rham sequence in \cite{LSZ2012}. The change has been made in the
current version because of its analogue to the notion of Yang-Mills
flow over $\C$.} and prove that the category of periodic Higgs-de
Rham flows over $X/W(k)$  is equivalent to the category of
Fontaine modules, hence further equivalent to the category
of crystalline representations of the  \'{e}tale fundamental group $\pi_1(X_K)$ of the generic fiber of $X$, after Fontaine-Laffaille and Faltings. Moreover, we prove that every semistable Higgs bundle over the special fiber $X_k$ of $X$ of rank $\leq p$ initiates a semistable Higgs-de Rham flow and thus those of rank $\leq p-1$ with trivial Chern classes induce $k$-representations of $\pi_1(X_K)$. A fundamental construction in this paper is the inverse Cartier transform over a truncated Witt ring. In characteristic $p$, it was constructed by Ogus-Vologodsky in the nonabelian Hodge theory in positive characteristic;  in the affine local case, our construction is related to the local Ogus-Vologodsky correspondence of Shiho.
\end{abstract}

\maketitle

\tableofcontents

\section{Introduction}
Let $k$ be the algebraic closure of a finite field of odd characteristic $p$, $W:=W(k)$ the ring of Witt vectors and $K$ its fraction field.
Let $X$ be a smooth projective scheme over $W$ which is geometrically connected in characteristic zero. The paper is aimed to establish a
correspondence between certain Higgs bundles over $X$ with trivial Chern classes which are stable over its special fiber $X_k:=X\times_Wk$ and
certain integral crystalline representations which are absolutely irreducible modulo $p$ of the \'{e}tale fundamental group $\pi_1(X_K)$ of the generic fiber $X_K:=X\times_WK$. \\

Inspired by the complex analytic theory of Simpson \cite{Si92}, Ogus and Vologodsky \cite{OV} has established the nonabelian Hodge theorem in positive characteristic, 
that is an equivalence of categories between a category of certain nilpotent Higgs modules and a category of certain nilpotent flat modules over a smooth variety over 
$k$ which is $W_2(k)$-liftable ($\textrm{char}\ k=2$ is also allowed). This equivalence generalizes both the classical Cartier descent theorem and the relation between a strict $p$-torsion Fontaine module and the associated graded Higgs module. Compared with the complex theory, an obvious distinction is that the semistability condition on Higgs modules does not 
play any role in Ogus-Vologodsky's correspondence. In the meantime, Faltings and others have been attempting to establish an analogue of Simpson's theory for varieties 
over $p$-adic fields. In  \cite{Fa3}, Faltings has established  an equivalence of categories between the category of generalized representations of the geometric fundamental 
group and the category of Higgs bundles over a $p$-adic curve, which has generalized the earlier work of Deninger-Werner \cite{DW} on a partial $p$-adic analogue of 
Narasimhan-Seshadri theory. However, the major problem  concerning the role of semistability remains open; Faltings asked whether semistable Higgs bundles of degree 
zero come from  genuine representations instead of   merely generalized ones. We will address  the semistability condition in this paper.\\

Recall that, when $X$ is a complex smooth projective
variety, such a correspondence has been established first by Hitchin
\cite{Hitchin} for  polystable Higgs bundles
over a smooth projective curve and then by Simpson \cite{Si88} in general.
The key step in the Hitchin-Simpson correspondence, as one may see in the
following diagram, is to solve the Yang-Mills-Higgs equation and
obtain a  correspondence from graded Higgs bundles to polarized complex variations of Hodge structure over $X$. \vspace{3mm}

\hspace*{-3em} 
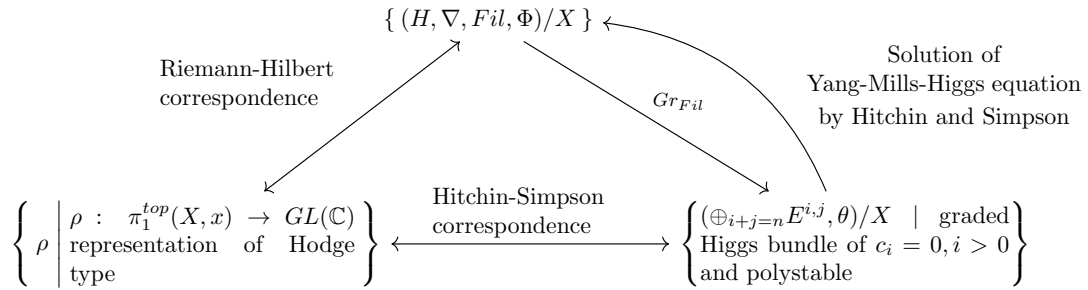
\begin{adjustbox}{scale=0.9}
\begin{tikzpicture}[commutative diagrams/every diagram]

\matrix[matrix of nodes, name=a] {
&
& [-3em] |[name=a12]| {\footnotesize $\Set{(H,\nabla, Fil, \Phi)/X}$}
& \\
\\
\\
& |[name=a21]| {\footnotesize $\Set{\rho|\parbox{12em}{$\rho: \ \pi^{top}_1(X,x)\to GL(\mathbb{C})$ representation of Hodge type}}$}
&
& [-2em] |[name=a23]| {\footnotesize $\left\{\parbox{13em}{$(\oplus_{i+j=n} E^{i,j}, \theta)/X$ $\,\,|\,\,$ graded Higgs bundle of $c_i=0, i>0$ and polystable}\right\}$}
\\
};

\path[commutative diagrams/.cd, every arrow, every label]
(a12) edge[commutative diagrams/rightarrow] node[auto]{\(Gr_{Fil}\)} (a23)
(a12) edge[commutative diagrams/leftarrow, bend left] node[auto, pos=0.82]{\(\parbox{9.8em}{\begin{center}{\footnotesize Solution of \\ Yang-Mills-Higgs equation \\ by Hitchin and Simpson}\end{center}}\)} (a23)
(a21) edge[commutative diagrams/leftrightarrow] node[auto]{\(\parbox{7em}{\footnotesize{Hitchin-Simpson \\ correspondence}}\)} (a23)
(a21) edge[commutative diagrams/leftrightarrow] node[auto]{\(\parbox{7em}{\footnotesize{Riemann-Hilbert \\ correspondence}}\)} (a12)
;

\end{tikzpicture}
\end{adjustbox}


\vspace{3mm}

The notion of polarized $\mathbb{Z}$-variation of Hodge structure $(H,\nabla, Fil,\Phi)$ was introduced by P. Griffiths \cite{G}, and later generalized to polarized  $\mathbb{C}$-variation of Hodge structure by P. Deligne \cite{De87}, where $(H,\nabla)$ is a flat bundle with a Hodge filtration, $\Phi$ is a polarization, which is horizontal with
respect to $\nabla$ and satisfies the Riemann-Hodge bilinear
relations;  $(\oplus_{i+j=n}E^{i,j},\theta)$ denotes a graded Higgs bundle, where $\theta$ is a direct sum of $\sO_X$-linear
morphisms $ \theta^{i,j}: E^{i,j}\to E^{i-1,j+1}\otimes \Omega_{X}$
and $\theta\wedge\theta=0$. \\

We return to the $p$-adic case. A good $p$-adic analogue of the category of polarized complex variations of Hodge structure is the category $MF^{\nabla}_{[0,w]}(X/W) ( w\leq p-1)$ introduced first by Fontaine-Laffaille \cite{FL} for $X=\Spec\ W$ and later generalized by Faltings \cite[Chapter II]{Fa1} to the general case. An object in the category, called Fontaine module, is also a quadruple $(H,\nabla,Fil,\Phi)$, where $(H,\nabla)$ is a flat bundle over $X$, by which we mean a locally free $\sO_X$-module $H$ of finite rank, equipped with an integrable $W$-connection $\nabla$, $Fil$ is a Hodge filtration, $\Phi$ is a relative Frobenius which is horizontal with respect to $\nabla$ and satisfies some compatibility and strong $p$-divisibility conditions. The latter condition is a $p$-adic analogue of the Riemann-Hodge bilinear relations. See \S2 for more details. Fontaine-Laffaille proved in the $X=\Spec\ W$ case and Faltings in the general case that, there exists a fully faithful functor from the category $MF^{\nabla}_{[0,w]}(X/W) (w\leq p-2)$ to the category of \emph{crystalline representations} of $\pi_1(X_{K})$, i.e. a $p$-adic Riemann-Hilbert correspondence. Our objective is to establish a $p$-adic analogue of the Higgs correspondence from a certain category of graded Higgs modules with trivial Chern classes to the category of Fontaine modules. Our results are encoded in the following big diagram:\\

\hspace*{-3em} 
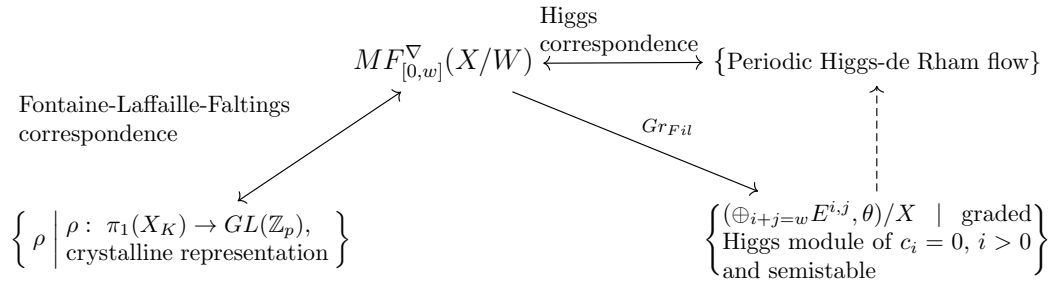
\begin{adjustbox}{scale=0.9}
\begin{tikzpicture}[commutative diagrams/every diagram]

\matrix[matrix of nodes, name=b] {
&
& [-3em] |[name=b12]| {${\footnotesize{MF^\nabla_{[0,w]}(X/W)}}$}
& [3em] |[name=b13]| {\{\footnotesize{Periodic Higgs-de Rham flow}\}} \\
\\
& |[name=b21]| {\footnotesize $\Set{\rho|\parbox{11em}{$\rho:\ \pi_1(X_{K})\to GL(\mathbb{Z}_p)$, \\ crystalline representation}}$}
&
& |[name=b23]| {\footnotesize $\left\{\parbox{13em}{$(\oplus_{i+j=w} E^{i,j}, \theta)/X$ $\,\,|\,\,$ graded Higgs module of $c_i=0$, $i>0$ and semistable}\right\}$}
\\
};

\path[commutative diagrams/.cd, every arrow, every label]
(b12) edge[commutative diagrams/rightarrow] node[auto]{\(Gr_{Fil}\)} (b23)
(b23) edge[commutative diagrams/dashed]node[auto,right]{ } (b13)
(b13) edge[commutative diagrams/leftrightarrow] node[auto,above]{\({\footnotesize \parbox{7em}{Higgs\\ correspondence}}\)} (b12)
(b21) edge[commutative diagrams/leftrightarrow] node[auto, pos=0.45]{\({\footnotesize \parbox{12em}{Fontaine-Laffaille-Faltings\\ correspondence}}\)} (b12)
;

\end{tikzpicture}%
\end{adjustbox}
\vspace{3mm}

As seen from the above diagram, the central notion in our theory is \emph{Higgs-de
Rham flow} (especially the periodic one), which can be viewed as an analogue of Yang-Mills-Higgs flow. To define a Higgs-de Rham flow over $X_n:=X\otimes \Z/p^{n}\Z$
for all $n\in \N$, the key ingredient is the inverse Cartier transform $C_n^{-1}$ over $W_n:=W\otimes \Z/p^n\Z$, see Definition \ref{HDRFn}.  It is built upon the seminal work of Ogus-Vologodsky \cite{OV}. In \cite{OV}, they construct the inverse Cartier transform from the category of (suitably nilpotent) Higgs modules over $X_1'=X_1\times_{F_{k}}\Spec\ k$ to the category of (suitably nilpotent) flat modules over $X_1$. However, over a general smooth scheme $X'$ over $W_n, n\geq 2$, our lifting of inverse Cartier transform operates not on the whole category of (suitably nilpotent) Higgs modules, rather on a category which is over a subcategory of graded Higgs modules (over a proper scheme, there is a restriction on the Chern classes of Higgs modules). More details about the category will be given below. In comparison with the construction of Shiho \cite{Shiho}, one finds that the existence of a Frobenius lifting over a chosen lifting of $X'$ over $W_{n+1}$ is not assumed in our construction. On the other hand, we do not know whether the functor $C_n^{-1}$ is fully faithful for a proper $W_n$-scheme when $n\geq 2$.\\


The Higgs correspondence is established in an inductive way. That is, we shall first define the notion of a Higgs de-Rham flow in characteristic $p$ and then establish the Higgs correspondence for the periodic flows. This is the first step. In this step, we need only assume the scheme $X_1$ to be smooth over $k$ and $W_2$-liftable. A choice of such a lifting does matter in the theory. Let us choose and then fix a lifting $X_2/W_2$. Let $S_n:=\Spec\ W_n$ and $F_{S_n}: S_n\to S_n$ the Frobenius automorphism. Set $X_2':=X_2\times_{F_{S_2}}S_2$ which is a $W_2$-lifting of $X_1'=X_1\times_{F_{S_1}}S_1$. Let $(\sX,\sS)=(X_1/k,X_2'/W_2)$ and $C^{-1}_{\sX/\sS}$ the inverse Cartier transform of Ogus-Vologodsky \cite{OV} which restricts to an equivalence of categories from the full subcategory of nilpotent Higgs modules $HIG_{p-1}(X_1')$ of exponent $\leq p-1$ to the full subcategory of nilpotent flat modules $MIC_{p-1}(X_1)$ of exponent $\leq p-1$. See also our previous work \cite{LSZ} for an alternative approach via the exponential twisting to the Ogus-Vologodsky's theory over these subcategories. Let $\pi: X_1'\to X_1$ be the base change of $F_{S_1}$ to $X_1$. From the geometric point of view, it is more natural to make all terms in a flow defining over the same base scheme. So, instead of using $C^{-1}_{\sX/\sS}$, we take the composite functor $C_1^{-1}:=C^{-1}_{\sX/\sS}\circ \pi^*\circ \iota$ from $HIG_{p-1}(X_1)$ to $MIC_{p-1}(X_1)$, where $\iota$ is an automorphism of the category $HIG_{p-1}(X_1)$ defined by sending $(E,\theta)$ to $(E,-\theta)$. The reader is advised to take caution on this point. For a flat module $(H,\nabla)$ over $X_1$, a \emph{Griffiths transverse filtration} $Fil$ of level $w\geq 0$ on $(H,\nabla)$ is defined to be a finite exhaustive decreasing filtration of $H$ by $\sO_{X_1}$-submodules
\begin{align*}
H=Fil^0\supset Fil^1\supset\cdots\supset Fil^w\supset Fil^{w+1}=0,
\end{align*}
such that its grading $\oplus_i Fil^i/Fil^{i+1}$ is torsion free and such that $Fil$ obeys Griffiths' transversality
\begin{align*}
\nabla(Fil^i)\subset Fil^{i-1}\otimes \Omega_{X_1/k}, \ 1\leq i\leq w.
\end{align*}
The triple $(H,\nabla,Fil)$ is called a \emph{de Rham module}.
By taking the grading  with respect to the filtration $Fil$, to every de Rham module  $(H,\nabla,Fil)$ one can canonically associate a graded Higgs module $(E,\theta):= (\oplus_{i}Fil^i/Fil^{i+1}, \oplus_i\bar{\nabla}_i)$, where the $\sO_{X_1}$-morphism $\bar \nabla_i$ is induced from $\nabla$
$$
\bar{\nabla}_i:\frac{Fil^i}{Fil^{i+1}}\to \frac{Fil^{i-1}}{Fil^{i}}\otimes \Omega_{X_1/k}.
$$

\begin{definition}\label{Higgs-de Rham flow}
A Higgs-de Rham flow over $X_1$ is a diagram of the following form:
$$
\begin{adjustbox}{scale=1.1}
\xymatrix{
                &  (H_0,\nabla_0)\ar[dr]^{Gr_{Fil_0}}       &&  (H_1,\nabla_1)\ar[dr]^{Gr_{Fil_1}}    \\
 (E_0,\theta_0) \ar[ur]^{C_1^{-1}}  & &     (E_1,\theta_1) \ar[ur]^{C_1^{-1}}&&\ldots       }
\end{adjustbox}
$$
where the initial term $(E_0,\theta_0)\in HIG_{p-1}(X_1)$, $Fil_i, i\geq 0$ is a Griffiths transverse filtration on the flat module $(H_i,\nabla_i):=C_1^{-1}(E_i,\theta_i)$ of level $\leq p-1$ and $(E_i,\theta_i),i\geq 1$ is the associated graded Higgs module to the de Rham module $(H_{i-1},\nabla_{i-1},Fil_{i-1})$.
\end{definition}

If the filtrations $Fil_i, i\geq 0$ in the definition are all of level $\leq w$ ($w\leq p-1$), it is said to be a Higgs-de Rham flow of level $\leq w$. For a graded torsion-free Higgs module $(E=\oplus_iE^i,\theta)\in HIG_{p-1}(X_1)$ (where $\theta: E^i\to E^{i-1}\otimes \Omega_{X_1/k}$), there is a natural Griffiths transverse filtration on $(H,\nabla):=C_1^{-1}(E,\theta)$, which is actually nontrivial in general. Since the filtration $\{E_l:=\oplus_{i\leq l}E^i\}_l$ of $E$ is $\theta$-invariant, and since $C_1^{-1}$ is an exact functor, $\{(H_l,\nabla_l):=C_1^{-1}(E_l,\theta)\}_l$ is naturally a filtration of $(H,\nabla)$ and it is Griffiths transverse. But since $\{H_l\}_l$ is $\nabla$-invariant, the associated graded Higgs module has always the zero Higgs field. It is important to observe that there exist other nontrivial Griffiths transverse filtrations than the above one: the Hodge filtration in the geometric case (i.e. the strict $p$-torsion Fontaine modules) and the Simpson filtration in the $\nabla$-semistable case (see Theorem \ref{goodfil}). \\

A Higgs-de Rham flow is said  to be \emph{periodic} if there exists an isomorphism of graded Higgs modules $\phi: (E_f,\theta_f)\cong (E_0,\theta_0)$ (an explicit $\phi$ is a part of the definition); it is said  to be
\emph{preperiodic} if it becomes periodic after removing the first few terms. A Higgs module $(E,\theta)\in HIG_{p-1}(X_1)$ is said to be \emph{(pre)periodic} if it initiates a (pre)periodic Higgs-de Rham flow. The reader shall refer to Definition \ref{definition of periodic flow} for a precise definition. One may visualize a periodic Higgs-de Rham flow of period $f$ via the following diagram:

$$
\xymatrix{
                &  (H_0,\nabla_0)\ar[dr]^{Gr_{Fil_0}}     &&  (H_{f-1},\nabla_{f-1})\ar[dr]^{Gr_{Fil_{f-1}}}    \\
 (E_0,\theta_0) \ar[ur]^{C_1^{-1}}  & &    \cdots\cdot \ar[ur]^{C_1^{-1}}&&  (E_f,\theta_f)\ar@/^2pc/[llll]^{\stackrel{\phi}{\cong} } }
$$

\begin{theorem}[Theorem \ref{correspondence in the type (0,f) case}]\label{thm0}
Let $X_2$ be a smooth scheme over $W_2$. Let $w$ be an integer between $0$ and $p-1$. For each $f\in \N$, there is an equivalence of categories between the full subcategory of strict $p$-torsion Fontaine modules over $X_2/W_2$ of Hodge-Tate weight $\leq w$ with endomorphism structure $\F_{p^f}$ and the category of periodic Higgs-de Rham flows over its special fiber $X_1$ of level $\leq w$ and whose periods are $f$.
\end{theorem}

Next, we construct a lifting to $W_n, n\in \N$ of the inverse Cartier transform of Ogus-Vologodsky \cite{OV} restricted to the full subcategory $HIG_{p-2}(X_1')$. Let $X_n$ be a smooth scheme over $W_n$. We introduce a category $\sH(X'_n)$, where $X'_n:=X_n\times_{F_{S_n}}X_n$, whose object is a tuple $(E,\theta,\bar H,\bar \nabla,\bar{Fil},\bar \psi)$, where $(E,\theta)$ is a nilpotent graded Higgs module over $X'_n$ of exponent $\leq p-2$, $(\bar H,\bar \nabla,\bar{Fil})$ is a de Rham module over $X'_{n-1}$, and $\bar \psi$ is an isomorphism of graded Higgs modules $Gr_{\bar{Fil}}(\bar H,\bar \nabla)\cong (E,\theta)\otimes \Z/p^{n-1}\Z$. For $n=1$, such a tuple is reduced to a nilpotent graded Higgs module over $X_1'$ of exponent $\leq p-2$, and therefore $\sH(X'_1)$ is just the full subcategory $HIG_{p-2}(X_1')$.
\begin{theorem}[Theorem \ref{lifting of inverse cartier}]
Assume $X_n$ is $W_{n+1}$-liftable. There exists a functor $\sC_{n}^{-1}$ from the category $\sH(X_n')$ to the category $MIC(X_n)$ of flat modules over $X_n$ such that $\sC_{n}^{-1}$ lifts $\sC_{n-1}^{-1}$ and such that $\sC_{1}^{-1}$ agrees with the inverse Cartier transform $C^{-1}_{\sX/\sS}$ of Ogus-Vologodsky \cite{OV}.
\end{theorem}
Based on the existence of an inverse Cartier transform over $W_n$, we define inductively a periodic Higgs-de Rham flow over $X_n/W_{n}$ for an arbitrary $n$ (Definition \ref{def for periodic flow over Wn}) and establish the Higgs correspondence over $W_n$ (Theorem \ref{periodic corresponds to FF module over a truncated witt ring}). Passing to the limit, we obtain the $p$-torsion free analogue of Theorem \ref{thm0}:
\begin{theorem}\label{thm1}
Let $X$ be a smooth scheme over $W$. For each integer $0\leq w\leq p-2$ and each $f\in \N$, there is an equivalence of categories between the category of $p$-torsion free Fontaine modules over $X/W$ of Hodge-Tate weight $\leq m$ with endomorphism structure $W(\F_{p^f})$ and the category of periodic Higgs-de Rham flows over $X$ of level $\leq w$ and whose periods are $f$.
\end{theorem}
 
The above theorem (and its $p$-torsion version) and the Fontaine-Laffaille-Faltings correspondence (and its $p$-torsion version, see Theorem \ref{Faltings thm}) together constitute a $p$-adic (and $p$-torsion) version of Hitchin-Simpson correspondence. This correspondence can be regarded as a global and higher Hodge-Tate weight generalization of the Katz's correspondence on unit-root $F$-crystals. Using this correspondence, a $p$-divisible group with isomorphic Kodaira-Spencer map over $W$ is constructed over the Serre-Tate lifting of an ordinary Abelian variety over $k$. See Example \ref{p-divisible group}. \\

Now, we shall bring the semistability condition on Higgs modules into consideration, and explain our results along the vertical dotted line in the above big diagram. For a smooth projective variety $X_1$ over $k$ of positive dimension, we choose (and then fix) an ample divisor $Z_1$ of $X_1$. Let us define the ($\mu$-)semistability for a Higgs module $(E,\theta)$ over $X_1$ with respect to the $\mu_{Z_1}$-slope:
$$
\mu_{Z_1}(E)=c_1(E)Z_1^{\dim X_1-1}/\rk(E).
$$
The slope of a torsion $\sO_{X_1}$-module is set to be infinity. Our first result is the following 
\begin{theorem}[Theorem \ref{quasiperiodic equivalent to strongly semistable}, Theorem \ref{strsst}]\label{thm2}
Notation as above and assume additionally that $X_1$ is $W_2$-liftable. Let $(E,\theta)$ be a torsion free nilpotent Higgs module over $X_1$ of exponent $\leq p-1$. Suppose $\textrm{rank}\ E\leq p$ and $c_i(E)=0, i>0$. Then $(E,\theta)$ is semistable if and only if it is preperiodic. 
\end{theorem}
As a consequence, we obtain the following
\begin{corollary}[Theorem \ref{rank two semistable bundle corresponds to rep}]\label{thm3} 
Let $X/W$ be a smooth projective scheme over $W$. For each nilpotent semistable Higgs bundle $(E,\theta)$ over $X_k$ of $\textrm{rank}\ E\leq p-1$ and $c_i(E)=0, i>0$, one associates to it a unique $\textrm{rank}\ E$ crystalline $k$-representation of $\pi_1(X_K)$ up to isomorphism.
\end{corollary}
The stronger rank condition in the above corollary results from the application of the Fontaine-Laffaille-Faltings correspondence.\\

In the way to lift the previous result to the mixed characteristic situation, a nontrivial obstruction occurs, namely the lifting of a Griffiths transverse filtration in positive characteristic to a truncated Witt ring, which has prevented us from a direct generalization of Corollary \ref{thm3} to the mixed characteristic case. It turns out that, in order to kill the obstruction, one is led to various \emph{ordinary} conditions on the base varieties. By working on this problem for a very simple kind of rank two Higgs bundles of degree zero (the so-called Higgs bundle with maximal Higgs field) over a curve, we have found a $p$-adic analogue of the Hitchin-Simpson's uniformization of hyperbolic curves which relates intimately the above theory to the theory of ordinary curves due to S. Mochizuki (\cite{Mochizuki}). In particular, the canonical lifting theorem of Mochizuki for ordinary curves has been basically recovered in our recent work \cite{LSYZ}.\\

Stability, rather than merely semistability, on periodic Higgs bundles over $k$ makes the choices involved in Higgs-de Rham flows basically unique. By the unicity of one-periodic Higgs-de Rham flow initializing a stable Higgs bundle (Proposition \ref{unique filtration} and Proposition \ref{uniqueness for n}), one is able to identify the category of one-periodic stable Higgs modules with a full subcategory of the category of periodic Higgs-de Rham flows. Combined with this unicity, the Higgs correspondence implies the following rigidity result for Fontaine modules.
\begin{theorem}[Corollary \ref{fully faithful}]
Let $X/W$ be a smooth projective scheme. Let $(H_i,\nabla_i,Fil_i,\Phi_i), i=1,2$ be two $p$-torsion free Fontaine modules over $X/W$, and $(E_i,\theta_i)$ the associated graded Higgs modules. If $(E_i,\theta_i)$ are isomorphic as graded Higgs modules and additionally mod $p$ are Higgs stable, then $(H_i,\nabla_i,Fil_i,\Phi_i), i=1,2$ are isomorphic.
\end{theorem}
To conclude, our final result is the following correspondence. An $\F_{p}$-representation $\rho$ of $\pi_1$ is said to be absolutely irreducible if $\rho\otimes k$ is irreducible. 
\begin{theorem}[Corollary \ref{irreducible over W}]
Notation as above. There is an equivalence of categories between the category of crystalline $\Z_{p}$ (resp. $\Z/p^n\Z$) representations of $\pi_1(X_K)$ with Hodge-Tate weight $\leq p-2$ whose mod $p$ reduction is absolutely irreducible and the category of one-periodic Higgs bundles over $X/W$ (resp. $X_n/W_n$) whose exponent is $\leq p-2$ and mod $p$ reduction is stable.
\end{theorem}
 
In the study of Hitchin-Simpson correspondence, one usually concentrates on the subcategory of Higgs bundles with trivial Chern classes. However, the theory developed in this paper turns out to be also useful in the study of Higgs bundles with nontrivial Chern classes. This has been beautifully demonstrated in the recent work \cite{Langer2} of A. Langer on a purely algebraic proof of the Bogomolov-Giesecker inequality for semistable Higgs bundles in the complex case (\cite[Proposition 3.4]{Si88}) and the Miyaoka-Yau inequality for Chern numbers of complex algebraic surfaces of general type. In his work, the notion of (semistable) Higgs-de Rham flow in characteristic $p$ has played as similar role as the Yang-Mills-Higgs flow over the field of complex numbers.\\

The paper is structured as follows: basically, the paper consists of two parts: the first five sections \S2-\S5 are devoted to establish the theory of the Higgs correspondence between the category of $p$-torsion Fontaine modules with extra endomorphism structures and the category of periodic Higgs-de Rham flows; the last two sections \S6-\S7 and the appendix aim towards applications of the previous theory to produce the representations of $\pi_1$ from (semi)-stable Higgs bundles. In more detail, Section 2 is a preliminary, where we recall the basics of the theory on Fontaine modules; in Section 3, we introduce the notion of a Higgs-de Rham flow in positive characteristic and establish the Higgs correspondence in positive characteristic; in Section 4, we construct a lifting of the inverse Cartier transform of Ogus-Vologodsky over an arbitrary truncated Witt ring; in Section 5, we establish the Higgs correspondence over an arbitrary truncated ring; in Section 6, we introduce the notion of a strongly semistable Higgs module and show that a strongly semistable Higgs module with trivial Chern classes is preperiodic and vice versa, and consequently we produce crystalline representations of the algebraic fundamental groups of the generic fiber with $k$-coefficients from semistable nilpotent Higgs bundles of small ranks over the closed fiber with trivial Chern classes; in Section 7, we prove a rigidity theorem for Fontaine modules whose associated graded Higgs modules are mod $p$ stable; in Appendix A, we (joint with Y.-H. Yang) prove that a semistable Higgs module of small rank is strongly semistable, verifying partially a conjecture in the first version of the paper \cite{LSZ2012}.\\

\section{Preliminaries on Fontaine modules}
The category of Fontaine modules, as introduced by G. Fatlings in \cite{Fa1}, originates from number theory. In the seminal paper \cite{FL}, Fontaine and Laffaille introduced the category $MF^{f,q}$ (resp. $MF^{f,q}_{tor}$) of strongly divisible filtered modules over $W$ (resp. of finite length) and constructed an exact and faithful contravariant functor from the previous category to the category of representations of the Galois group of the local field $K$ (resp. of finite length). A representation lying in the image of the functor is said to be \emph{crystalline}. The significance of the category, as shown in another seminal paper \cite{FM} by Fontaine-Messing, is due to the fact that the crystalline cohomologies of many proper algebraic varieties over $W$ lie in the category. In the above cited paper, Faltings generalized both results to a geometric base (see also \cite{Fa2} for the generalization to the semistable reduction case and the case of a very ramified base ring). For us, this category plays the role connecting a certain category of $p$-adic Higgs modules with a cartain category of $p$-adic representations of algebraic fundamental groups. From the point of view of nonabelian Hodge theory, this category is a nice $p$-adic analogue of polarized complex variations of Hodge structure, a special but important class of the so-called harmonic bundles. One of principal aims of this paper is to establish a correspondence between the category of Fontaine modules and the category of one-periodic Higgs-de Rham flows, in both positive and mixed characteristics. This section is devoted to a brief exposition of this category and related known results.
\begin{remark}
We shall remind the reader of the category of $F$-$T$-crystals developed in the monograph \cite{O}, which is also a $p$-adic analogue of the category of complex variations of Hodge structure (with no emphasis on polarization). This category is intimately related to the category of Fontaine modules (see particularly \cite[Proposition 5.3.9]{O}). It is interesting to relate it to a certain category of Higgs modules. This task has not been touched upon in this paper.
\end{remark}

For clarity, we start with the $p$-torsion free Fontaine modules. Our exposition is based on Ch. II \cite{Fa1} and \S3 \cite{Fa2}. Let $X$ be a smooth $W$-scheme. For an affine subset $U$ flat over $W$, there exists a (nonunique) absolute Frobenius lifting $F_{\hat U}$ on its $p$-adic completion $\hat U$. An object in the category $MF_{[0,w]}^{\nabla}(\hat U)$ is a quadruple $(H, \nabla, Fil, \Phi_{F_{\hat U}})$, where
\begin{itemize}
    \item [i)] $(H,Fil)$ is a filtered free $\sO_{\hat U}$-module
    with a basis $e_i$ of $Fil^i, 0\leq i\leq w$.
    \item [ii)] $\nabla$ is an integrable connection on $H$ satisfying the Griffiths'
    transversality:
    $$
    \nabla(Fil^i)\subset Fil^{i-1} \otimes \Omega_{\hat U}.
    $$
    \item [iii)] The relative Frobenius is an $\sO_{\hat U}$-linear morphism $\Phi_{F_{\hat U}}: F_{\hat U}^*H\to H$ with the strong
    $p$-divisible property: $\Phi_{F_{\hat U}}(F_{\hat U}^*Fil^i)\subset p^iH$
    and
\begin{equation}\label{strong divisibility non p-torsion case}
\sum_{i=0}^{w}\frac{\Phi_{F_{\hat U}}(F_{\hat U}^*Fil^i)}{p^i}=H.
\end{equation}
    \item [iv)] The relative Frobenius $\Phi_{F_{\hat U}}$ is horizontal
    with respect to the connection $F_{\hat U}^*\nabla$ on $F_{\hat U}^*H$ and
    $\nabla$ on $H$.
\end{itemize}
The filtered-freeness in i) means that the filtration $Fil$
on $H$ has a splitting such that each $Fil^i$ is a direct sum of
several copies of $\sO_{\hat U}$. Equivalently, $Fil$ is a finite exhaustive decreasing filtration of free $\sO_{\hat U}$-submodules which is split. The pull-back connection $F_{\hat
U}^*\nabla$ on $F_{\hat U}^*H$ is the connection defined by the formula
$$
F^*_{\hat U}\nabla(f\otimes e)=df\otimes e+f\cdot(dF_{\hat U}\otimes 1)(1\otimes \nabla(e)), \ f\in
\sO_{\hat U}, e\in H|_{\hat U}.
$$
The horizontal condition iv) is expressed by the commutativity of
the diagram
$$
 \xymatrix{
    F_{\hat U}^*H \ar[d]_{F_{\hat U}^*\nabla} \ar[r]^{\Phi_{F_{\hat U}}} &  H \ar[d]^{\nabla} \\
    F_{\hat U}^*H\otimes
\Omega_{\hat U}\ar[r]^{\Phi_{F_{\hat U}}\otimes Id} &  H\otimes
\Omega_{\hat U}. }
$$
As there is no canonical Frobenius liftings on $\hat U$, one must
know how the relative Frobenius changes under another Frobenius
lifting. This is expressed by a Taylor formula. Let $\hat U=\Spf R$
and $F: R\to R$ an absolute Frobenius lifting. Choose a system of \'{e}tale
local coordinates $\{t_1,\cdots,t_d\}$ of $U$ (namely fix an
\'{e}tale map $U\to \Spec(W[t_1,\cdots,t_d])$). Let $R'$ be any
$p$-adically complete, $p$-torsion free $W$-algebra, equipped with a
Frobenius lifting $F': R'\to R'$ and a morphism of $W$-algebras
$\iota: R\to R'$. Then the relative Frobenius $\Phi_{F'}:
F'^*(\iota^*H)\to \iota^*H$ is the composite
$$
F'^*\iota^*H\stackrel{\alpha}{\cong}
\iota^*F^*H\stackrel{\iota^*\Phi_{F}}{\longrightarrow} \iota^*H,
$$
where the isomorphism $\alpha$ is given by the formula:
$$
\alpha(e\otimes 1)=\sum_{\underline i}\nabla_{\partial}^{\underline
i}(e)\otimes \frac{z^{\underline i}}{\underline{i}!}.
$$
Here $\underline{i}=(i_1,\cdots,i_d)$ is a multi-index, and
$z^{\underline i}=z_1^{i_1}\cdots z_d^{i_d}$ with $z_i=F'\circ
\iota(t_i)-\iota\circ F(t_i), 1\leq i\leq d$, and
$\nabla_{\partial}^{\underline j}=\nabla_{\partial_{
t_1}}^{i_1}\cdots\nabla_{\partial_{t_d}}^{i_d}$.\\

Let $\sU=\{U_i\}_{i\in I}$ be an open affine covering of $X$ over $W$. For each $i$, let $F_{\hat{U}_i}$ be an absolute Frobenius lifting over $\hat U_i$. Then, after \cite[Theorem 2.3]{Fa1}, when $w\leq p-1$, the local categories $\{MF_{[0,w]}^{\nabla}(\hat U_i)\}_{i\in I}$ glue into the category $MF_{[0,w]}^{\nabla}(X)$. Its object is called a Fontaine module over $X/W$, and will be denoted again by a quadruple $(H,\nabla,Fil,\Phi)$, although the relative Frobenius $\Phi$ is only \emph{locally} defined and depends on a choice of absolute Frobenius lifting locally. For an open affine $U$ together with an absolute Frobenius lifting $F_{\hat U}$ over $\hat U$, the symbol $\Phi_{(U,F_{\hat U})}$ means the evaluation of $\Phi$ at $(U,F_{\hat U})$.

\begin{example}\label{geometric situation}
Let $f: Y\to X$ be a proper smooth morphism over $W$ of relative dimension
$w\leq p-2$ between smooth $W$-schemes. Assume that the relative Hodge cohomologies $R^if_*\Omega^j_Y, i+j=w$ have no torsion. By Theorem 6.2
\cite{Fa1},  the crystalline direct image $R^wf_{*}(\sO_{Y},d)$ is an object in $MF^{\nabla}_{[0,w]}(X/W)$.
\end{example}

The fundamental theorem of Fontaine-Laffaille (see \cite[Theorem 3.3]{FL} for $X=\Spec \ W$) and Faltings (see \cite[Theorem 2.6*]{Fa1}), which is a $p$-adic analogue of the Riemann-Hilbert correspondence over $\C$, reads:
\begin{theorem}[Fontaine-Laffaille-Faltings correspondence]\label{Faltings thm}
Notation as above. Assume furthermore $X$ is proper over $W$ and $w\leq p-2$. There exists a fully faithful contravariant functor ${\mathbf D}$ from the category $MF^{\nabla}_{[0,w]}(X/W)$ to the category of \'{e}tale local systems over $X_{K}$.
\end{theorem}
The image of the functor $\mathbf D$ is called crystalline sheaves of Hodge-Tate weight $n$ over $X_K$. We remind the reader that the functor $\mathbf D$ in \cite{Fa1} is covariant and its image is the category of dual crystalline sheaves. In the above theorem for torsion free Fontaine modules as well as its $p$-torsion analogue below, we actually use the dual of the functor $\mathbf D$. \\

{\itshape Variant 1: $p$-torsion.} A $p$-torsion Fontaine module is formulated in a similar way. In fact, the previous category $MF^{\nabla}_{[0,w]}(X/W)$ is the $p$-adic limit of its torsion variant (see c)-d). Ch. II \cite{Fa1}). The major modification in the $p$-torsion case occurs to the formulation of strong $p$-divisibility. Note that Equation (\ref{strong divisibility non p-torsion case}) does not make sense in the $p$-torsion case. For each natural number $n$, a strict $p^n$-torsion Fontaine module $(H,\nabla,Fil,\Phi)$ means the following: $H$ is a \emph{finitely generated} $\sO_{X_n}$-module; $Fil$ is a finite exhaustive decreasing filtration of $\sO_{X_n}$-submodules on $H$ satisfying Griffiths' transversality with respect to an integrable connection $\nabla$; $\Phi$ is strongly $p$-divisible, namely, the evaluation of $\Phi$ at $(U,F_{\hat U})$ is an $\sO_{U_n}$-isomorphism ($U_n:=U\otimes \Z/p^n\Z, F_{U_n}:=F_{\hat U}\otimes \Z/p^n\Z$):
$$
\Phi_{(U_n,F_{U_n})}: F_{U_n}^*\tilde H|_{U_n}\cong H|_{U_n},
$$
where $\tilde H$ is the quotient $\oplus_{i=0}^{w}Fil^i/\sim$ with $x\sim py$ for any $x\in Fil^{i}$ and $y$ the image of $x$ under the natural inclusion $Fil^i\hookrightarrow Fil^{i-1}$; the horizontal property for $\Phi$ is formulated as in the non-torsion case, which means explicitly the following commutative diagram:
$$
 \xymatrix{
     F_{U_n}^*\tilde H|_{U_n} \ar[rr]^{\Phi_{(U_n,F_{U_n})}} \ar[d]_{F_{U_n}^*\tilde \nabla}& &  H|_{U_n} \ar[d]^{\nabla} \\
        F_{U_n}^*\tilde H|_{U_n}\otimes
\Omega_{U_n}\ar[rr]^{\Phi_{(U_n,F_{U_n})}\otimes Id} &&  H|_{U_n}\otimes
\Omega_{U_n}, }
$$
where $\tilde \nabla$ is the integral $p$-connection (see Definition \ref{def of p-connection}) over $\tilde H$ induced by $\nabla$ and $F_{U_n}^*\tilde \nabla$ is similarly defined as $F^*_{\hat U}\nabla$ in the non $p$-torsion case by replacing $\nabla$ resp. $dF_{\hat U}$ in the composite therein with $\tilde \nabla$ resp. $\frac{dF_{U_n}}{p}$ (see Formula (\ref{formula for connection}) for a local expression). Let $MF^{\nabla}_{[0,w],n}(X/W)$ denote the category of strict $p^m$-torsion Fontaine modules. The Fontaine-Laffaille-Faltings correspondence as given above is achieved by taking the $p$-adic limit of its $p$-torsion analogue.

\begin{remark}
Using Fitting ideals, Faltings \cite[Theorem 2.1]{Fa1} shows that $(H,Fil)$ is indeed locally filtered free. A slight generalization obtained by A. Ogus using a different method is given in \cite[Theorem 5.3.3]{O}. Note also that, in the formulation of the category $MF^{\nabla}_{[0,w],n}(X/W)$, one actually requires only the existence of a model $X_{n+1}$ over $W_{n+1}$. Although objects of this category are defined over $X_n$, the horizontality of the relative Frobenius uses the operator $\frac{dF_{U_{n+1}}}{p}$, where $F_{U_{n+1}}$ is an absolute Frobenius lifting on an open affine $U_{n+1}\subset X_{n+1}$. Also, the existence of $X_{n+1}$ is required for the sake of the transition of two evaluations of the relative Frobenius via the Taylor formula. Therefore, this category requires only the existence of a smooth $W_{n+1}$-scheme $X_{n+1}$. In this case, we shall denote it by $MF^{\nabla}_{[0,w]}(X_{n+1}/W_{n+1})$.
\end{remark}

{\itshape Variant 2: extra endomorphism.} For our purpose, we need to introduce the category $MF^{\nabla}_{[0,w],f}(X/W)$ of Fontaine modules with endomorphism structure $W(\F_{p^f})$ for each $f\in \N$. It consists of five tuples  $(H,\nabla,Fil,\Phi,\iota)$,
where $(H,\nabla,Fil,\Phi)$ is a torsion-free Fontaine module and
$$
\iota: W(\F_{p^f})\hookrightarrow \End_{MF}(H,\nabla,Fil,\Phi)
$$
is an embedding of $\Z_p$-algebras. A morphism of this category is a morphism in
$MF^{\nabla}_{[0,w]}(X/W)$ respecting the endomorphism structure $\iota$. Clearly, the category for $f=1$ is nothing but the category of Fontaine modules. We introduce similarly its $p$-torsion counterpart $MF^{\nabla}_{[0,w],n,f}(X/W)$ (and $MF^{\nabla}_{[0,w],f}(X_{n+1}/W_{n+1})$), where the extra endomorphism structure is given by an embedding of $\Z/p^n\Z$-algebras:
$$
\iota: W_n(\F_{p^f}) \hookrightarrow \End_{MF}(H,\nabla,Fil,\Phi).
$$

\section{Higgs correspondence in positive characteristic}
Let us begin with the definitions of a (pre)periodic Higgs-de Rham flow and a (pre)periodic Higgs module. Let $X_1$ be a smooth variety over $k$ and $X_2$ a $W_2$-lifting of $X_1$.
\begin{definition}\label{definition of periodic flow}
A preperiodic Higgs-de Rham flow over $X_1$ (with respect to the given $W_2$-lifting $X_2$) is a tuple $(E,\theta,Fil_0,\cdots,Fil_{e+f-1},\phi)$, where $e\geq 0$ and $f\geq 1$ are two integers, $(E,\theta)$ is a Higgs module in the category $HIG_{p-1}(X_1)$, $Fil_i, 0\leq i\leq e+f-1$ is a Griffiths transverse filtration on $C_1^{-1}(E_i,\theta_i)$ where $(E_0,\theta_0)=(E,\theta)$ and
$$(E_i,\theta_i):=Gr_{Fil_{i-1}}(H_{i-1},\nabla_{i-1}), 1\leq i\leq e+f$$ is inductively
defined, and $\phi$ is an isomorphism of graded Higgs modules
$$
(E_{e+f},\theta_{e+f})\cong (E_e,\theta_e).
$$
It is said to be periodic of period $f$ (or $f$-periodic) if the integer $e$ in above is zero. The $(E_i,\theta_i)$s (resp. $(H_i,\nabla_i)$s) are called the Higgs (resp. de-Rham) terms of the flow. A Higgs module $(E,\theta)$ over $X_1$ is said to be (pre)periodic if there exists a (pre)periodic Higgs-de Rham flow with the leading Higgs term $(E,\theta)$.
\end{definition}
One may complete a preperiodic Higgs-de Rham flow over $X_1$ into a Higgs-de Rham flow in a natural way: note that the isomorphism $\phi$ induces the isomorphism $$C_1^{-1}(\phi): C_1^{-1}(E_{e+f},\theta_{e+f})\cong C_1^{-1}(E_e,\theta_e),$$ and therefore one obtains naturally the filtration $Fil_{e+f}$ on $C_1^{-1}(E_{e+f},\theta_{e+f})$ by pulling back $Fil_{e}$ via the isomorphism. Then continue successively.\\

This section aims to establish the Higgs correspondence between the category of periodic Higgs-de Rham flows over $X_1$ and the category of strict $p$-torsion Fontaine modules with extra endomorphism structure. Let us introduce the category $HDF_{w,f}(X_2/W_2)$ as follows : its object is given by a periodic Higgs-de Rham flow $(E,\theta,Fil_0,Fil_1,\cdots,Fil_{f-1},\phi)$ over $X_1$ such that each filtration $Fil_i$ is of level $\leq w$. Note that $(E,\theta)$ in a periodic Higgs-de Rham must be a graded Higgs module.  A morphism between two objects is a morphism of graded Higgs modules respecting the additional structures. As an illustration, we
explain a morphism in the category $HDF_{w,1}(X_2/W_2)$ of one-periodic Higgs-de Rham flow in detail.
Let $(E_i,\theta_i,Fil_i,\phi_i), i=1,2$ be two objects in the category. Then a morphism
$$f: (E_1,\theta_1,Fil_1,\phi_1)\to (E_2,\theta_2,Fil_2,\phi_2)$$ is given by a morphism of graded Higgs modules
$$
f: (E_1,\theta_1)\to (E_2,\theta_2)
$$
such that the induced morphism of flat modules (by the functoriality of $C_1^{-1}$)
$$C_1^{-1}(f): C_1^{-1}(E_1,\theta_1)\rightarrow C_1^{-1}(E_2,\theta_2)$$
preserves the filtrations, and moreover the induced
morphism of graded Higgs modules is compatible with $\phi$s,
that is, the following diagram of natural morphisms commutes:
\begin{equation*}\label{eq1}
 \begin{CD}
 Gr_{Fil_1}C_1^{-1}(E_1,\theta_1)@>\phi_1>>(E_1,\theta_1)\\
 @VGrC_1^{-1}(f)VV@  VVfV\\
  Gr_{Fil_2}C_1^{-1}(E_2,\theta_2)@>\phi_2>>(E_2,\theta_2).
 \end{CD}
\end{equation*}
Recall that $MF^{\nabla}_{[0,w],f}(X_2/W_2)$ is the category of strict $p$-torsion Fontaine modules with extra endomorphism $\F_{p^f}$. The Higgs correspondence in positive characteristic is the following
\begin{theorem}\label{correspondence in the type (0,f) case}
Notation as above. Let $w\leq p-1$ and $f$ be a natural number. Then there is an equivalence of categories between the category $MF^{\nabla}_{[0,w],f}(X_2/W_2)$ and the category $HDF_{w,f}(X_2/W_2)$.
\end{theorem}

We take an open covering $\{U_i\}$ of $X_2/W_2$ consisting of open affine subsets which are smooth over $W_2$, together with an absolute Frobenius
lifting $F_{U_i}$ on each $U_i$. By modulo $p$, one obtains an open affine covering $\{U_{i,1}\}$ for $X_1$. We show first
a special case of the theorem, namely the $f=1$ case.
\begin{proposition}\label{correspondence in the type (0,1) case}
There is an equivalence of categories between the category of strict $p$-torsion Fontaine modules and the category of one-periodic Higgs-de Rham flows over $X_1$.
\end{proposition}
For simplicity, we denote $MF$ for $MF^{\nabla}_{[0,w],1}(X_2/W_2)$ and $HDF$ for $HDF_{w,1}(X_2/W_2)$. In the following paragraph, we shall construct two functors
$$\mathcal{GR}: MF\to HDF, \quad \mathcal{IC}: HDF\to MF,
$$
and then show they are quasi-inverse to each other. For an $(H,\nabla,Fil,\Phi)\in MF$, let $(E,\theta):=Gr_{Fil}(H,\nabla)$ be the associated graded Higgs module. The following lemma gives the first functor.
\begin{lemma}\label{from Faltings to Higgs the fixed point case}
There is a filtration $Fil_{\exp}$ on $C_1^{-1}(E,\theta)$ together
with an isomorphism of graded Higgs modules
$$
\phi_{\exp}: Gr_{Fil_{exp}}(C_1^{-1}(E,\theta))\cong (E,\theta),
$$
which is induced by the filtration $Fil$ and the relative
Frobenius $\Phi$.
\end{lemma}
\begin{proof}
By \cite[Proposition 1.4]{LSZ}, the relative Frobenius
induces an isomorphism of flat modules
$$
\tilde \Phi: C_1^{-1}(E,\theta)\cong (H,\nabla).
$$
So we define $Fil_{exp}$ on $C_1^{-1}(E,\theta)$ to be the inverse
image of $Fil$ on $H$ by $\tilde \Phi$. It induces tautologically an
isomorphism of graded Higgs modules
$$
\phi_{\exp}=Gr(\tilde \Phi): Gr_{Fil_{exp}}(C_1^{-1}(E,\theta))\cong
(E,\theta).
$$
\end{proof}
Next, the functor $C_1^{-1}$ induces the second functor $\mathcal{IC}$ as follows. Given an object $(E,\theta,Fil,\phi)\in HDF$, we define the triple by
$$(H,\nabla,Fil)=(C_1^{-1}(E,\theta),Fil).$$ What remains is to
produce a relative Frobenius $\Phi$ from the $\phi$. This is the most technical point of the whole proof. Following Faltings \cite[Ch. II. d)]{Fa1}, it suffices to give for each pair
$(U_i,F_{ U_i})$ an $\sO_{U_{i,1}}$-morphism
$$\Phi_{(U_i,F_{ U_i})}: F_{U_{i,1}}^*Gr_{Fil}H|_{U_{i,1}}\to H|_{U_{i,1}},$$
where $F_{U_{i,1}}$  is the absolute Frobenius of $U_{i,1}$, satisfying the following conditions:
\begin{enumerate}
    \item strong $p$-divisibility, that is, $\Phi_{(U_i,F_{ U_i})}$ is an isomorphism,
    \item horizontal property,
    \item over each $U_{ij,1}:=U_{i,1}\cap U_{j,1}$, $\Phi_{(U_i,F_{ U_i})}$ and $\Phi_{(U_j,F_{ U_j})}$ are related via the Taylor formula. Precisely, if the gluing map for $H_{i}:=H|_{U_{i,1}}$ and $H_j:=H|_{U_{j,1}}$ is $G_{ij}$, then we shall have the following commutative diagram:
     $$   \xymatrix{
          F^*_{U_{i,1}}Gr_{Fil}H_i|_{U_{ij,1}} \ar[d]_{\widetilde{G}_{ij}} \ar[rr]^{\Phi_{(U_i,F_{ U_i})}|_{U_{ij,1}}}& & H_i|_{U_{ij,1}} \ar[d]^{G_{ij}} \\
          F^*_{U_{i,1}}Gr_{Fil}H_j|_{U_{ij,1}} \ar[d]_{\varepsilon_{ij}}  && H_j|_{U_{ij,1}} \ar[d]^{Id} \\
          F^*_{U_{j,1}}Gr_{Fil}H_j|_{U_{ij,1}} \ar[rr]^{\Phi_{(U_j,F_{ U_j})}|_{U_{ij,1}}} && H_j|_{U_{ij,1}},  }
        $$
where  $\widetilde{G}_{ij}$ denotes the obvious map induced by $G_{ij}$, and $\varepsilon_{ij}$ is defined by the Taylor formula which is given by the following expression:
\begin{equation}\label{Taylor_formula}
e\otimes 1\to e\otimes 1+\sum_{|\underline{k}|=1}^w(\theta_{\partial}')^{\underline{k}}(e)\otimes\frac{z^{\underline{k}}}{p^{|\underline{k}|}\underline{k}!},
\end{equation}
where $\theta'$ denotes the Higgs field  of  $Gr_{Fil}(H_j,\nabla)$. Here we take a system of \'{e}tale local coordinates $\{\tilde{t}_1,\cdots,\tilde{t}_d\}$ of $U_{ij,2} $, which induces a system of  \'{e}tale local coordinates  $\{t_1,\cdots,t_d\}$ on $U_{ij,1} $, and $\underline {k}=(k_1,\cdots,k_d)$ is a multi-index, $z^{\underline{k}}=z_1^{k_1}\cdots z_d^{k_d}$ with $$z_k=F^*_{U_{i,2}}(\tilde{t}_k)-F^*_{U_{j,2}}(\tilde{t}_k),$$ and $(\theta')^{\underline{k}}_{\partial}=(\theta'_{\partial_{t_1}})^{k_1}\cdots (\theta'_{\partial_{t_d}})^{k_d}$.
\end{enumerate}
Recall that $H=\{H_i:=F^*_{U_{i,1}}E|_{U_{i,1}},G_{ij}\}_{i\in I}$, where $G_{ij}$ has similar expression  of  $\varepsilon_{ij}$ as \ref{Taylor_formula} (see \cite[The poof of Proposition 1.4]{LSZ}) :
\[
G_{ij}(e\otimes 1) = e\otimes 1  + \sum_{|\underline{k}|=1}^w\theta_{\partial}^{\underline{k}}(e)\otimes\frac{z^{\underline{k}}}{p^{|\underline{k}|}\underline{k}!}.
\]
We define
$$
\Phi_{(U_i,F_{ U_i})}=F_{U_{i,1}}^*\phi: F_{U_{i,1}}^*Gr_{Fil}H|_{U_{i,1}}\to F_{U_{i,1}}^*E|_{U_{i,1}} .
$$
By construction, $\Phi_{(U_i,F_{ U_i})}$ is strongly
$p$-divisible (this is Condition (1)). As $\phi$ is globally defined, we have the following diagram:
\[\begin{CD}
Gr_{Fil}H_i|_{U_{ij,1}}@>\phi_i|_{U_{ij,1}}>>E|_{U_{ij,1}}\\
@VGr(G_{ij})VV@ VVIdV\\
Gr_{Fil}H_j|_{U_{ij,1}}@>\phi_j|_{U_{ij,1}}>>E|_{U_{ij,1}}.\\
\end{CD}
\]
Pulling back the above diagram via $F^*_{U_{i,1}}$ , we get the following diagram:
\[\begin{CD}
F^*_{U_{i,1}}Gr_{Fil}H_i|_{U_{ij,1}}@>F^*_{U_{i,1}}(\phi_i)|_{U_{ij,1}}>>F^*_{U_{i,1}}E|_{U_{ij,1}}\\
@V\widetilde{G_{ij}}VV@ VVIdV\\
F^*_{U_{i,1}}Gr_{Fil}H_j|_{U_{ij,1}}@>F^*_{U_{i,1}}(\phi_j)|_{U_{ij,1}}>>F^*_{U_{i,1}}E|_{U_{ij,1}}.\\
\end{CD}
\]
Then we extend it to the following diagram:
\[
\begin{CD}
F^*_{U_{i,1}}Gr_{Fil}H_i|_{U_{ij,1}}@>F^*_{U_{i,1}}(\phi_i)|_{U_{ij,1}}>>F^*_{U_{i,1}}E|_{U_{ij,1}}\\
@V\widetilde{G_{ij}}VV@ VVIdV\\
F^*_{U_{i,1}}Gr_{Fil}H_j|_{U_{ij,1}}@>F^*_{U_{i,1}}(\phi_j)|_{U_{ij,1}}>>F^*_{U_{i,1}}E|_{U_{ij,1}}\\
@V\varepsilon_{ij} VV @VV  G_{ij} V\\
F^*_{U_{j,1}}Gr_{Fil}H_j|_{U_{ij,1}}@>F^*_{U_{j,1}}(\phi_j)|_{U_{ij,1}}>>F^*_{U_{j,1}}E|_{U_{ij,1}}\\
\end{CD}
\]
As $\phi\circ \theta'= \theta\circ \phi$ , we have for any local section $e$ of $Gr_{Fil}H_j|_{U_{ij,1}}$,
\begin{eqnarray*}
G_{ij}\circ F^*_{U_{i,1}}(\phi_j)(e\otimes 1)&=&\phi_j(e)\otimes 1+\sum_{|\underline{k}|=1}^w\theta_{\partial}^{\underline{k}}(\phi_j(e))\otimes\frac{z^{\underline{k}}}{p^{|\underline{k}|}\underline{k}!}\\
&=&\phi_j(e)\otimes 1+\sum_{|\underline{k}|=1}^w\phi_j((\theta')_{\partial}^{\underline{k}}(e))\otimes\frac{z^{\underline{k}}}{p^{|\underline{k}|}\underline{k}!}\\
&=&F^*_{U_{j,1}}(\phi_j)\circ\varepsilon_{ij}(e\otimes 1).
\end{eqnarray*}
So the lower square of the last diagram is commutative (this is Condition (3)). What remains to show is Condition (2).
\begin{lemma}\label{horizontal property}
Each $\Phi_{(U_i,F_{ U_i})}$ is horizontal with respect
to $\nabla$.
\end{lemma}
\begin{proof}
Put $\tilde H=Gr_{Fil}H$, $\theta'=Gr_{Fil}\nabla$,
$\Phi_i=\Phi_{(U_i,F_{ U_i})}$ and $F_{U_{i,1}}$ the absolute
Frobenius over $U_{i,1}$. Following Faltings \cite[Ch. II.. d)]{Fa1}, it
suffices to show the commutativity of the following diagram
$$
\CD
  F^*_{U_{i,1}}\tilde H|_{U_{i,1}} @>\Phi_i>> H|_{U_{i,1}} \\
  @V F_{ U_i}^*\nabla VV @V \nabla VV  \\
  F^*_{U_{i,1}}\tilde H|_{U_{i,1}}\otimes \Omega_{U_{i,1}}@>\Phi_i\otimes Id>> H|_{U_{i,1}}\otimes
  \Omega_{U_{i,1}},
\endCD
$$
where $F_{ U_i}^*\nabla$ is a connection induced by $\frac{dF_{{U}_i}}{p}(F^*_{U_{i,1}}\theta')$, i.e. the composite of
$$
F^*_{U_{i,1}}\tilde{H}|_{U_{i,1}}\stackrel{F^*_{U_{i,1}}(\theta')}{\longrightarrow}F^*_{U_{i,1}}\tilde{H}|_{U_{i,1}}\otimes
F^*_{U_{i,1}}\Omega_{U_{i,1}} \stackrel{Id\otimes\frac{dF_{ U_i}}{p}}{\longrightarrow}F^*_{U_{i,1}}\tilde{H}|_{U_{i,1}}\otimes
\Omega_{U_{i,1}}.
$$
Thus it is to show the commutativity of the next diagram:
$$
\CD
  F^*_{U_{i,1}}\tilde H|_{U_{i,1}} @>F^*_{U_{i,1}}(\phi)>>F^*_{U_{i,1}}E|_{U_{i,1}} \\
  @V \frac{dF_{{U}_i}}{p}(F^*_{U_{i,1}}(\theta')) VV @V \frac{dF_{ U_i}}{p}(F^*_{U_{i,1}}(\theta)) VV  \\
  F^*_{U_{i,1}}\tilde H|_{U_{i,1}}\otimes \Omega_{U_{i,1}}@>F^*_{U_{i,1}}(\phi)\otimes Id>>F^*_{U_{i,1}}E|_{U_{i,1}}\otimes
  \Omega_{U_{i,1}}.
\endCD
$$
As $\phi$ is a morphism of graded Higgs modules, one has the following
commutative diagram:
\begin{equation*}
\begin{CD}
 \tilde{H}|_{U_ i}  @>\phi >> E|_{U_{i,1}}\\
 @V\theta' VV       @VV\theta V\\
 \tilde{H}|_{U_{i,1}}\otimes\Omega_{U_{i,1}}@>\phi\otimes Id>>  E|_{U_{i,1}}\otimes
 \Omega_{U_{i,1}}.
 \end{CD}
\end{equation*}
The pullback via $F^*_{U_{i,1}}$ of the above diagram yields
commutative diagrams:
$$
\xymatrix{
 F^*_{U_{i,1}}\tilde{H}|_{U_{i,1}}\ar[d]_{F^*_{U_{i,1}}(\theta')} \ar[r]^{F^*_{U_{i,1}}(\phi)}
                & F^*_{U_{i,1}}E|_{U_{i,1}} \ar[d]_{F^*_{U_{i,1}}(\theta)} \ar@/^/[ddr]^{\frac{dF_{ U_i}}{p}(F^*_{U_{i,1}}(\theta)) }  \\
  F^*_{U_{i,1}}\tilde{H}|_{U_{i,1}}\otimes F^*_{U_{i,1}}\Omega_{U_{i,1}}\ar[r]^{F^*_{U_{i,1}}(\phi)\otimes Id} \ar@/_/[drr]_{F^*_{U_{i,1}}(\phi)\otimes \frac{dF_{ U_i}}{p}}
                &  F^*_{U_{i,1}}E|_{U_{i,1}}\otimes F^*_{U_{i,1}}\Omega_{U_{i,1}} \ar@{>}[dr]|-{ Id\otimes\frac{dF_{ U_i}}{p}}            \\
                &               & F^*_{U_{i,1}}E|_{U_{i,1}}\otimes \Omega_{U_{i,1}}.              }
$$
Chasing the outside of the above diagram gives the required commutativity.
\end{proof}
Now we can prove Proposition \ref{correspondence in the type (0,1) case}.
\begin{proof}
The equivalence of categories follows by providing natural
isomorphisms of functors:
$$
\mathcal{GR}\circ \mathcal{IC}\cong Id, \quad
\mathcal{IC}\circ \mathcal{GR}\cong Id.
$$
We define first a natural isomorphism $\sA$ from
$\mathcal{IC}\circ \mathcal{GR}$ to $Id$: for
$(H,\nabla,Fil,\Phi)\in MF$, put $$(E,\theta,
Fil,\phi)=\mathcal{GR}(H,\nabla,Fil,\Phi),\quad (H',\nabla',Fil',
\Phi')=\mathcal{IC}(E,\theta,Fil,\phi).$$ Then one verifies
that the map $$\tilde \Phi: (H',\nabla')=C_1^{-1}\circ
Gr_{Fil}(H,\nabla)\cong (H,\nabla)$$ gives an isomorphism from
$(H',\nabla',Fil',\Phi')$ to $(H,\nabla,Fil,\Phi)$ in the category
$MF$. We call it $\sA(H,\nabla,Fil,\Phi)$. It is
straightforward to verify that $\sA$ is indeed a transformation.
Conversely, a natural isomorphism $\mathcal{B}$ from
$\mathcal{GR}\circ\mathcal{IC}$ to $Id$ is given as follows:
for $(E,\theta,Fil,\phi)$, put $$(H,\nabla,Fil,
\Phi)=\mathcal{IC}(E,\theta,Fil,\phi)\quad
(E',\theta',Fil',\phi')=\mathcal{GR}(H,\nabla,Fil,\Phi).$$ Then
$\phi: Gr_{Fil}\circ C_1^{-1}(E,\theta)\cong (E,\theta)$ induces an
isomorphism from $(E',\theta',Fil',\phi')$ to $(E,\theta,Fil,\phi)$
in $HDF$, which we define to be
$\sB(E,\theta,Fil,\phi)$. It is direct to check that $\sB$ is a
natural isomorphism.
\end{proof}
Before moving to the proof of Theorem \ref{correspondence in the
type (0,f) case} in general, we shall introduce an intermediate
category, the category of one-periodic Higgs-de Rham flows with endomorphism structure $\F_{p^f}$: an object is a five
tuple $(E,\theta,Fil,\phi,\iota)$, where $(E,\theta,Fil,\phi)$ is
object in $HDF$ and $\iota: \F_{p^f}\hookrightarrow
\End_{HDF}(E,\theta,Fil,\phi)$ is an embedding of
$\F_p$-algebras. As an immediate consequence of Proposition \ref{correspondence in the type
(0,1) case}, we have
\begin{corollary}\label{Corresponendence between Faltings catgory with endo and Higgs with endo} There is an equivalence of categories between the category of strict $p$-torsion Fontaine modules with endomorphism structure $\F_{p^f}$ and the category of one-periodic Higgs-de Rham flows over $X_1$ with endomorphism structure $\F_{p^f}$.
\end{corollary}
Obviously, Corollary \ref{Corresponendence between Faltings catgory with endo and Higgs with endo} and the next proposition will complete the proof
of Theorem \ref{correspondence in the type (0,f) case}.
\begin{proposition}\label{correspondence from HB_f and HB_(0,f)}
There is an equivalence of categories between the category of one-periodic Higgs-de Rham flows of level $\leq w$ over $X_1$ with endomorphism structure $\F_{p^f}$ and the category $HDF_{w,f}(X_2/W_2)$.
\end{proposition}
Start off with an object $(E,\theta,Fil_0,\cdots,Fil_{f-1},\phi)$ in
$HDF_{w,f}(X_2/W_2)$. Put
$$(G,\eta):=\bigoplus_{i=0}^{f-1}(E_i,\theta_i)$$ with
$(E_0,\theta_0)=(E,\theta)$. As the functor $C_1^{-1}$ is compatible
with direct sum, one has the identification
$$
C_1^{-1}(G,\eta)=\bigoplus_{i=0}^{f-1}C_1^{-1}(E_i,\theta_i).
$$
We equip $C_1^{-1}(G,\eta)$ with the filtration $Fil=\bigoplus_{i=0}^{f-1}Fil_i$ by the above identification. Also
$\phi$ induces a natural isomorphism of graded Higgs modules $$\tilde \phi:
Gr_{Fil}C_1^{-1}(G,\eta)\cong (G,\eta)$$ as follows: as
$$Gr_{Fil}C_1^{-1}(G,\eta)=\bigoplus_{i=0}^{r-1}Gr_{Fil_i}C_1^{-1}(E_i,\theta_i),$$
we require that $\tilde \phi$ maps the factor
$Gr_{Fil_i}C_1^{-1}(E_i,\theta_i)$ identically to the factor
$(E_{i+1},\theta_{i+1})$ for $0\leq i\leq f-2$ (assume $f\geq 2$ to
avoid the trivial case) and the last factor
$Gr_{Fil_{f-1}}(E_{f-1},\theta_{f-1})$ isomorphically to
$(E_0,\theta_0)$ via $\phi$. Thus the constructed quadruple
$(G,\eta,Fil,\tilde \phi)$ is a periodic Higgs-de Rham flow of period
one.
\begin{lemma}\label{lemma from 0,f to f}
Notation as above. There is a natural embedding of
$\F_p$-algebras
$$
\iota: \F_{p^r}\to \End_{HDF}(G,\eta,Fil,\tilde
\phi).
$$
Thus the extended tuple $(G,\eta,Fil,\tilde \phi,\iota)$ is a one-periodic Higgs-de Rham flow with endomorphism structure $\F_{p^f}$.
\end{lemma}
\begin{proof}
Choose a primitive element $\xi_1$ in $\F_{p^f}|\F_{p}$ once and for all. To define the embedding $\iota$, it suffices to specify the image $s:=\iota(\xi_1)$,
which is defined as follows: write
$$(G,\eta)=(E_0,\theta_0)\oplus (E_1,\theta_1)\oplus\cdots\oplus (E_{f-1},\theta_{f-1}).$$
Then $s=m_{\xi_1}\oplus m_{\xi_1^p}\oplus\cdots\oplus m_{\xi_1^{p^{f-1}}}$, where
$m_{\xi_1^{p^i}},i=0,\cdots,f-1$ is the multiplication map by $\xi_1^{p^i}$. It
defines an endomorphism of $(G,\eta)$ and preserves $Fil$ on
$C_1^{-1}(G,\eta)$. Write $(Gr_{Fil}\circ C_1^{-1})(s)$ to be the
induced endomorphism of $Gr_{Fil}C_1^{-1}(G,\eta)$. It remains to
verify the commutativity $$\tilde \phi\circ s=(Gr_{Fil}\circ
C_1^{-1})(s)\circ \tilde \phi.$$ In terms of a local basis, it boils
down to the following obvious equality
{\small
$$
\left(
  \begin{array}{cccc}
    0 & 1 &\ldots&0\\
    \vdots&\vdots&\ddots&\vdots\\
    0&0&\ldots&1\\
    \phi & 0 &\ldots&0\\
  \end{array}
\right)\left(
  \begin{array}{cccc}
    \xi_1 & 0 &\ldots&0\\
        0&\xi_1^{p}&\ldots&0\\
    \vdots&\vdots&\ddots&\vdots\\
 0 & 0 &\ldots&\xi_1^{p^{f-1}}\\
  \end{array}
\right)=\left(
  \begin{array}{cccc}
    \xi_1^p & 0 &\ldots&0\\
       0&\xi_1^{p^2}&\ldots&0\\
    \vdots&\vdots&\ddots&\vdots\\
 0 & 0 &\ldots&\xi_1\\
  \end{array}
\right)\left(
  \begin{array}{cccc}
    0 & 1 &\ldots&0\\
    \vdots&\vdots&\ddots&\vdots\\
    0&0&\ldots&1\\
    \phi & 0 &\ldots&0\\
  \end{array}
\right).
$$
}
\end{proof}
Conversely, given a one-periodic Higgs-de Rham flow with endomorphism structure $\F_{p^f}$, say $(G,\eta,Fil,\phi,\iota)$, we can associate it an object in $HDF_{w,f}(X_2/W_2)$ as follows: the endomorphism $\iota(\xi_1)$
decomposes $(G,\eta)$ into eigenspaces:
$$
(G,\eta)=\bigoplus_{i=0}^{f-1}(G_i,\eta_i),
$$
where $(G_i,\eta_i)$ is the eigenspace to the eigenvalue
$\xi_1^{p^i}$. The isomorphism $C_1^{-1}(\iota(\xi_1))$ induces the
eigen-decomposition of the de Rham module as well:
$$
(C_1^{-1}(G,\eta),Fil)=\bigoplus_{i=0}^{f-1}(C_1^{-1}
(G_i,\eta_i),Fil_i).
$$
Under the decomposition, the isomorphism $\phi:
Gr_{Fil}C_1^{-1}(G,\eta)\cong (G,\eta)$ decomposes into
$\oplus_{i=0}^{f-1}\phi_i$ such that for $i\leq f-2$,
$$
\phi_i: Gr_{Fil_i}C_1^{-1}(G_i,\eta_i)\cong (G_{i+1},\eta_{i+1}),
$$
and $\phi_{f-1}: Gr_{Fil_{f-1}}C_1^{-1}(G_{f-1},\eta_{f-1})\cong (G_{0},\eta_{0})$.
Set $(E,\theta)=(G_0,\eta_0)$.
\begin{lemma}\label{lemma from f to 0,f}
Let $(E_0,\theta_0)=(E,\theta)$. Then the filtrations $\{Fil_i\}$s and isomorphisms of graded Higgs modules
$\{\phi_i\}$s induce inductively the filtration $\widetilde{Fil}_i$
on $C_1^{-1}(E_i,\theta_i), i=0,\cdots,f-1$ and the isomorphism of
graded Higgs modules
$$
\tilde \phi: Gr_{\widetilde{Fil}_{f-1}}C_1^{-1}(E_{f-1},\theta_{f-1})\cong
(E,\theta).
$$
Thus the extended tuple
$(E,\theta,\widetilde{Fil}_0,\cdots,\widetilde{Fil}_{f-1},\tilde
\phi)$ is an object in $HDF_{w,f}(X_2/W_2)$.
\end{lemma}
\begin{proof}
The filtration $\widetilde{Fil}_{0}$ on
$C_1^{-1}(E_0,\theta_0)$ is just $Fil_0$. Set $$(E_1,\theta_1)=Gr_{Fil_0}C_1^{-1}(E_0,\theta_0).$$ Via the isomorphism
$$C_1^{-1}(\phi_0):C_1^{-1}Gr_{Fil_0}C_1^{-1}(G_0,\eta_0)\cong
C_1^{-1}(G_1,\eta_1),$$ we obtain the filtration
$\widetilde{Fil}_{1}$ on $C_1^{-1}(E_1,\theta_1)$ from the $Fil_1$ on $C_1^{-1}(G_1,\eta_1)$ by pull-back. By construction, one has the isomorphism
$$
GrC_1^{-1}(\phi_0): Gr_{\widetilde{Fil}_1}C_1^{-1}(E_1,\theta_1)\cong Gr_{Fil_1}C_1^{-1}(G_1,\eta_1).
$$
Repeating the same procedure for $(E_2,\theta_2)$ and so on, we shall inductively obtain the filtration $\widetilde{Fil}_i$ on $C_1^{-1}(E_i,\theta_i)$ for $i=1,\cdots,f-1$. Finally we define
$$
\tilde \phi: Gr_{\widetilde{Fil}_{f-1}}C_1^{-1}(E_{f-1},\theta_{f-1})= (Gr_{\widetilde{Fil}_{f-1}}C_1^{-1})\circ\cdots\circ  (Gr_{\widetilde{Fil}_{0}}C_1^{-1})(E,\theta)\to (E,\theta)
$$
to be the composite $(GrC_1^{-1})^{f-1}(\phi_0)\circ\cdots \circ(GrC_1^{-1})(\phi_{f-2})\circ \phi_{f-1}$.
\end{proof}
We come to the proof of Proposition \ref{correspondence from HB_f
and HB_(0,f)}.
\begin{proof}
For $f=1$, there is nothing to prove. Suppose $f\geq 2$ in the following. Note that Lemma \ref{lemma from 0,f to f} gives a functor
$\sE$ from $HDF_{w,f}(X_2/W_2)$ to the category of one-periodic Higgs-de Rham flows with endomorphism structure $\F_{p^f}$, while Lemma
\ref{lemma from f to 0,f} gives a functor $\sD$ in the opposite direction.
We show that they give an equivalence of categories. It is direct to
see that
$$
\sD\circ\sE=Id.
$$
So it remains to give a natural isomorphism $\tau$ between
$\mathcal{E}\circ\mathcal{D}$ and $Id$. For $(E,\theta,Fil,\phi,\iota)$, let
$$\mathcal{D}(E,\theta,Fil,\phi,\iota)=(G,\eta,Fil_0,
\cdots,Fil_{f-1},\tilde \phi),$$ and $$\mathcal{E}(G,\eta,Fil_0,\cdots, Fil_{f-1},\tilde
\phi)=(E',\theta',Fil',\phi',\iota').$$
By the construction, we get that $(E',\theta')$ is equal to
$$
(G,\eta)\oplus Gr_{Fil_0}C_1^{-1}(G,\eta)\oplus \cdots\oplus (Gr_{Fil_{f-2}}C_1^{-1})\circ\cdots \circ(Gr_{Fil_0}C_1^{-1})(G,\eta).
$$
Let $(E,\theta)=(E_0,\theta_0)\oplus (E_1,\theta_1)\oplus \cdots\oplus (E_{f-1},\theta_{f-1})$ be the eigen-decomposition of $(E,\theta)$ under $\iota(\xi_1)$. For $1\leq i\leq f-1$, there is a natural isomorphism $\phi_{i-1}\circ (GrC^{-1})\phi_{i-2}\circ\cdots\circ (GrC^{-1})^{i-1}\phi_0$ of graded Higgs modules between the factors:
$$
(Gr_{Fil_{i-1}}C_1^{-1})\circ(Gr_{Fil_{i-2}}C_1^{-1})\cdots \circ(Gr_{Fil_0}C_1^{-1})(G,\eta)\cong (E_i,\theta_i).
$$
Thus $Id\oplus \bigoplus_{i=1}^{f-1}\phi_{i-1}\circ (GrC^{-1})\phi_{i-2}\circ\cdots\circ (GrC^{-1})^{i-1}\phi_0$ provides us an isomorphism of graded Higgs modules from $(E',\theta')$ to $(E,\theta)$. It is easy to check that it yields an isomorphism of
$\tau(E,\theta,Fil,\phi,\iota)$ in the latter category. The functorial property of $\tau$ is easily verified.
\end{proof}
This completes the Higgs correspondence in positive characteristic. In the following we deduce from it some direct consequences.\\

{\itshape Crystalline $\F_{p^f}$-representations.}\\
 
Let $X/W$ be a smooth proper scheme over $W$. An $\F_{p^f}$-representation of $\pi_1(X_{K})$ is said to be crystalline if it is crystalline as an $\F_p$-representation by restriction of scalar. In other words, a crystalline $\F_{p^f}$-representation is a crystalline $\F_p$-representation $\V$ together with an embedding of $\F_p$-algebras $\F_{p^f}\hookrightarrow \End_{\pi_1(X_{K})}(\V)$. Similarly, one has the notion of crystalline $W_n(\F_{p^f})$-representation for $n\in \N\cup \{\infty\}$. The following corollary is immediate from Theorem \ref{Faltings thm} and Theorem \ref{correspondence in the type (0,f) case}.
\begin{corollary}\label{correspondence from crystalline represenations and
HB_(0,f)}
Let $X/W$ be a smooth proper scheme.  Assume $w\leq p-2$. There is an equivalence of categories between the category of crystalline $\F_{p^f}$-representations of $\pi_1(X_{K})$ with Hodge-Tate weight $\leq w$ and the category of $f$-periodic Higgs-de Rham flows of level $\leq w$ over $X_1$.
\end{corollary}
For an object $(E,\theta,Fil_0,\cdots,Fil_{f-1},\phi)\in HDF_{w,f}(X_2/W_2)$,  we define its shift and lengthening as follows: note for $(E_f,\theta_f)=Gr_{Fil_{f-1}}(H_{f-1},\nabla_{f-1})$,
$C_1^{-1}(\phi)$ induces the pullback filtration
$(C_1^{-1}(\phi))^*Fil_0$ on $C^{-1}_1(E_f,\theta_f)$ and an
isomorphism of graded Higgs modules $GrC_1^{-1}(\phi)$ on the gradings. Then it
is easy to check that the tuple
$$(E_1,\theta_1,Fil_1,\cdots,Fil_{f-1},C_1^{-1}(\phi)^*Fil_0,GrC_1^{-1}(\phi))$$
is an object in $HDF_{w,f}(X_2/W_2)$, which we call the \emph{shift} of $(E,\theta,Fil_0,\cdots,Fil_{f-1},\phi)$. For any
multiple $lf, l\geq 1$, we can lengthen
$(E,\theta,Fil_0,\cdots,Fil_{f-1},\phi)$ to an object of
$HDF_{w,lf}(X_2/W_2)$: similar to above, we can inductively define the
induced filtration on $(H_j,\nabla_j), f\leq j\leq lf-1$ from
$Fil_i$s via $\phi$. One has the induced isomorphism of graded Higgs
modules
$$(GrC_1^{-1})^{l'f}(\phi):
(E_{(l'+1)f},\theta_{(l'+1)f})\cong (E_{l'f},\theta_{l'f}), \quad 0\leq
l'\leq l-1.
$$
The isomorphism $\phi_{l}: (E_{lf},\theta_{lf})\cong
(E_0,\theta_0)$ is defined to be the composite of them. The obtained
object $(E,\theta,Fil_0,\cdots,Fil_{lf-1},\phi_l)$ is called the
$l-1$-th \emph{lengthening} of
$(E,\theta,Fil_0,\cdots,Fil_{f-1},\phi)$. The following result will be obvious from the proof of Theorem \ref{correspondence in the type (0,f) case}.
\begin{corollary}\label{operations on Higgs-de Rham sequences}
Notation as in Corollary \ref{correspondence from crystalline represenations and
HB_(0,f)}. Let $\rho$ be the corresponding crystalline $\F_{p^f}$-representation to $(E,\theta,Fil_0,\cdots,Fil_{f-1},\phi)$. Then the followings are true:
\begin{itemize}
    \item [(i)] The shift of
    $(E,\theta,Fil_0,\cdots,Fil_{f-1},\phi)$ corresponds to
   $\rho^{\sigma}=\rho\otimes_{\F_{p^f},\sigma}\F_{p^f}$, the $\sigma$-conjugation
   of $\rho$. Here $\sigma\in \Gal(\F_{p^f}|\F_p)$ is the Frobenius
   element.
    \item [(ii)] For $l\in \N$, the $l-1$-th lengthening of
    $(E,\theta,Fil_0,\cdots,Fil_{f-1},\phi)$ corresponds to the extension of scalar
    $\rho\otimes_{\F_{p^f}}\F_{p^{lf}}$.
\end{itemize}
\end{corollary}

{\itshape Locally freeness of preperiodic Higgs modules.}\\

In the following, we explain that a preperiodic Higgs module is locally free, a posteriori property.
\begin{proposition}\label{locally freeness of periodic bundles}
Periodic Higgs module is locally free.
\end{proposition}
\begin{proof}
Let $(E,\theta)$ be a periodic Higgs module. Then a periodic Higgs-de Rham
flow with its leading term $(E,\theta)$ gives an object in the category
$HDF_{w,f}(X_2/W_2)$ for some $f$. Let
$(H,\nabla,Fil,\Phi,\iota)$ be the corresponding object in
$MF^{\nabla}_{[0,w],f}(X_2/W_2)$ after Theorem \ref{correspondence in the type (0,f) case}. The proof of \cite[Theorem 2.1]{Fa1} (cf. page 32
loc. cit.) asserts that $Fil$ is a filtration of locally free
subsheaves of $H$ and locally split, which implies that $Gr_{Fil}H$ is locally free.
It follows immediately that $(E,\theta)$ is also locally free.
\end{proof}
A. Langer \cite[Proposition 1.1 \S5.3]{Langer2} has obtained the following enhancement of the previous result (notice however that the assumption that $\rk \ E\leq p$ in the cited statement on the Higgs module is in our case unnecessary due to a slightly different definition of Higgs-de Rham flow).
\begin{corollary}[Langer]\label{locally free}
Preperiodic Higgs module is locally free.
\end{corollary}
\begin{proof}
It follows from Corollary \ref{locally freeness of periodic bundles} and Lemma 3 \cite{Langer2}.
\end{proof}
\begin{corollary}\label{Hodge filtration property of filtrations in preperioidc flow}
The Griffiths transverse filtrations in a preperiodic Higgs-de Rham flow are filtrations by locally free subsheaves and locally split.
\end{corollary}
\begin{proof}
The periodic case has been explained in the proof of Proposition \ref{locally freeness of periodic bundles}. But it also follows from Corollary \ref{locally free}, by induction on the level of filtrations. So does the preperiodic case.  
\end{proof}

\section{Inverse Cartier transform over a truncated Witt ring}
The inverse Cartier transform of Ogus-Vologodsky \cite{OV} has played a pivotal role in the notion of a (periodic) Higgs-de Rham flow in characteristic $p$. In order to obtain the analogous notion of (periodic) Higgs-de Rham flow over a truncated Witt ring, we need to construct a lifting of the inverse Cartier transform. In this section, $X_n$ is a smooth scheme over $W_n$ and $X_{n+1}$ is a $W_{n+1}$-lifting of $X_n$.  \\

An anonymous referee has kindly pointed to us that the work of A. Shiho \cite{Shiho} is related to our construction below. Recall, for each $n\in \N$, $S_n=\Spec\ W_n$ and $F_{S_n}: S_n\to S_n$ the Frobenius automorphism. In \cite{Shiho} Shiho constructs a functor from the category of quasi-nilpotent Higgs modules on $X_n^{(n)}$ to the category of quasi-nilpotent flat modules on $X_n$, where $X_n$ is a smooth scheme over $W_n$ and $X_n^{(n)}=X_n\times_{F_{S_n}^n}S_{n}$, under the assumption that $X_n^{(m)}=X_n\times_{F_{S_n}^m}S_{n}, 0\leq m\leq n$ admits a smooth lifting $X^{(m)}_{n+1}$ to $S_{n+1}$ and the Frobenius liftings $F_{n+1}^{m}: X^{(m-1)}_{n+1}\to X^{(m)}_{n+1}$ over $S_{n+1}$ exist. The functor is a nice $p$-adic reincarnation of the notion of $\lambda$-connection in complex differential geometry. However, the assumption on the Frobenius lifting is very restrictive for a projective $W_n$-scheme, which is however the basic assumption to formulate the semistability for Higgs modules. For example, a smooth projective curve over $W_2$ admits no Frobenius lifting once its genus is greater than one. Our construction was inspired by the fact that Ogus-Vologodsky's construction extends the theory of strict $p$-torsion Fontaine modules (see \S4 \cite{OV}, see also \cite{SXZ}, \cite{LSZ}). We have worked out a generalization for sub strict $p^n$-torsion Fontaine modules in \cite{SZPeriodic} and the current construction is then a further generalization (without assuming the existence of an ambient strict $p^n$-torsion Fontaine module). \\

Let us put $X_n'=X_n\times_{F_{S_n}}S_n$. Then $X'_{n+1}$ is a smooth lifting of $X_n'$ over $W_{n+1}$. The mod $p^{n-1}$ reduction $X_n\otimes \Z/p^{n-1}\Z$ of $X_n$ is denoted by $X_{n-1}$. Similarly for $X_{n-1}'$. Let us introduce a category $\sH(X_n')$ of Higgs modules over $X_n'$ as follows: an object is given by a tuple $$(E,\theta, \bar{H}, \bar{\nabla},\overline{Fil},\bar{\psi} ),$$ where $(E,\theta)$ is a graded Higgs module over $X_n'$ of exponent $\leq p-2$, $(\bar{H}, \bar{\nabla}, \overline{Fil})$ a de Rham module over $X'_{n-1}$ with the level of Hodge filtration $\leq p-2$ and
$$
\bar{\psi}: Gr_{\overline{Fil}}(\bar{H}, \bar{\nabla})\cong (\bar{E},\bar{\theta}):=(E,\theta)\otimes \Z/p^{n-1}\Z
$$
an isomorphism of graded Higgs modules over $X_{n-1}'$. The morphism in the category is defined in the obvious way. For $n=1$, the above tuple is reduced to a nilpotent graded Higgs module over $X_1'$ of exponent $\leq p-2$. So $\sH(X_1')$ is a full subcategory of $HIG_{p-1}(X_1')$.

\begin{theorem}\label{lifting of inverse cartier} 
Notation as above. There exists a functor $\sC_{n}^{-1}$ from the category $\sH(X_n')$ to the category $MIC(X_n)$ of flat modules over $X_n$ such that $\sC_{n}^{-1}$ lifts $\sC_{n-1}^{-1}$ and such that $\sC_{1}^{-1}$ agrees with the inverse Cartier transform $C^{-1}_{\sX/\sS}$ of Ogus-Vologodsky \cite{OV} with $(\sX,\sS)=(X_1/k,X_2'/W_2)$.
\end{theorem}

We shall also introduce an intermediate category $\widetilde{MIC}(X_n')$ which we call the category of twisted flat modules over $X_n'$. The construction of the functor $\sC_n^{-1}$ consists of constructing the first functor $\sT_n: \sH(X_n')\to \widetilde{MIC}(X_n')$ and the second functor $\sF_n: \widetilde{MIC}(X_n')\to MIC(X_n)$. The category $\widetilde{MIC}(X_n')$ is closely related to the category of quasi-nilpotent $\sO_{X_n'}$-modules with integrable $p$-connections of Shiho \cite{Shiho} which we shall explain later. The motivation to introduce this new category is mainly because of the necessity to make sense of $p$-powers in the denominators appearing in the Taylor formula (\ref{gluing function}).\\

Let $X$ be a smooth scheme over $S_n$. First recall that a \emph{Lie algebroid} on $X$ is a locally free $\sO_{X}$-module $\sA$ equipped with a skew-symmetric $\sO_{X}$-bilinear pairing
$$
[\cdot,\cdot]_{\sA}:\sA\times \sA\to \sA
$$
satisfying the Jacobi identity and an action of $\sA$ on $\sO_{X}$ by derivations (so-called \emph{anchor map}) which is given by an $\sO_{X}$-linear morphism of Lie algebras
$$
\alpha: \sA\to T_{X}
$$
subject to the compatibility condition (Leibniz rule)
$$
[x,fy]_{\sA}=\alpha(x)(f)y+f[x,y]_{\sA}.
$$
Let us consider the following Lie algebroid  $(T_{X},\alpha,\{\cdot,\cdot\}:=p[\cdot,\cdot])$, with the anchor $\alpha=p\cdot Id: T_{X}\to T_{X}$ and $\{D_1,D_2\}=p[D_1,D_2] $ for any local sections $D_i, i=1,2$ of $T_{X}$, where $[\cdot,\cdot]$ is the usual Lie bracket for the tangent sheaf. Let $\mathcal{D}^{(-1)}_{X}$ denote the sheaf of enveloping algebras of the Lie algebroid $(T_{X},\alpha,\{\cdot,\cdot\})$. Thus $ \mathcal{D}^{(-1)}_{X}$ is generated by the algebra of functions $\mathcal{O}_{X}$ and the $\mathcal{O}_{X}$-module of derivations $T_{X}$, subject to the module and commutator relations
$$f\cdot D =fD, \quad D \cdot f-f\cdot D=pD(f), \quad D \in T_{X}, \quad f\in \mathcal{O}_{X},$$ and the Lie algebroid relation
$$
D_1\cdot D_2-D_2\cdot D_1=\{D_1,D_2\}, \quad D_1,D_2\in T_{X}.
$$
Next we introduce a sheaf of twisted differential operators on $X$ as follows.
\begin{definition}
Notation as above. Let $U\subset X$ be an open affine subset over $S_n$. Set $\widetilde{\mathcal{D}}_{X}(U)$ to be the algebra generated over $ \mathcal{D}^{(-1)}_{X}(U)$ by symbols
$$\Bigg\{ \gamma_m(D_1,\cdots ,D_{p-1+m}) \!\!\ \Big| \textrm{ for } m\in \mathbb{N}, \textrm{ and any } D_1,\cdots D_{p-1+m}\in T_{X}(U) \Bigg\},$$ subject to the following six relations:
\begin{enumerate}
\item $ p^m\cdot \gamma_m(D_1,\cdots ,D_{p-1+m})= D_1\cdots D_{p-1+m}$;
\item $\gamma_m(D_1,\cdots ,\alpha D_i+\alpha'  D_i',\cdots D_{p-1+m})= \alpha\cdot\gamma_m(D_1,\cdots, D_i,\cdots ,D_{p-1+m})$\\ $+\alpha'\cdot\gamma_m(D_1,\cdots, D_i',\cdots,D_{p-1+m})$,  $\alpha, \alpha'\in W_n$, $D_i,D_i'\in  T_{X}(U)$;
\item for $m\geq 1$,   $ \gamma_m(D_1,\cdots,  D_{p-1+m})\cdot f=  f\cdot \gamma_m(D_1,\cdots,  D_{p-1+m})$\\$+ \sum_{i=1}^{p-2+m} \gamma_{m-1}(D_1,\cdots, D_{i-1}, D_i(f)D_{i+1},\cdots , D_{p-1+m})+$\\$
    \gamma_{m-1}(D_1,\cdots , D_{p-2+m})\cdot D_{p-1+m}(f)$;
  \item   $\gamma_{m}(D_1,\cdots,D_i,D_{i+1},\cdots, D_{p-1+m})=\gamma_{m}(D_1,\cdots,D_{i+1},D_{i},\cdots, D_{p-1+m})$\\$+ \gamma_{m-1}(D_1,\cdots ,D_{i-1},[D_i,D_{i+1}],D_{i+2},\cdots, D_{p-1+m}) $;
\item $\gamma_{m_1}(D_1,\cdots,D_{p-1+m_1})\cdot \gamma_{m_2}(D_{p+m_1},\cdots, D_{2p-2+m_1+m_2})$\\$=p^{p-1}\gamma_{p-1+m_1+m_2}(D_1,\cdots,D_{2p-2+m_1+m_2})$;
  \item $\gamma_{m}(D_1,\cdots,D_{p-1+m})\cdot D_{p+m}=D_1\cdot\gamma_{m}(D_2,\cdots,D_{p+m})$\\
        $= p\cdot \gamma_{m+1}(D_1,\cdots,D_{p-1+m}, D_{p+m}) $.
\end{enumerate}
The sheaf of twisted differential operators $\widetilde{\mathcal{D}}_{X}$ is the associated sheaf to the presheaf $U\mapsto \widetilde{\mathcal{D}}_{X}(U)$.
\end{definition}
The above definition is to legitimate the $p$-powers appearing in the denominators of differential operators. In particular, one may regard $\gamma_m(D_1,\cdots ,D_{p-1+m})$ as a symbol for $\frac{D_1\cdots D_{p-1+m}}{p^m}$. Since $\widetilde{\mathcal{D}}_{X}$ contains $\mathcal{O}_{X}$ as a subsheaf of algebras, it has a natural left $\mathcal{O}_{X}$-module structure. Thus it contains $T_{X}$ as a left $\mathcal{O}_{X}$-submodule.
\begin{definition}\label{twisted connection}
A twisted connection on an $\sO_{X}$-module $H$ is a $W_n$-morphism between sheaves of $W_n$-algebras
$$
\tilde{\nabla}: \widetilde{\mathcal{D}}_{X}\to \End_{W_n}(H)
$$
extending the structural morphism $\mathcal{O}_{X}\to \End_{W_n}(H)$.
\end{definition}
A coherent $\sO_{X}$-module equipped with a twisted connection is said to a \emph{twisted flat module}. Let $\widetilde{MIC}(X)$ be the category of
twisted flat modules over $X$. An explanation of the relation with the notion of an $\sO_{X}$-module with integrable $p$-connection is in order. We cite the following definition from Shiho \cite{Shiho}:
\begin{definition}[{\cite[Definition 1.1-1.2,1.5]{Shiho}}]\label{def of p-connection}
Notation as above. A $p$-connection $\nabla$ on an $\sO_{X}$-module $H$ is an $W_n$-linear map $\nabla:H\to H\otimes\Omega_{X}$, such that
$$
\nabla(fh)= p\cdot df\otimes h + f\nabla(h), \quad f\in \mathcal{O}_{X},h\in H.
$$
It is said to be integrable if $\nabla_1\circ \nabla=0$, where $$\nabla_1: H\otimes \Omega_{X}\to H\otimes \Omega^2_X$$ is the $W_n$-linear map defined by via the formula $$\nabla_1(h\otimes \omega)=\nabla(h)\wedge \omega+ph\otimes d\omega.$$ Moreover, an integrable $p$-connection $(H,\nabla)$ is quasi-nilpotent if locally
$\theta^{a}=0$ once $|a|\geq N$ for some natural number $N$ where $\theta$ is the connection one-form $\nabla(e)=\sum_i\theta_i(e)dt_i$
with respect to a set of \'{e}tale local coordinates $\{t_1,\cdots,t_d\}$ of $X$ and $\theta^a=\prod_i\theta_i^{a_i}$ for a multi-index $a=(a_1,\cdots,a_d)\in \N^d$ and $|a|=\sum_ia_i$.
\end{definition}
In the above definition, the natural number $N$ depends on the local expression of a local section $e$. However, in Lemma 1.6 \cite{Shiho}, an integrable $p$-connection $\nabla$ on $H$ being quasi-nilpotent is shown to be independent of choices of local coordinates. Let us denote the category of $\sO_X$-modules with $p$-connections by $MC^{(-1)}(X)$ and the integrable ones by $MIC^{(-1)}(X)$. To our purpose, we also
need to introduce a level structure on the quasi-nilpotency of an integrable $p$-connection.
\begin{definition}\label{quasi-nilpotence}
Notation as above. Let $(H,\nabla)$ be a quasi-nilpotent $\sO_X$-module with an integrable $p$-connection. It is said to be quasi-nilpotent of level $\leq m$ if
$$
\nabla_{D_1}\circ\cdots\circ \nabla_{D_{m+1}}=0, \qquad \textrm{  for any   } D_1,\cdots D_{m+1}\in T_{X}.
$$
Denote the category of quasi-nilpotent $\sO_X$-modules with an integrable $p$-connection of level $\leq m$ by $MIC^{(-1)}_{m}(X)$.
\end{definition}
We shall point out that it is rather subtle to read the level structure in terms of  local coordinates. Now there is a functor from $\widetilde{MIC}(X)$ to $MIC^{(-1)}_{p-2+n}(X)$: given a twisted connection $\tilde \nabla$ on $H$, one considers only the restriction to $T_{X}\subset \tilde{\sD}_{X}$ which yields a $W_n$-linear morphism
$$
\nabla:=\tilde \nabla|_{T_{X}}: H\to H\otimes \Omega_X.
$$
It follows directly from the definition that this is indeed an integrable $p$-connection and quasi-nilpotent of level $\leq p-2+n$. Notice that the integrability amounts to the relation
$$
\nabla(D_1)\circ \nabla(D_2)-\nabla(D_2)\circ \nabla(D_1)= \nabla(\{D_1,D_2\}), \quad  D_1,D_2\in  T_{X}.
$$
For $m\geq p$, we may write
$$
\frac{\nabla(D_1)\circ\cdots \circ \nabla(D_m)}{m!}=
\frac{p^{m+1-p}}{m!}\cdot\widetilde{\nabla}(\gamma_{m+1-p}(D_1,\cdots,D_m)),
$$
and as the factor $\frac{p^{m+1-p}}{m!}$ converges to zero $p$-adically as $m$ goes to the infinity, $\frac{\nabla(D_1)\circ\cdots \circ \nabla(D_m)}{m!}\to 0$ when $m\to\infty$. This convergence property is crucial in the construction of our second functor $\sF_n$. As a side remark, we do not know whether the above functor from $\widetilde{MIC}(X)$ to $MIC^{(-1)}_{p-2+n}(X)$ is essentially surjective (but it is not faithful). Later, we shall come back to this point again.\\

Now we can proceed to the construction of our first functor $\sT_n: \sH(X_n')\to \widetilde{MIC}(X_n')$. We shall present two approaches in the following. The first approach, which is our original approach, is a method of ``local lifting-global gluing" which may be more familiar to a reader whose background is in complex algebraic geometry. The second approach is based on a beautiful construction suggested by the referee who kindly allows us to reproduce his/her idea here. This direct approach is much more transparent. \\

{\itshape First approach.}\\

Let $X$ be a smooth scheme over $S_n$. Let $\bar X:=X\otimes \Z/p^{n-1}\Z$ be its reduction mod $p^{n-1}$. Since our arguments rely heavily on local calculations using a basis of local sections of an $\sO_X$-module, we will restrict ourselves in this approach to the full subcategory $\sH_{lf}(X)\subset \sH(X)$ consisting of  \emph{locally free} objects. However, this restriction is caused mainly for simplicity in the local arguments.  For a general coherent object, we may use a set of minimal generators with possible relations for a coherent object locally.  Also, we need to introduce some other categories which will be used only in this approach. The category $\sH^{ni}(X)$ (resp. $\sH^{ni}_{lf}(X)$) is a variant of $\sH(X)$ (resp. $\sH_{lf}(X)$): its object is also a tuple $(E,\theta, \bar{H}, \bar{\nabla},\overline{Fil},\bar{\psi} )$; but the integrabilities on $\theta$ and $\bar{\nabla}$ are not required. For brevity, an object in $\sH(X)$ or $\mathcal{H}^{ni}(X)$ is written as $(E,\bar{H})$. Second, let $MCF_{p-2}(X)$ be a category of filtered $\sO_X$-modules equipped with (not necessarily integrable) connections: its object is a triple $(H,\nabla,Fil)$, where $H$ is a locally free $\sO_{X}$-module, $\nabla$ is a $W_n$-linear connection on $H$, and $Fil$ a finite exhaustive decreasing filtration of locally free $\sO_X$-submodules on $H$ of level $\leq p-2$ which is locally split and satisfies Griffiths' transversality. There is a diagram of these categories connected by natural functors:
$$
\begin{tikzcd}
\,&MCF_{p-2}(X)\arrow{d}{\,}[swap]{R_n}\arrow{rd}{G_n}\\
\sH(X)\arrow[hook]{r}&\sH^{ni}(X)&MC^{(-1)}(X)
\arrow[hookleftarrow]{r}&MIC^{(-1)}_{p-3+n}(X)
\end{tikzcd}
$$

The functor $R_n$ is the obvious one: to $(H,\nabla, Fil)$, one associates the graded Higgs module $(E,\theta)=Gr_{Fil}(H,\nabla)$, which is locally free by the assumption on the filtration,  and also $(\bar{H},\bar{\nabla},\overline{Fil})$,  its mod $p^{n-1}$ reduction. One notices that there is a natural isomorphism of graded Higgs modules
$$
\bar{\psi}: Gr_{\overline{Fil}}(\bar{H}, \bar{\nabla})\cong Gr_{Fil}(H,\nabla)\otimes \Z/p^{n-1}\Z.
$$
The functor $G_n$ is a variant of a construction due to Faltings \cite[Ch.II]{Fa1}, which originates from \cite{FL} in the case $X=S_n$. Given an object $(H,\nabla, Fil)\in MCF_{p-2}(X)$, the object $(\tilde H,\tilde \nabla)=G_n(H,\nabla, Fil)$ is defined as follows: $\tilde{H}$ is the cokernel of the first  map of the following exact sequence
\begin{equation}\label{defining tilde H}
\begin{CD}
\oplus_i Fil^i@>[-1]-p\cdot Id>> \oplus_i Fil^i@>\rho>> \tilde{H}\to 0,
 \end{CD}
 \end{equation}
where $[-1]:=\oplus_i(Fil^i\hookrightarrow Fil^{i-1})$,  $Id$ denotes the identity map, and $\rho$ is the natural projection map. Here we have extended the filtration so that $$Fil^i=Fil^0, i\leq -1, \quad Fil^{j}=0, j\geq p.$$ Consider the $W_n$-linear map
 $$
 \nabla':=\oplus_i\nabla|_{Fil^i},\quad \nabla|_{Fil^i}: Fil^i\to Fil^{i-1}\otimes \Omega_{X}.
 $$
The image $([-1]-p\cdot Id)(\oplus_iFil^i)\subset \oplus_i\nabla|_{Fil^i}$ being preserved, $\nabla'$ induces a $W_n$-linear map
$$
\tilde{\nabla}: \tilde{H}\to \tilde{H}\otimes \Omega_{X}.
$$
One checks immediately that $\tilde{\nabla}$ is indeed a $p$-connection. \\

First, we show that the functor $R_n$ is locally essentially surjective.
\begin{lemma}\label{local lifting}
Let $X$ be a smooth affine scheme over $S_n$. Let $(E,\bar H)$ be an object in $\sH_{lf}^{ni}(X)$. Suppose each component of $E=\oplus_kE_k$ is a free $\sO_X$-module. Then there exists an object $(H,\nabla,Fil)\in MCF_{p-2}(X)$ such that $R_n(H,\nabla,Fil)=(E,\bar H)$.
\end{lemma}
\begin{proof}
Write $E=\oplus_{k=0}^w E^{w-k,k}, \theta=\oplus_{k=0}^{w-1}\theta_w$, and   for $0\leq k \leq w$ take a set of basis $\{e_k\}$ for the $\sO_X$-module $E^{w-k,k}$. Under the basis, $\theta_k$ is expressed in terms of a matrix of differential one-forms which we write again by $\theta_k$.
Put the bar over $\theta_k$ as well as $\{e_k\}$ to mean their mod $p^{n-1}$-reduction. Then take the basis $\{f'_k\}_{0\leq k\leq w}$ of $Gr_{\overline{Fil}}(\bar{H})$ such that
$$
\bar{\psi}(f'_k)=\bar{e}_k.
$$
Also choose a basis $\{\bar{f}_k\}_{0\leq k\leq w}$ of $\bar{H}$ such that its image in $Gr_{\overline{Fil}}(\bar{H})$ is $\{f'_k\}_{0\leq k\leq w}$.
By Griffiths' transversality, the connection matrix $(\bar{a}_{ij})$ representing the connection $\bar{\nabla}$ under the basis $\{\bar{f}_{k}\}$, i.e., $$\bar{\nabla}(\bar{f}_i)=\sum_j \bar{a}_{ij}\bar{f_j},$$ has the property
$\bar{a}_{ij}=0, \ j>i+1$. Now we take a matrix of differential one-forms $(a_{ij})$ over $X$ as follows: for $i\geq j$, take any lift $a_{ij}$ of $\bar{a}_{ij}$; for $j=i+1$, take $a_{ij}=\theta_i$; for $j>i+1$, take $a_{ij}=0$. Now let $H$ be the free $\sO_X$-module generated by $\{f_{k}\}_{0\leq k\leq w}$ whose mod $p^{n-1}$ reduction are $\{\bar{f}_k\}_{0\leq k\leq w}$ and let $Fil^{w-k}$ be the submodule freely generated by the elements $\{f_j \}_{0\leq j\leq k}$, which gives the filtration $Fil$ on $H$. Then a connection $\nabla$ on $H$ is determined by the formula $\nabla(f_i)=\sum_j a_{ij}f_j$. The so-constructed triple $(H,\nabla,Fil)$ is indeed an object in $MCF_{p-2}(X)$ lifting $(E,\bar H)$, i.e., $R_n(H,\nabla,Fil)=(E,\bar H)$ as required.
\end{proof}
Next we show the following
\begin{lemma}\label{landing in Shiho's category}
If $R_n(H,\nabla,Fil)$ lies in the full subcategory $\sH_{lf}(X)\subset \sH_{lf}^{ni}(X)$,
then $G_n(H,\nabla,Fil)$ lies in the full subcategory $MIC^{(-1)}_{p-3+n}(X)$.
\end{lemma}
\begin{proof}
Set $(\tilde H,\tilde \nabla):=G_n(H,\nabla,Fil)$. Once the integrability of $\tilde \nabla$ is verified, the statement on quasi-nilpotency of level $\leq p-3+n$ follows directly from the definition. As the integrability is a local property, we assume a set of \'{e}tale local coordinate $\{t_1,\cdots, t_d\}$ of $X$ and a filtered basis $\{f_k\}_{0\leq k \leq w}$ of $H$. It suffices to check that for  $1\leq i,j\leq d$, $\tilde{\nabla}(\partial t_i)$ commutates with $\tilde{\nabla}(\partial t_j)$.\\

Let $\tilde{f}_k$ be the image of $f_k$ (elements of $Fil^{k}$) in $\tilde{H}$. Then $\{\tilde{f}_k\}_{0\leq k\leq w}$ forms a basis of $\tilde{H}$. As in Lemma (\ref{local lifting}), we may assume the matrix $A:=(a_{ij})$ of $\nabla$ under the basis $\{f_k\}_{0\leq k\leq w}$ has the property that $a_{ij}=0$ for $j>i+1$. Then the matrix of $\tilde{\nabla}$ under the basis $\{\tilde{f}_k\}_{0\leq k\leq m}$ is $\tilde{A}=(\tilde{a}_{ij})$, with $\tilde{a}_{ij}=p^{i+1-j}a_{ij}$ for $j\leq i+1$, and the rest are zero. Then $\tilde{\nabla}(\partial t_i)\circ \tilde{\nabla}(\partial t_j)$  under the basis $\{\tilde{f}_k\}$ is represented by the matrix
$$ p\partial t_i(\partial t_j\tilde{A})+ (\partial t_j\tilde{A})\cdot(\partial t_i\tilde{A}).
$$
It remains to show the equality for all $i,j$:
$$
 p\partial t_i(\partial t_j\tilde{A})+ (\partial t_j\tilde{A})\cdot(\partial t_i\tilde{A})- p\partial t_j(\partial t_i\tilde{A})- (\partial t_i\tilde{A})\cdot(\partial t_j\tilde{A})=0,
$$
For $0\leq r\leq w-2$ with $t=r+1$ and $s=r+2$, the corresponding entry in the above equality means
$$
 (\partial t_i a_{rt}) (\partial t _j a_{ts})-(\partial t_j a_{rt}) (\partial t _i a_{ts}) =0,
$$
which is actually equivalent to the integrability of the Higgs field $\theta$; For $0\leq r,s\leq w$ with $s\leq r+1$, it means then
$$
p^{r+2-s} (\partial t_i\wedge \partial t_j) (da_{rs})= p^{r+2-s}\sum_{t=0}^w \left[(\partial t_i a_{rt}) (\partial t _j  a_{ts})-(\partial t_j a_{rt}) (\partial t _i a_{ts})\right].
$$
As $r+2-s\geq 1$, the above equation is implied by the following equation, which is equivalent to the integrability of $\bar{\nabla}$ :
$$
(\partial \bar{t}_i\wedge \partial \bar{t}_j) (d\bar{a}_{rs})=  \sum_{t=0}^w \left[(\partial \bar{t}_i \bar{a}_{rt}) (\partial \bar{t}_j  \bar{a}_{ts})-(\partial \bar{t}_j \bar{a}_{rt}) (\partial \bar{t}_i  \bar{a}_{ts})\right]
$$
Here $\{\bar{t}_1,\cdots,\bar{t}_d \}$ means the induced set of \'{e}tale local coordinates on $\bar X$.
\end{proof}
We shall take a step further to land the object $G_n(H,\nabla,Fil)$ in the last lemma in the category of twisted connections.
\begin{lemma}\label{extend to twisted connection}
Notation as above. Then the integrable $p$-connection of $G_n(H,\nabla,Fil)$ can be extended to a twisted connection in a functorial way.
\end{lemma}
\begin{proof}
If $(H,\nabla,Fil)$ is the mod $p^n$ reduction of an object over $W$, then it is clear how to extend $\tilde \nabla$ to a twisted connection: for $\gamma_m(D_1,\cdots,D_{p-1+m})$ with $D_i\in T_X$, one defines $\tilde \nabla(\gamma_m(D_1,\cdots,D_{p-1+m}))$ to be $\frac{\tilde \nabla_{D_1}\circ\cdots\circ\tilde\nabla_{D_{p-1+m}}}{p^m}$, which is defined by first lifting to $W$, then dividing $p^m$ and finally taking the reduction modulo $p^n$. The reason that $\tilde \nabla_{D_1}\circ\cdots \circ\tilde \nabla_{D_{p-1+m}}$ is divisible by $p^m$ is examined as follows: by Griffiths' transversality, it follows that $$\nabla'_{D_{m+1}}\circ\cdots \circ \nabla'_{D_{p-1+m}}: Fil^i\to Fil^{i+1-p}.$$  As $Fil^j=0$ for $j\geq w+1$, and $w\leq p-2$, one sees that the image of  $\nabla'_{D_{m+1}}\circ\cdots \circ \nabla'_{D_{p-1+m}}$ lies in $Fil^i$ with $i\leq -1$. In the quotient $\tilde H$, for $i\leq -1$, the image of each element in $Fil^{i-1}$ is equal to $p$ times the same element in $Fil^{i}$. Thus, since $\nabla'$ always shifts the indices of direct factors of $\oplus_iFil^i$ by minus one, the restriction of $\tilde\nabla$ to the image $\bar{Fil}^{-1}$ (which is isomorphic to $H$) of $Fil^{-1}$ in $\tilde H$ is divisible by $p$. On the other hand, this also gives us a way to extend $\tilde \nabla$ to an twisted connection in the general case. Indeed, we simply define $\tilde{\nabla} (\gamma_m(D_1,\cdots, D_{p-1+m}))$ to be the composite $\frac{\nabla'_{D_1}}{p}\circ\cdots \circ \frac{\nabla'_{D_m}}{p}\circ \nabla'_{D_{m+1}}\circ \cdots \circ \nabla'_{D_{p-1+m}}$, where the symbol $\frac{\nabla_{D_i}'}{p}$ is defined to be the map $\nabla_{D_i}$ on $\bar{Fil}^{-1}\subset \tilde H$. One verifies directly that this defines a twisted connection which extends $\tilde \nabla$.
\end{proof}
We believe the truth of the following
\begin{question}
Can $MIC^{(-1)}_{p-3+n}(X)$ be realized as a full subcategory of  $\widetilde{MIC}(X)$?
\end{question}
We shall also provide the gluing morphisms.
\begin{lemma} \label{unique}
Let $(H_i,\nabla_i,Fil_i)\in MCF_{p-2}(X), i=1,2,3$.  Suppose there are isomorphisms for $1\leq i<j\leq 3$,
$$
 \bar{f}_{ij}: (\bar{H_i},\bar{\nabla}_i,\overline{Fil}_i)\cong (\bar{H_j},\bar{\nabla}_j,\overline{Fil}_j),\quad  f^{G}_{ij}: Gr_{Fil_i}(H_i,\nabla_i)\cong Gr_{Fil_j}(H_j,\nabla_j),
 $$
satisfying
 $$
 Gr(\bar{f}_{ij})= f^G_{ij}\mod p^{n-1}.
 $$
Then there are isomorphisms in the category of $\widetilde{\textrm{MIC}}(X)$  between the objects $(\tilde{H}_i,\tilde{\nabla} _i):=G_n(H_i,\nabla_i,Fil_i)$
 $$
 \tilde{f}_{ij}: (\tilde{H}_i,\tilde{\nabla} _i)\cong (\tilde{H}_j,\tilde{\nabla}_j).
 $$
Moreover, if there are cocycle conditions
$$
 \bar{f}_{13}=\bar{f}_{23}\circ \bar{f}_{12},\quad    f^G_{13}=f^G_{23}\circ f^{G}_{12},
$$
then $\{\tilde{f}_{ij}\}$s also satisfy the cocycle condition $$\tilde{f}_{13}=\tilde{f}_{23}\circ\tilde{f}_{12}.$$
\end{lemma}
\begin{proof}
For an $s\in Fil_i^k \setminus Fil^{k+1}_i$, we denote by $\tilde{s}$ (resp. $\hat s$, $\bar s$) its image under the natural map $Fil^k_i\to \tilde{H}_i$ (resp. $Fil_i^k\to Fil_i^k/Fil_i^{k+1}$, $H_i\to \bar{H}_i$).  Consider sections of $Fil^{k}_j$ whose images under the map $H_j\to \bar{H}_j$ are equal to $\bar{f}_{ij}(\bar{s})$ and under the map $Fil_j^k\to Fil_j^k/Fil_j^{k+1}$ equal to $f_{ij}^G(\hat{s})$ at the same time. As the difference of any two such sections lies in $p^{n-1}Fil_j^{k+1}$, they give rise to a unique section $\tilde{f}_{ij}(\tilde{s})\in\tilde{H}_j$.  So we define $\tilde{f}_{ij}: \tilde H_i\to \tilde H_j$ by sending $\tilde{s}$ to $\tilde{f}_{ij}(\tilde{s})$. It is straightforward to verify the well-definedness of $\tilde{f}_{ij}$ as well as the compatibility with twisted connections. The cocycle condition for $\{\tilde f_{ij}\}$ follows directly from the definition.
\end{proof}
Now we can give the construction of the first functor. One shall notice that it does not require $X$ to be $W_{n+1}$-liftable.
\begin{proposition}\label{gluing}
Let $X$ be a smooth scheme over $S_n$. Then there is a functor $\mathcal{T}_n$ from the category $\mathcal{H}_{lf}(X)$ to the category $\widetilde{MIC}(X)$.
\end{proposition}
\begin{proof}
Let $(E,\bar H)$ be an object in $\sH_{lf}(X)$. Take an open affine covering $\sU=\{U_i\}_{i\in I}$ of $X$ such that $E$ and $\bar H$ are free modules over each $U_i$. By Lemma \ref{local lifting}, we can take an object $(H_i,\nabla_i,Fil_i)\in MCF_{p-2}(U_i)$ for $U_i\in \sU$ such that $R_n(H_i,\nabla_i,Fil_i)=(E,\bar H)|_{U_i}$ (we call $(H_i,\nabla_i,Fil_i)$ a local lifting).  Since over $U_i\cap U_j$, the restrictions of $(H_i,\nabla_i,Fil_i)$ and $(H_j,\nabla_j,Fil_j)$ are both local liftings of $(E,\bar H)|_{U_i\cap U_j}$, there are isomorphisms by Lemma \ref{unique}
$$
\tilde f_{ij}: G_n(H_i,\nabla_i,Fil_i)|_{U_i\cap U_j}\cong G_n(H_j,\nabla_j,Fil_j)|_{U_i\cap U_j}
$$
satisfying the cocycle condition on $U_i\cap U_j\cap U_k$. Thus gluing the local objects $ \{G_n(H_i,\nabla_i,Fil_i)\}_{U_i\in \sU}$ in $\widetilde{MIC}(U_i)$ via the isomorphisms $\{\tilde{f}_{ij}\}$ yields an object in $\widetilde{MIC}(X)$. We denote a so-constructed object by $(\tilde H,\tilde \nabla)_{\sU,\sL_{\sU}}$, where $\sU$ denotes for an open affine covering of $X$ and $\sL_{\sU}$ consists of a local lifting of $(E,\bar H)$ restricted to each $U\in \sU$. We need to show that this object is independent of the choice of local liftings and the choice of open affine coverings up to canonical isomorphism. For two sets of local liftings $\sL^1_{\sU}=\{(H^1_i,\nabla^1_i,Fil_i^{1})\}_{i\in I}$ and  $\sL^2_{\sU}=\{(H^2_i,\nabla^2_i,Fil_i^{2})\}_{i\in I}$, Lemma \ref{unique} provides an isomorphism over each $U_i\in \sU$:
$$
\mu_i: G_n(H^1_i,\nabla^1_i,Fil_i^{*1})\cong  G_n(H^2_i,\nabla^2_i,Fil_i^{*2})
$$
and over $U_i\cap U_j$, the equality $\tilde f^{2}_{ij}\circ \mu_i=\mu_j\circ \tilde f^{1}_{ij}$ holds. So the set of local isomorphisms $\{\mu_i\}_{i\in I}$ glues into a global isomorphism
$$\mu: (\tilde H,\tilde \nabla)_{\sU,\sL_{\sU}}\cong (\tilde H,\tilde \nabla)_{\sU,\sL'_{\sU}}.$$
One can further show that for any three choices of local liftings, the so-obtained isomorphisms satisfy the cocycle relation. So the object $(\tilde H,\tilde \nabla)_{\sU,\sL_{\sU}}$ is independent of the choice of local liftings up to canonical isomorphism and therefore can be denoted by $(\tilde H,\tilde \nabla)_{\sU}$. Next for any two choices $\sU,\sU'$ of open affine coverings of $X$, we find a common refinement $\sU''$ of both. Again, by Lemma \ref{unique}, one constructs an isomorphism
$$
\nu_{\sU,\sU''}: (\tilde H,\tilde \nabla)_{\sU}\cong (\tilde H,\tilde \nabla)_{\sU''}
$$
and similarly $\nu_{\sU',\sU''}$. Then define the isomorphism
$$
\nu_{\sU,\sU'}=\nu_{\sU',\sU''}^{-1}\circ \nu_{\sU,\sU''}: (\tilde H,\tilde \nabla)_{\sU}\cong (\tilde H,\tilde \nabla)_{\sU'}
$$
which satisfies also the cocycle relation for any three choices of open affine coverings (another way to remove this independence is to take the direct limit with respect to the directed set of all open affine coverings with the partial order given by refinement). Thus, we get our functor $\sT_n$ by associating $(\tilde H,\tilde \nabla)$ to $(E,\bar H)$  as above which is defined up to canonical isomorphism.
\end{proof}

{\itshape Second approach.}\\

This approach is due to the referee, who gave us a global construction of the functor from the whole category $\sH(X)$ to $MIC^{(-1)}(X)$ in the weight one case. We shall generalize his/her method in the following.\\

Notation as above. For an object $(E,\theta,\bar{H},\bar{\nabla},\overline{Fil},\bar{\psi})\in \mathcal{H}(X)$, one notices that the isomorphism of graded Higgs modules $\bar{\psi}: Gr_{\overline{Fil}}(\bar{H},\bar{\nabla)}\cong (\bar{E},\bar{\theta})$ allows us to define the following composite morphism:
$$
j: \overline{Fil}^i\to \overline{Fil}^i/\overline{Fil}^{i+1}\stackrel{\bar \psi}{\cong} \bar E^i.
$$
Let $\eta: E^i\to \bar E^i$ be the natural projection by the reduction modulo $p^{n-1}$. Let $\overline{Fil}^i\times_{\bar{E}^i}E^i$ denote the kernel of the morphism of $\sO_X$-modules:
$$
 \begin{CD}
 \overline{Fil}^i\oplus E^i@>-j+\eta>> \bar{E}^i.
 \end{CD}
$$
We are going to associate to $(E,\bar H)$ a twisted flat module $(H^{\sharp},\nabla^{\sharp})$. The $\sO_X$-module $H^{\sharp}$ is defined to be the cokernel of the morphism
$$
\begin{CD}
 \oplus_{i}(\overline{Fil}^{i}\times_{\bar{E}^{i}} E^{i})@>\epsilon>> \oplus_{i}(\overline{Fil}^{i}\times_{\bar{E}^{i}} E^{i}),
 \end{CD}
 $$
where for $x\times y\in \overline{Fil}^{i+1}\times_{\bar{E}^{i+1}} E^{i+1}$,
$$
\epsilon(x\times y)=(x\times0)+ (-px\times-py) \in (\overline{Fil}^{i}\times_{\bar{E}^{i}} E^{i})\oplus (\overline{Fil}^{i+1}\times_{\bar{E}^{i+1}} E^{i+1}).
$$
Let $c: \oplus_{i}(\overline{Fil}^{i}\times_{\bar{E}^{i}} E^{i})\to H^{\sharp}$ be the natural morphism. We define a $W_n$-linear additive map (not a connection) as follows:
\begin{equation}\label{connection map}
\nabla'': \oplus_{i}(\overline{Fil}^{i}\times_{\bar{E}^{i}} E^{i})\to \oplus_{i}(\overline{Fil}^{i-1}\times_{\bar{E}^{i-1}} E^{i-1})\otimes \Omega_{X},
\end{equation}
which takes $x\times y\in  \overline{Fil}^{i}\times_{\bar{E}^{i}} E^{i}$ to $\bar{\nabla}(x)\times \theta(y)\in (\overline{Fil}^{i-1}\times_{\bar{E}^{i-1}} E^{i-1})\otimes\Omega_{X}$. Here we have used the natural identification
$$
(\overline{Fil}^{i-1}\otimes \Omega_{\bar{X}})\times_{(\bar{E}^{i-1}\otimes\Omega_{\bar{X}})} (E^{i-1}\otimes\Omega_{X})\cong(\overline{Fil}^{i-1}\times_{\bar{E}^{i-1}} E^{i-1})\otimes\Omega_{X}.
$$
We shall show that this map descends to a $p$-connection over $H^{\sharp}$.
\begin{lemma}
Let $A_i$ be the image of $\overline{Fil}^{i}\times_{\bar{E}^{i}} E^{i}$ under the map $\epsilon$. Then $$\nabla''(A_i)\subset A_{i-1}\otimes \Omega_{X}$$ and the induced $W_n$-linear additive map $\nabla^\sharp:  H^\sharp\to  H^\sharp\otimes \Omega_{X}$ is an integrable $p$-connection.
\end{lemma}
\begin{proof}
For $x\times y \in \overline{Fil}^{i}\times_{\bar{E}^{i}} E^{i}$,  one computes that
$$\nabla''(\epsilon(x\times y))=  \bar{\nabla}(x)\times0+  \bar{\nabla}(-px)\times \theta(-py)=\epsilon(\bar \nabla (x)\times \theta(y)),$$ which
lies in  $A_{i-1}\otimes \Omega_{X}$. Therefore, $\nabla''$ induces an additive $W_n$-linear map $\nabla^\sharp$ on $H^\sharp$. It suffices to show $\nabla^\sharp$ is a $p$-connection, as the integrability follows directly from the integrabilities of $\bar \nabla$ and $\theta$. For $f\in \sO_X$ and $x\times y\in \overline{Fil}^{i}\times_{\bar{E}^{i}} E^{i}$,  it is reduced to show the equality
$$
\nabla^\sharp(fc(x\times y))= f\nabla^\sharp(c(x\times y))+ pc(x\times y)\otimes df.
$$
Because
\begin{eqnarray*}
\nabla''(f(x\times y))&=&\nabla''(\bar{f}\cdot x\times f\cdot y )=\bar{\nabla}(\bar{f}\cdot x)\times \theta(fy)\\
&=& (d\bar{f}\cdot x+ \bar{f}\bar{\nabla}(x))\times f\theta(y)= d\bar{f}\cdot x\times 0 + \bar{f}\bar{\nabla}(x)\times f\theta(y)\\
&=&  (x\times 0)\otimes df + f\nabla''(x\times y),
\end{eqnarray*}
where $\bar f\in \sO_{\bar X}$ is the mod $p^{n-1}$-reduction of $f$, and because
$c(x\times 0)=c(p(x\times y))$, the required equality follows.
\end{proof}
\begin{lemma}\label{twisted connection in second approach}
Notation as above. The $p$-connection $\nabla^\sharp$  extends to a twisted connection.
\end{lemma}
\begin{proof}
This step is similar to Lemma \ref{extend to twisted connection}. Let $D_1,\cdots, D_{p-1+m}\in T_{X}$ for $m\geq 0$. Note that
$$
\nabla''_{D_{m+1}}\circ\cdots \circ \nabla''_{D_{p-1+m}} :\overline{Fil}^i\times_{\bar{E}^i}E^i\to \overline{Fil}^{i+1-p}\times_{\bar{E}^{i+1-p}}E^{i+1-p},
$$
and as $\overline{Fil}^i=0$, $E^i=0$ for $i\geq w+1$ and $w\leq p-2$, the image of   $\nabla''_{D_{m+1}}\circ\cdots \circ \nabla''_{D_{p-1+m}}$ lies in  $\overline{Fil}^i\times_{\bar{E}^i}E^i$ for $i\leq -1$. As $\overline{Fil}^{i}=\bar{H}$ and $\bar{E}^{i }=E^{i}=0$ for $i\leq -1$, one defines the map
$$
 \frac{\nabla''}{p}:=\bar{\nabla}\times 0: \overline{Fil}^i\times_{\bar{E}^i}E^i\to \overline{Fil}^i\times_{\bar{E}^i}E^i\otimes \Omega_{X}, \qquad i\leq -1.
$$
Then the map
$$
\nabla^\sharp(\gamma_m(D_1,\cdots, D_{p-1+m})):= \frac{\nabla''_{D_1}}{p}\circ\cdots \circ \frac{\nabla''_{D_m}}{p}\circ \nabla''_{D_{m+1}}\circ \cdots \circ \nabla''_{D_{p-1+m}}
$$
gives rise to a twisted connection extending the $p$-connection $\nabla^{\sharp}$.
 \end{proof}

 {\itshape Equivalence.}

 \begin{proposition}
The functor $\sT_n^{\sharp}$, restricted to the subcategory $\sH_{lf}(X)$, is naturally equivalent to the functor $\sT_n$.
 \end{proposition}
\begin{proof}
Given an object $(E,\bar{H})\in\mathcal{H}_{lf}(X)$,  we set
$$
 (\tilde{H},\tilde{\nabla} ):=\mathcal{T}_n(E,\bar{H}), \quad (H^\sharp,\nabla^\sharp):=\mathcal{T}^\sharp_n(E,\bar{H})
$$
defined as above. We are going to exhibit a natural isomorphism $$\lambda: (H^\sharp,\nabla^\sharp)\cong (\tilde{H},\tilde{\nabla}).$$ 
Without loss of generality, we may assume that there is an object $(H ,\nabla , Fil)\in MCF_{p-2}$ such that $R_n(H,\nabla,Fil)=(E,\bar H)$ and $G_n(H,\nabla,Fil)=(\tilde H,\widetilde{\nabla})$. Consider the following commutative diagram:
\begin{equation}\label{diagram}
\xymatrixcolsep{6pc}
\xymatrix{
               &   \oplus_{i}Fil^{i}/pFil^{i}  \ar[ld]_{p^{n-1}} \ar[d]^{p^{n-1}\cdot[-1]} \ar[rd]^{0}       \\
\oplus_{i}Fil^{i} \ar[r]^{[-1]-p\cdot Id}\ar@{->>}[d]_{\eta\times j} \ar[rd]^{\kappa} & \oplus_{i} Fil^i\ar@{->>}[r]^{\rho}\ar@{->>}[d]^{\eta\times j}&   \tilde{H} \\
\oplus_{i}\overline{Fil}^{i}\times_{\bar{E}^{i}}E^{i}\ar[r]_{[-1]\times0-p\cdot Id}&\oplus_{i}\overline{Fil}^{i}\times_{\bar{E}^{i}}E^{i}\ar@{->>}[ru]_{\bar{c}}.}
\end{equation}
The middle row is (\ref {defining tilde H}), the exact sequence defining $\tilde{H}$.  The middle column is also an exact sequence because of locally free assumption, where the map $\eta$ denotes the reduction map by mod $p^{n-1}$ and $j$ the natural projection $Fil^i\to Fil^i/Fil^{i+1}=E^i$. By the commutativity of left upper triangle, we see  that the map $\rho$ factors through the surjective map $\bar{c}$, so that there is an exact sequence
$$
\begin{CD}
\oplus_{i}Fil^{i}@>\kappa>>\oplus_{i}\overline{Fil}^{i}\times_{\bar{E}^{i}}E^{i}@>\bar{c}>>\tilde{H}@>>>0.
\end{CD}
$$
By the commutativity of  the left lower triangle and the surjectivity of the left vertical map $\eta\times j$, the above exact sequence is replaced by the following exact sequence:
$$
\begin{CD}
\oplus_{i}\overline{Fil}^{i}\times_{\bar{E}^{i}}E^{i}@>[-1]\times0-p\cdot Id>>\oplus_{i}\overline{Fil}^{i}\times_{\bar{E}^{i}}E^{i}@>\bar{c}>>\tilde{H}@>>>0.
\end{CD}
$$
Thus there is a unique isomorphism $\lambda: H^{\sharp}\to \tilde H$, to render the following diagram commutative
$$\xymatrixcolsep{6pc}
\xymatrix{
 \oplus_{i}\overline{Fil}^{i}\times_{\bar{E}^{i}}E^{i}\ar[r]^{[-1]\times0-p\cdot Id}&  \oplus_{i}\overline{Fil}^{i}\times_{\bar{E}^{i}}E^{i}\ar@{->>}[r]^{\bar{c}}\ar@{->>}[dr]^{c} &\tilde{H}    \\
 & & H^\sharp\ar[u]^{\wr}_{\lambda}}.
$$
The isomorphism $\lambda$  is compatible with twisted connections. Indeed, the twisted connection $\tilde{\nabla}$ on $\tilde{H}$ is induced by the maps
$$
\nabla': Fil^i\to Fil^{i-1}\otimes \Omega_{X_n},  \qquad i\geq 0,
$$
and
$$
\frac{\nabla'}{p} : Fil^j\to Fil^{j}\otimes \Omega_{X_n},  \qquad j\leq -1,
$$
while $ \nabla^\sharp$ on $H^\sharp$ by the formula (\ref{connection map}) and $\frac{\nabla''}{p}$ in Lemma \ref{twisted connection in second approach}. Since the vertical map $\eta\times j$ in the middle column of diagram (\ref{diagram}) is compatible with the map $\nabla'$ on $\oplus_{i}Fil^i$ and the map $\nabla''$ on $\oplus_{i}\overline{Fil}^{i}\times_{\bar{E}^{i}}E^{i}$, and it is also compatible with the map $\frac{\nabla'}{p}$ on $\oplus_{j\leq-1}Fil^j$ and the map $\frac{\nabla''}{p}$ on $\oplus_{j\leq -1}\overline{Fil}^{j}\times_{\bar{E}^{j}}E^{j}$, it follows that the map $\lambda$ is compatible with $\widetilde{\nabla}$ on $\tilde{H}$ and $ \nabla^\sharp$ on $ H^\sharp$.
\end{proof}
This completes the construction of the first functor from $\sH(X_n')$ to $\widetilde{MIC}(X_n')$, by setting $X=X_n'$ in the above two approaches. \\

Next, we are going to construct the second functor $\sF_n: \widetilde{MIC}(X_n')\to MIC(X_n)$. In contrast to the first functor, it requires the $W_{n+1}$-liftability of $X'_n$ and depends on the choice of $W_{n+1}$-liftings. Again, the method of ``local lifting-global gluing'' in the first approach of the former functor will be applied. Here ``local lifting'' means a local lifting of the relative Frobenius and the gluing morphisms are provided by the difference of two relative Frobenius liftings via the Taylor formula. We shall remind that A. Shiho \cite{Shiho} has obtained the functor from the category of quasi-nilpotent modules with integrable $p$-connections over $X_n'$ to the category of flat modules over $X_n$ under the assumptions both on the existence of $W_{n+1}$-liftings of $X_n'$, $X_n$ and on the existence of the relative Frobenius lifting over $W_{n+1}$. His construction is closely related to ours in the local case, but they were independently obtained.
\begin{proposition}\label{second functor}
Let $X$ be a smooth scheme over $S_{n}$ and $X'=X\times_{F_{S_n}}S_n$. Assume $X'$ admits a smooth lifting $\tilde X'$ over $S_{n+1}$. Then there is a functor $\mathcal{F}_n$ from the category of $\widetilde{MIC}(X')$ to the category $MIC(X)$.
\end{proposition}
\begin{proof}
We divide the whole construction into several small steps.  \\

{\itshape Step 0. } For convenience, let us introduce $\tilde X=\tilde X'\times_{F_{S_{n+1}}^{-1}}S_{n+1}$, which lifts $X$ and will be used in Step 2.  Let $\tilde \sU'=\{\tilde U'_i\}_{i\in I}$ be an open affine covering of $\tilde X'$ (assume each member is flat over $S_{n+1}$). Then by the obvious base change, it induces an open affine covering $\tilde \sU=\{\tilde U_i\}$ of $\tilde X$, and by reduction modulo $p^{n}$, open affine coverings $\sU'$ (resp. $\sU$) of $X'$ (resp. $X$). For each $i\in I$, we take a morphism $\tilde F_i: \tilde U_i\to \tilde U_i'$ over $S_{n+1}$ lifting the relative Frobenius $\tilde U\otimes \Z/p\Z\to \tilde U'\otimes \Z/p\Z$ over $k$ and denote $\tilde F_i\otimes \Z/p^{n}\Z$ by $F_i$. Given a twisted flat module $(\tilde{H},\tilde{\nabla} )$ over $X'$, its restriction to $U'_i$ is denoted by $(\tilde H_i,\tilde \nabla_i)$.  \\

{\itshape Step 1.} Each $\tilde F_i$ defines a morphism over $S_n$:
$$
\frac{d\tilde F_{i}}{p}:F_{i}^*\Omega_{U_{i}'}\to \Omega_{U_{i}}.
$$
Put $H_i:=F_{i}^*\tilde{H}_i$. As discussed before (see the paragraph after Definition \ref{quasi-nilpotence}), one can naturally regard $\tilde \nabla_i$ as an integrable $p$-connection on $\tilde H_i$. By doing so, we can consider the following formula:
\begin{align}\label{formula for connection}
\nabla_i(f\otimes e):=df\otimes e + f\cdot (\frac{d\tilde F_{i}}{p}\otimes1)( 1\otimes \tilde{\nabla}_i(e)),\quad f\in\mathcal{O}_{U_{i}},\ e\in\tilde{H}_i.
\end{align}
 \begin{claim}\label{def of connection from p-connection}
 $\nabla_i$ is well-defined and it defines an integrable connection on $H_i$.
 \end{claim}
 \begin{proof}
One computes that
 \begin{eqnarray*}
 \nabla_i(1\otimes fe)&=&(\frac{d\tilde F_{i}}{p}\otimes1)( 1\otimes \tilde{\nabla_i}(fe))=(\frac{d\tilde F_{i}}{p}\otimes1)( 1\otimes pdf\cdot e + 1\otimes f\cdot\tilde{\nabla_i}(e))\\
 &=&d(F_{i}^*f)\otimes e+ F_{i}^*f\cdot (\frac{d\tilde F_{i}}{p}\otimes1)( 1\otimes \tilde{\nabla_i}(e))=\nabla_i(F_{i}^*f\otimes e ),
\end{eqnarray*}
which shows the well-definedness. Let $\{t_1,\cdots t_d\}$ be a set of \'{e}tale local coordinates of $U_{i}$ and $\{t'_{\alpha}\}_{1\leq \alpha\leq d}$ the corresponding one of $U'_i$. Note for $1\leq j,k \leq d$ and any $e\in\tilde{H}_i$,
 $$
 \nabla_i(\partial t_j)(1\otimes e)= \sum_{\alpha=1}^d\left(\partial t_j \frac{d\tilde F_{i}}{p}(1\otimes dt'_{\alpha})\right)\otimes \tilde{\nabla}(\partial t'_{\alpha})(e).
 $$
So one computes that
$$
\begin{CD}
 \nabla_i(\partial t_k)\circ \nabla_i(\partial t_j)(1\otimes e)=\sum_{\alpha=1}^{d} \partial_{t_k}\left(\partial t_j  \frac{d\tilde F_{i}}{p}(1\otimes dt'_{\alpha})\right)\otimes \tilde{\nabla} (\partial t'_{\alpha})(e)\\
 +\sum_{1\leq \alpha,\beta\leq d}\left(\partial t_j  \frac{d\tilde F_{i}}{p}(1\otimes dt'_{\alpha})\right)\cdot \left(\partial t_k  \frac{d\tilde F_{i}}{p}(1\otimes dt'_{\beta})\right)\otimes \tilde{\nabla} (\partial t'_{\beta})\circ\tilde{\nabla} (\partial t'_{\alpha})(e)\\
 =\sum_{\alpha=1}^{d} \partial_{t_j}\left(\partial t_k  \frac{d\tilde F_{i}}{p}(1\otimes dt'_{\alpha})\right)\otimes \tilde{\nabla} (\partial t'_{\alpha})(e)\\
 +\sum_{1\leq \alpha,\beta\leq d}\left(\partial t_j  \frac{d\tilde F_{i}}{p}(1\otimes dt'_{\alpha})\right)\cdot \left(\partial t_k  \frac{d\tilde F_{i}}{p}(1\otimes dt'_{\beta})\right)\otimes \tilde{\nabla} (\partial t'_{\alpha})\circ\tilde{\nabla} (\partial t'_{\beta})(e)\\
=\nabla_i(\partial t_j)\circ \nabla_i(\partial t_k)(1\otimes e).
\end{CD}
$$
In the above second equality we have used the integrability of $\tilde \nabla_i$, i.e., $\tilde{\nabla}_i (\partial t'_\alpha)$ commutes with  $\tilde{\nabla}_i (\partial t'_\beta)$. Then  for any $f\in\mathcal{O}_{U_{i}}$ and any $e\in\tilde{H} _i$, one has

{\small
 $$
 \begin{CD}
 \nabla_i(\partial t_k)\circ \nabla_i(\partial t_j)(f\otimes e)\\
 = \frac{\partial^2 f}{\partial t_k\partial t_j}\otimes e + \frac{\partial f}{\partial t_j}\cdot\nabla_i(\partial t_k)(1\otimes e)+
 \frac{\partial f}{\partial t_k}\cdot\nabla_i(\partial t_j)(1\otimes e)+ f\cdot\nabla_i(\partial t_k)\nabla_i(\partial t_j)(1\otimes e)\\
  = \frac{\partial^2 f}{\partial t_j\partial t_k}\otimes e + \frac{\partial f}{\partial t_j}\cdot\nabla_i(\partial t_k)(1\otimes e)+
 \frac{\partial f}{\partial t_k}\cdot\nabla_i(\partial t_j)(1\otimes e)+ f\cdot\nabla_i(\partial t_j)\nabla_i(\partial t_k)(1\otimes e)\\
 =\nabla_i(\partial t_j)\circ \nabla_i(\partial t_k)(f\otimes e).
  \end{CD}
 $$
}

 So $\nabla_i(\partial t_k)$ commutes with $\nabla_i(\partial t_j)$ for $1\leq k,j \leq d$, that is,  $\nabla_i$ is integrable as claimed.
 \end{proof}
{\itshape Step 2.}  The previous step provides us a set of local flat modules $\{(H_i,\nabla_i)\}_{i\in I}$ and we want to glue them into one global flat module. The point is to use the Taylor formula involving the difference of two relative Frobenius liftings to construct the gluing morphisms. This is well-known if $p$ is non-nilpotent in the base ring, and we use the formalism of a twisted connection to make sense of the $p$-powers in the denominators of the Taylor formula. \\

Let $\tilde U_i=\Spec\ \tilde R_i, i=1,2$ be two smooth schemes over $S_{n+1}$, equipped with a lifting of relative Frobenius $\tilde F_i: \tilde U_i\to \tilde U'_i$. Assume there is a set of \'{e}tale local coordinates $\{\tilde t_i\}$ for $\tilde U_1$ and suppose there is a morphism $\iota: \tilde U_2\to \tilde U_1$ over $S_{n+1}$, which induces the morphism $\iota': \tilde U'_2\to \tilde U'_1$. Let $R_i=\tilde R_i\otimes \Z/p^{n}\Z$, $U_i=\Spec\ R_i$ and $F_i=\tilde F_i\otimes \Z/p^{n}\Z$. Let $\{\tilde t'_i\}$ (resp. $\{t_i\}$) be the induced coordinate functions on $\tilde U'_1$ (resp. $U_1$). Given a twisted flat module $(\tilde H,\tilde \nabla)$ over $U'_1$, there is an isomorphism of $R_2$-modules $G_{21}: F_2^*\iota'^*\tilde H\cong \iota^*F_1^*\tilde H$ given by the Taylor formula
\begin{equation}\label{gluing function}
G_{21}(e\otimes 1)=\sum_{J}\frac{\tilde{\nabla}(\partial)^J}{J!}(e)\otimes z^{J}, \ e\in \tilde H,
\end{equation}
where $J$ is a multi-index $J:=(j_1,\cdots, j_d)$ with each component $j_l\geq 0$,  $J!:=\prod_{l=1}^{d}j_l!$, and
$$
\frac{\tilde{\nabla}(\partial)^J}{J!}:= \frac{(\tilde{\nabla}(\partial t'_1))^{j_1}}{j_1!}\circ\cdots \circ \frac{(\tilde{\nabla}(\partial t'_d))^{j_d}}{j_d!}
$$
with
$$
z^J:=\prod_{l=1}^dz_l^{j_l}, \quad
 z_l:= \frac{(\tilde F_2^*\iota'^*)(\tilde t'_l)-(\iota^*\tilde F_1^*)(\tilde t'_l)}{p}.
$$
As already explained right after Definition \ref{quasi-nilpotence}, $\frac{\tilde{\nabla}(\partial)^J}{J!}$ converges to zero $p$-adically as $|J|\to\infty$, the summation in the formula (\ref{gluing function}) is actually a finite sum. There is the cocycle relation between the isomorphisms for three objects, namely the equality $G_{31}=G_{21}\circ G_{32}$ holds. The proof is mostly formal: suppressing $\iota$s in the Taylor formula (\ref{gluing function}) and writing $\hat{z}$ (resp. $\tilde{z}$) for the $z$-function appeared in $G_{32}$ (resp. $G_{31}$). Then as $\tilde{z}=z+\hat{z}$, one calculates that
\begin{eqnarray*}
G_{21}\circ G_{32}(e\otimes1)&=&\sum_{I}\sum_{J}\frac{\tilde{\nabla}(\partial)^{I+J}}{I!J!}(e)\otimes  z^J \cdot  \hat{z}^I\\
 &=&\sum_{K}\frac{\tilde{\nabla}(\partial)^K}{K!}(e)\otimes \left(\sum_{I+J=K}\frac{K!}{I!J!}\cdot z^J\cdot\hat{z}^I \right) \\
 &=&\sum_{K}\frac{\tilde{\nabla}(\partial)^K}{K!}(e)\otimes(z+\hat{z})^K\\
 &=&\sum_{K}\frac{\tilde{\nabla}(\partial)^K}{K!}(e)\otimes \tilde{z}^K\\
 &=&G_{31}(e\otimes 1).
\end{eqnarray*}
The cocycle condition follows. Returning to our case, we shall apply the Taylor formula in the following way: write $\tilde U_{ij}=\tilde U_{i}\cap \tilde U_j$ and take a relative Frobenius lifting $\tilde F_{ij}: \tilde U_{ij}\to \tilde U'_{ij}$. Let $\iota_{1}: \tilde U_{ij}\to \tilde U_i$ and $\iota_2: \tilde U_{ij}\to \tilde U_j$ be the natural inclusions. Then we obtain the isomorphisms $\alpha_i$ and $\alpha_j$ in the following diagram:
$$
\begin{tikzcd}
\,&F_{ij}^*\tilde H|_{U'_{ij}}\arrow{rd}{\alpha_j}[swap]{\cong}\arrow{ld}{\cong}[swap]{\alpha_i}\\
H_i|_{U_{ij}}=F_i^*\tilde H|_{U'_{ij}}&&F_j^*\tilde H|_{U'_{ij}}=H_j|_{U_{ij}}
\end{tikzcd}
$$
The we define the isomorphism $G_{ij}:=\alpha_j\circ \alpha_i^{-1}: H_i|_{U_{ij}}\to H_j|_{U_{ij}}$ and the set of isomorphisms $\{G_{ij}\}$ satisfies the cocycle condition. There is one more property of $G_{ij}$, namely, it is compatible with connections. For $1\leq l\leq d$, let us identify $l$ with the multiple index $(0,\cdots ,1,\cdots, 0)$ with 1 at the $l$-th position. Then one calculates that
$$
(Id\otimes G_{ij})(\nabla_i(e\otimes 1))=\sum_{J}\sum_{l=1}^d\frac{ \tilde{\nabla}(\partial)^{J+l}}{J!}(e)\otimes z^J\cdot \frac{d\tilde F_{i}}{p}(dt'_l) ,
$$
and
$$
\nabla_j(G_{ij}(e\otimes 1))=\sum_{J}\sum_{l=1}^d\frac{\tilde{\nabla}(\partial)^{J+l}}{J!}(e)\otimes z^J\cdot \left( \frac{d\tilde F_{j}}{p}(dt'_l)
+ dz_l\right).
$$
As
$$
 dz_l=\frac{d\tilde F_{i}}{p}(dt'_l)-\frac{d\tilde F_{j}}{p}(dt'_l),
$$
it follows that
$$
(Id\otimes G_{ij})(\nabla_i(e\otimes 1))=\nabla_j(G_{ij}(e\otimes 1)).
$$
Thus, we use the set $\{G_{ij}\}$ to glue the local flat modules $\{(H_i,\nabla_i)\}$ together and obtain an object $(H,\nabla)$ in $MIC(X)$. \\

{\itshape Step 3.} We need to show the so-constructed object $(H,\nabla)$ is independent of the choice of relative Frobenius liftings and the choice of open affine coverings. But, after our argument on a similar independence-type statement in Proposition \ref{gluing}, this step becomes entirely formal. Therefore, to summarize the whole steps, we construct a flat module $(H,\nabla)$ over $X$ up to canonical isomorphism from any twisted flat module $(\tilde H,\tilde \nabla)$ over $X'$, and this gives our second functor $\sF_n$.
\end{proof}
\begin{remark}\label{global existence of Frobenius lifting}
In above, a simplification can be made when a global lifting $F_{\tilde X}$ of the relative Frobenius over $\tilde X$ (a $W_{n+1}$-lifting of $X$) exists. Set $F_X=F_{\tilde X}\otimes \Z/p^{n}\Z$. In this case, one has a globally defined morphism $\frac{dF_{\tilde X}}{p}: F_X^*\Omega_{X'}\to \Omega_X$. Then $H=F_X^*\tilde H$ and $\nabla$ is define by Formula (\ref{formula for connection}) in which various local $\frac{d\tilde F_{i}}{p}$s are replaced simply by $\frac{dF_{\tilde X}}{p}$. The reason is as follows: in Step 0, we may take an open affine covering $\tilde \sU=\{\tilde U_i\}$ of $\tilde X$ (whose element is flat over $W_{n+1}$) such that $F_{\tilde X}: \tilde U_i\to \tilde U'_i$ for all $i$ and then take $\tilde F_i$ to be the restriction to $F_{\tilde X}$ to $\tilde U_i$. Then it follows that the gluing functions $G_{ij}$s are all identity. 
\end{remark}

Now we proceed to the proof of Theorem \ref{lifting of inverse cartier}.
\begin{proof}[Proof of Theorem \ref{lifting of inverse cartier}]
We define the functor $\sC_n^{-1}: \sH(X_n')\to MIC(X_n)$ to be the composite of the
functor $\mathcal{T}_n: \sH(X_n')\to \widetilde{MIC}(X'_n)$ in Proposition \ref{gluing} and the functor
$\sF_n: \widetilde{MIC}(X'_n)\to MIC(X_n)$ in Proposition \ref{second functor}. From their very constructions, the two functors $\sT_n$ and $\sF_n$ are compatible with the reduction modulo $p^{n-1}$. So is their composite. The equivalence of the functor $\sC_1^{-1}$ over $W_1=k$ with the inverse Cartier transform of Ogus-Vologodsky \cite{OV} has been verified in \cite{LSZ}.
\end{proof}

\section{Higgs correspondence in mixed characteristic}
The section aims to establish the Higgs correspondence between the category of strict $p^n$-torsion Fontaine modules with endomorphism $W_n(\F_{p^f})$ (Variant 2 \S2) and the category of periodic Higgs-de Rham flows over $X_n$, for $n\geq 2$, so that it lifts the one in positive characteristic and its limit, as $n$ goes to infinity, yields the Higgs correspondence in mixed characteristic. \\

We proceed first to the definition of a Higgs-de Rham flow over $X_n$. We use the notation from \S 4. Let $\pi_n: X'_n\to X_n$ be the natural morphism by base change, which induces the obvious equivalence of categories $\pi_n^*: \sH(X_n)\cong \sH(X'_n)$. We define the functor
$$
C_n^{-1}:=\sC_{n}^{-1}\circ \pi_n^*: \sH(X_n)\to MIC(X_n).
$$
\begin{definition}\label{HDRFn}
A Higgs-de Rham flow over $X_n$ is a diagram of the following form:

\begin{adjustbox}{scale=0.85}
\xymatrix{
                                                                                                   &   &     &  (H_0,\nabla_0)\ar[dr]^{Gr_{Fil_0}}    &   &  (H_1,\nabla_1)\ar[dr]^{Gr_{Fil_1}}    \\
   (\bar{H}_{-1},\bar{\nabla}_{-1})\ar[dr]^{Gr_{\bar{Fil}_{-1}}}&&(E_0,\theta_0) \ar[ur]^{C_n^{-1}}  \ar[d] & &(E_1,\theta_1) \ar[ur]^{C_n^{-1}}&&\ldots.       \\
                                                                                                   &(\bar E_{-1},\bar \theta_{-1})\ar[r]^{\stackrel{\bar\phi}{\cong}}& (\bar{E}_{0}, \bar{\theta}_0)}
\end{adjustbox}

Here in the upper line $(H_i,\nabla_i):=C_n^{-1}(E_i,\theta_i)$ and $Fil_i$ is a finite exhaustive decreasing filtration of $\sO_{X_n}$-submodules on $H_i$ satisfying Griffiths' transversality with respect to $\nabla_i$; in the middle line each term $(E_i,\theta_i)$ is a graded Higgs module over $X_n$ and $(\bar{H}_{-1},\bar{\nabla}_{-1},\bar{Fil}_{-1})$ is a de Rham module over $X_{n-1}$; in the bottom line $\bar \phi$ is an isomorphism of graded Higgs modules from $(\bar E_{-1},\bar \theta_{-1})=Gr_{\bar{Fil}_{-1}}(\bar{H}_{-1},\bar{\nabla}_{-1})$ to $(\bar{E}_{0}, \bar{\theta}_0)$, the mod $p^{n-1}$ reduction of $(E_0,\theta_0)$.
\end{definition}
We emphasize that, in the above definition, the filtrations $Fil_i, i\geq 0$ and the isomorphism $\bar \phi$ are part of the defining data of the flow. Also, an explanation of the various inverse Cartier transforms in the flow is in order:  $C_n^{-1}(E_0,\theta_0)$ is the abbreviation for $$C_n^{-1}(E_0,\theta_0,\bar{H}_{-1},\bar{\nabla}_{-1},\bar{Fil}_{-1},\bar \phi),$$ and $C_n^{-1}(E_i,\theta_i)$ for $i\geq 1$ is the abbreviation 
for $$C_n^{-1}(E_i,\theta_i, (C_n^{-1}(E_{i-1},\theta_{i-1}),Fil_{i-1})\otimes \Z/p^{n-1}\Z), Id).$$
Comparing with the notion of a periodic Higgs-de Rham flow over $X_1$ (Definition \ref{definition of periodic flow}), one finds that the extra data in the lower left corner in Definition \ref{HDRFn} puts an additional condition on $(E_0,\theta_0)$ when $n\geq 2$. This is caused by the construction of the functor $C_n^{-1}, n\geq 2$. It is interesting to characterize those graded Higgs modules satisfying the condition, but more importantly, to know whether this assumption made on $C_n^{-1}, n\geq 2$ could be relaxed after all. \\

Next, we turn to the central notion of a periodic Higgs-de Rham over $X_n$ which is defined in an inductive way. For a flat bundle $(H,\nabla)$ over $X_n, n\geq 1$, a finite exhaustive decreasing filtration on $H$ is said to be a \emph{Hodge filtration}, if it consists of locally free subsheaves of $H$, is locally split, and obeys Griffiths' transversality with respect to $\nabla$. 
\begin{definition}\label{def for periodic flow over Wn}
A periodic Higgs-de Rham flow over $X_n, n\geq 2$ of period $f\in \N$ consists of the following data:
\begin{enumerate}
   \item  A periodic Higgs-de Rham flow $(\bar{E},\bar{\theta}, \bar{Fil}_0,\cdots, \bar{Fil}_{f-1}, \bar \phi)\in HDF_{p-2,f}(X_{n}/W_n)$;
\item A graded Higgs bundle $(E,\theta)\in HIG_{p-2}(X_n)$ lifting $(\bar{E},\bar{\theta})$;
  \item A Hodge filtration $Fil_i$ of level $\leq p-2$ on  $C_n^{-1}(E_{i},\theta_{i})$ lifting the Hodge filtration $\bar{Fil}_{i}$ on $C_{n-1}^{-1}(\bar E_{i},\bar \theta_{i})$;
  \item An isomorphism of graded Higgs modules over $X_n$
      $$
      \phi: Gr_{Fil_{f-1}}(C_n^{-1}(E_{f-1},\theta_{f-1}))\cong (E,\theta)
      $$
      lifting $\bar{\phi}$.
  \end{enumerate}
\end{definition}
In the above definition, $(E_0,\theta_0):=(E,\theta)$ and $(C_{n-1}^{-1}(\bar{E}_{f-1},\bar{\theta}_{f-1}), \bar{Fil}_{f-1},\bar{\phi})$ together make an object in the category $\mathcal{H}(X_n)$ so that $C_n^{-1}(E_0,\theta_0)$ is naturally defined. Also, $C_{n}^{-1}(E_i,\theta_i)$ for $i\geq 1$ is naturally defined, which is simply
$$C_n^{-1}(E_i,\theta_i,C_{n-1}^{-1}(E_{i-1},\theta_{i-1}),\bar{Fil}_{i-1}).$$ By Corollary \ref{Hodge filtration property of filtrations in preperioidc flow}, the filtrations on a periodic Higgs-de Rham flow over $X_1$ are indeed Hodge filtrations.\\

Thus the data of a periodic Higgs-de Rham flow of period $f$ over $X_n,  n\geq 2$ is encoded in the following tuple
$$
(E,\theta,Fil_0,\cdots,Fil_{f-1},\phi,\bar E,\bar \theta,\bar{Fil}_0,\cdots,\bar{Fil}_{f-1},\bar \phi).
$$
A morphism between two tuples is given by a pair $(f,\bar f)$, where $\bar f$ is a morphism in the category $HDF_{p-2,f}(X_{n}/W_n)$ and $f$ is a morphism of graded Higgs modules over $X_n$ lifting $\bar f$ and satisfying natural properties. Let us explain this in the one-periodic case in detail. Let $(E_i,\theta_i,Fil_i,\phi_i,\bar E_i,\bar \theta_i,\bar{Fil}_i,\bar \phi_i), i=1,2$ be two one-periodic flows over $X_n$. Then a morphism
$$
(f,\bar f): (E_1,\theta_1,Fil_1,\phi_1,\bar E_1,\bar \theta_1,\bar{Fil}_1,\bar \phi_1)\to (E_2,\theta_2,Fil_2,\phi_2,\bar E_2,\bar \theta_2,\bar{Fil}_2,\bar \phi_2)
$$
means the following: first,
$$
\bar f: (\bar E_1,\bar \theta_1,\bar{Fil}_1,\bar \phi_1)\to (\bar E_2,\bar \theta_2,\bar{Fil}_2,\bar \phi_2)
$$
is a morphism in $HDF_{p-2,1}(X_{n}/W_n)$; second, $f: (E_1,\theta_1)\to (E_1,\theta_2)$ is a morphism of graded Higgs modules over $X_n$ lifting $$\bar f: (\bar E_1,\bar \theta_1)\to (\bar E_2,\bar \theta_2);$$ third, the morphism
$$
C_{n}^{-1}(f): C_{n}^{-1}(E_1,\theta_1)\to C_n^{-1}(E_2,\theta_2),
$$
which is naturally defined by the previous two properties, is compatible with the Hodge filtrations and the induced morphism on graded Higgs modules is compatible with $\phi$s, that is, the following diagram commutes:
\begin{equation*}
 \begin{CD}
 Gr_{Fil_1}C_n^{-1}(E_1,\theta_1)@>\phi_1>>(E_1,\theta_1)\\
 @VGrC_n^{-1}(f)VV@  VVfV\\
  Gr_{Fil_2}C_n^{-1}(E_2,\theta_2)@>\phi_2>>(E_2,\theta_2).
 \end{CD}
\end{equation*}
Thus, the category $HDF_{p-2,f}(X_{n+1}/W_{n+1})$ of periodic Higgs-de Rham flow of period $f$ over $X_n$ is indeed inductively defined. Also, based upon the previous two definitions, it is straightforward to define a preperiodic Higgs-de Rham flow over $X_n$ (the key is to assure the well-definedness of the inverse Cartier transform of each Higgs term). Since this will be not used in the sequel, we leave this task to the reader.
Recall that $MF^{\nabla}_{[0,p-2],f}(X_{n+1}/W_{n+1})$ is the category of strict $p^n$-torsion Fontaine modules with extra endomorphism $W_{n}(\F_{p^f})$. The following theorem lifts Theorem \ref{correspondence in the type (0,f) case} in characteristic $p$ (but notice the stronger restriction on the Hodge-Tate weight).
\begin{theorem}\label{periodic corresponds to FF module over a truncated witt ring}
Let $X_{n+1}$ be a smooth scheme over $W_{n+1}$. For $f\in \N$, there is an equivalence of categories between the category $MF^{\nabla}_{[0,p-2],f}(X_{n+1}/W_{n+1})$ and the category $HDF_{p-2,f}(X_{n+1}/W_{n+1})$.
\end{theorem}
Similar to the characteristic $p$ case, the theorem will be reduced to the one-periodic case. Let us introduce the category of one-periodic Higgs-de Rham flows over $X_n$ with endomorphism structure $W_n(\F_{p^{f}})$. Its object is a tuple $(E,\theta,Fil,\phi,\iota,\bar E,\bar \theta,\bar{Fil},\bar \phi)$, where $(E,\theta,Fil,\phi,\bar E,\bar \theta,\bar{Fil},\bar \phi)$ is an object in the category $HDF_n:=HDF_{p-2,f}(X_{n+1}/W_{n+1})$ and
$$
\iota: W_n(\F_{p^f})\hookrightarrow \End_{HDF_{n}}(E,\theta,Fil,\phi,\bar E,\bar \theta,\bar{Fil},\bar \phi)
$$
is an embedding of $W_n(\F_p)$-algebras. A morphism of this category is a morphism of one-periodic Higgs-de Rham flows compatible with endomorphism structures. Thus Theorem \ref{periodic corresponds to FF module over a truncated witt ring} follows from the next two propositions.
\begin{proposition}\label{one-periodic case}
There is an equivalence of categories between the category $MF^{\nabla}_{[0,p-2],1}(X_{n+1}/W_{n+1})$ and the category $HDF_{p-2,1}(X_{n+1}/W_{n+1})$.
\end{proposition}
This proposition is just the one-periodic case of Theorem \ref{periodic corresponds to FF module over a truncated witt ring}, which implies immediately that $MF^{\nabla}_{[0,p-2],f}(X_{n+1}/W_{n+1})$ is equivalent to the category of one-periodic Higgs-de Rham flows over $X_n$ with endomorphism structure $W_n(\F_{p^f})$. We postpone its proof and show first the next
\begin{proposition}\label{linear algebra prop}
There is an equivalence of categories between the category of one-periodic Higgs-de Rham flows over $X_n$ with endomorphism structure $W_n(\F_{p^f})$ and the category $HDF_{p-2,f}(X_{n+1}/W_{n+1})$.
\end{proposition}
\begin{proof}
Recall that we have chosen a primitive element $\xi_1\in \F_{p^f}$ in the proof of Lemma \ref{lemma from 0,f to f}. Let $\xi$ be the Teichm\"{u}ller lift of $\xi_1$ in $W(\F_{p^f})$ and let $\xi_n$ be the mod $p^{n}$ reduction of $\xi$. Thus $\xi_n$ is a generator of $W_n(\F_{p^f})$ as $W_{n}(\F_p)$-algebra. The Frobenius automorphism of $W_n=W_n(k)$ acts on $\xi_n$ by the power $p$ map. We prove by induction on $n$. The $n=1$ case is Proposition \ref{correspondence from HB_f and HB_(0,f)}, where we have constructed the functor $\sE$ and its quasi-inverse $\sD$. Rewrite them into $\sE_1$ and respectively $\sD_1$. As the induction hypothesis, we assume that we have constructed a sequence of functors $\{\sE_{i}\}_{1\leq i\leq n-1}$ and $\{\sD_i\}_{1\leq i\leq n-1}$ such that (i) $\sE_i$ and $\sD_i$ are quasi-inverse to each other and therefore give a equivalence of categories between the category of one-periodic Higgs-de Rham flows over $X_i$ with endomorphism structure $W_i(\F_{p^f})$ and the category $HDF_{p-2,f}(X_i)$; (ii) $\sE_i$ (resp. $\sD_i$) lifts $\sE_{i-1}$ (resp. $\sD_{i-1}$). Now we proceed to show the case for $n$. Let us start with an object of $HDF_{p-2,f}(X_n)$:
$$
(E,\theta,Fil_0,\cdots,Fil_{f-1},\phi,\bar E,\bar \theta,\bar{Fil}_0,\cdots,\bar{Fil}_{f-1},\bar \phi).
$$
Following the characteristic $p$ case, we put
$$
(G,\eta):=\bigoplus_{i=0}^{f-1}(E_i,\theta_i),
$$
where $(E_0,\theta_0)=(E,\theta)$ and $(E_i,\theta_i)=C^{-1}_n(E_{i-1},\theta_{i-1}), i\geq 1$ are the remaining Higgs terms defined inductively in the flow. The Hodge filtration $Fil$ on $C_n^{-1}(G,\eta)$ as well as the isomorphism $\tilde \phi$ of graded Higgs modules over $X_n$ are defined exactly in the same way as the characteristic $p$ case (see the paragraph before Lemma \ref{lemma from 0,f to f}). Clearly, $(G,\theta)$ (resp. $Fil$ and $\tilde \phi$) lifts $(\bar G,\bar \theta)$ (resp. $\bar{Fil}$ and $\bar{ \tilde{ \phi}}$) and therefore $(G,\eta,Fil,\tilde \phi,\bar G,\bar \theta,\bar{Fil},\bar{ \tilde{\phi}})$ is a one-periodic flow over $X_n$. Replacing $\xi_1$ in the proof of Lemma \ref{lemma from 0,f to f} with $\xi_n$, one equips this one-periodic flow with an endomorphism structure $W_n(\F_{p^f})$. This gives us a functor $\sE_n$. Conversely,
to a one-periodic Higgs-de Rham flow $(G,\eta,Fil,\phi,\iota, \bar G,\bar \eta,\bar{Fil},\bar \phi)$ over $X_n$ with the endomorphism structure $W_n(\F_{p^f})$, we associate an $f$-periodic Higgs-de Rham flow over $X_n$ as follows: first, the induction hypothesis gives us an $f$-periodic flow over $X_{n-1}$ (by abuse of notation we have omitted the part over $X_{n-1}$ in the following expression):
$$(\bar E,\bar \theta, \bar{\widetilde{Fil_0}},\cdots,\bar{\widetilde{Fil_{f-1}}},\bar{\tilde \phi}).$$
Second, let $(G,\eta)=\oplus_{i=0}^{f-1}(G_i,\eta_i)$ be the eigen-decomposition under the endomorphism $\iota(\xi_n)$. Then $C_n^{-1}(G_i,\eta_i)$ is naturally defined and one has the eigen-decomposition under $C_n^{-1}(\xi_n)$:
$$
(C_n^{-1}(G,\eta),Fil)=\bigoplus_{i=0}^{f-1}(C_n^{-1}(G_i,\eta_i),Fil_i).
$$
So we put $(E,\theta)=(G_0,\eta_0)$, and following strictly the constructions in the characteristic $p$ case (see Lemma \ref{lemma from f to 0,f}), one obtains
the filtrations $\widetilde {Fil_i}, 0\leq i\leq f-1$ and the isomorphism $\tilde \phi$ of graded Higgs modules over $X_n$, so that the extended tuple
$$
(E,\theta, \widetilde{Fil_0},\cdots,\widetilde{Fil_{f-1}},\tilde \phi, \bar E,\bar \theta, \bar{\widetilde{Fil_0}},\cdots,\bar{\widetilde{Fil_{f-1}}},\bar{\tilde \phi})
$$
is an object in $HDF_{p-2,f}(X_n)$. This give us the functor $\sD_n$ in the reverse direction. Given the proof of Proposition \ref{correspondence from HB_f and HB_(0,f)}, the proof for that $\sE_n$ and $\sD_n$ give an equivalence of categories becomes completely formal and is therefore omitted. Finally, the lifting properties of $\sE_n$ and $\sD_n$
are direct consequences of our choice of $\xi_n$ at the beginning.
\end{proof}
The remaining paragraphs are devoted to the proof of Proposition \ref{one-periodic case}. In the following, we choose and then fix an open affine covering $\mathcal{U}= \{U_i\}_{i\in I}$ of $X_{n+1}/W_{n+1}$ whose element is flat over $W_{n+1}$, and for each $i$, an absolute Frobenius lifting $F_i$ over $U_i$. First a lemma:
\begin{lemma}\label{relation}
Let $(H,\nabla,Fil,\Phi)$ be a strict $p^n$-torsion Fontaine module. Let $(\bar H,\bar \nabla,\bar{Fil})$ be the mod $p^{n-1}$ reduction of $(H,\nabla,Fil)$ and $(E,\theta)$ the associated graded Higgs bundle. Then the relative Frobenius $\Phi$ naturally induces an isomorphism of flat bundles over $X_n$:
$$
\tilde \Phi: C_{n}^{-1}(E,\theta)\cong (H,\nabla).
$$
\end{lemma}
\begin{proof}
The lemma follows from the strong $p$-divisibility and horizontality of $\Phi$ (see \S2) and the very construction of $C_n^{-1}$. For $n=1$, this is Proposition 1.4 \cite{LSZ}. For simplicity, let us ignore the issue of the obvious base change caused by the Frobenius automorphism of the base ring $W_n$ in the argument. By the first approach of the functor $\mathcal{T}_n$, it follows that
$$
\mathcal{T}_n(E,\theta,\bar H,\bar \nabla,\bar{Fil})= G_n(H,\nabla, Fil).
$$
Write it to be $(\tilde{H},\tilde{\nabla})$. Let $\mathcal{U}_n= \{U_{i,n}\}_{i\in I}$ be the induced open affine covering of $X_n$, and let $F_{i,n}$ be the induced absolute Frobenius lifting over $U_{i,n}$. Then the evaluation $\Phi_i:=\Phi_{(U_{i,n},F_{i,n})}$ of $\Phi$ is an isomorphism of local flat bundles:
$$
\Phi_i :   (F_{i,n}^*\tilde{H}|_{U_{i,n}}, F_{i,n}^*(\tilde{\nabla}|_{U_{i,n}})) \cong  (H, \nabla)|_{U_{i,n}},
$$
where the connection $F_i^*(\tilde{\nabla}|_{U_{i,n}})$ is defined via the formula (\ref{formula for connection}). The fact that different evaluations of $\Phi$ are related via the Taylor formula means
\begin{eqnarray}\label{taylor formula}
 \Phi_i &=&  \Phi_j\circ \varepsilon_{ij},
\end{eqnarray}
where the $ \varepsilon_{ij}$ is given by the formula (\ref{gluing function}) (replacing the indices $21$ in that formula with $ij$). Thus, the very construction of $C_n^{-1}$ means exactly that the local isomorphisms $\Phi_i$s glue together into a global one $\tilde \Phi$ from $C_n^{-1}(E,\theta)$ to $(H,\nabla)$.
\end{proof}
\begin{proof}[Proof of Proposition \ref{one-periodic case}]
We divide the whole proof into three steps, following the one in the characteristic $p$ case. The proof is by induction on $n$, where the $n=1$ case is Proposition \ref{correspondence in the type (0,1) case}. Rewrite the functors in the proof of Proposition \ref{correspondence in the type (0,1) case} by $\mathcal{GR}_1:=\mathcal{GR}$ and $\mathcal{IC}_1:=\mathcal{IC}$. As the induction hypothesis, we assume then the existence of functors $\mathcal{GR}_i: MF^{\nabla}_{[0,p-2],1}(X_{i+1}/W_{i+1})\to HDF_{p-2,1}(X_{i+1}/W_{i+1})$ and $\mathcal{IC}_i$ in the opposite direction for $1\leq i\leq n-1$ such that (i) $\mathcal{GR}_i$ and $\mathcal{IC}_i$ are quasi-inverse to each other, and (2) $\mathcal{GR}_i$ (resp. $\mathcal{IC}_i$) lifts $\mathcal{GR}_{i-1}$ (resp. $\mathcal{IC}_{i-1}$). In the following, we construct a lifting $\mathcal{GR}_n$ (resp. $\mathcal{IC}_n$) of $\mathcal{GR}_{n-1}$ (resp. $\mathcal{IC}_{n-1}$) such that $\mathcal{GR}_n$ and $\mathcal{IC}_n$ are quasi-inverse to each other.\\

{\itshape From Fontaine module to one-periodic Higgs-de Rham flow}: Let $(H,\nabla,Fil,\Phi)$ be a strict $p^{n}$-torsion Fontaine module over $X_{n+1}/W_{n+1}$. Its reduction mod $p^{n-1}$ $(\bar H,\bar \nabla,\bar{Fil},\bar \Phi)$ is a strict $p^{n-1}$-torsion Fontaine module. Let
$$
(\bar E,\bar \theta, \bar{Fil}_{exp},\bar \phi, \bar{\bar E},\bar{\bar \theta}, \bar{\bar{Fil}}_{exp},\bar{\bar \phi})
$$
be the corresponding one-periodic flow over $X_{n-1}$ to $(\bar H,\bar \nabla,\bar{Fil},\bar \Phi)$ via the functor $\mathcal{GR}_{n-1}$. Set $(E,\theta)=Gr_{Fil}(H,\nabla)$. By Lemma \ref{relation}, we define $Fil_{exp}$ on $C_n^{-1}(E,\theta)$ to be the pull-back of $Fil$ on $H$ via the isomorphism $\tilde \Phi$. Thus, we obtain also an isomorphism of graded Higgs modules over $X_n$:
$$
\phi:=Gr(\tilde{\Phi}): Gr_{Fil_{exp}}(C_n^{-1}(E,\theta))\cong Gr_{Fil}(H,\nabla)=(E,\theta).
$$
The lifting property of the inverse Cartier transform in Theorem \ref{lifting of inverse cartier} implies that the so-obtained tuple $(E,\theta,Fil_{exp},\phi)$ lifts $(\bar E,\bar \theta, \bar{Fil}_{exp},\bar \phi)$, so that
$$
\mathcal{GR}_n(H,\nabla,Fil,\Phi):=(E,\theta,Fil_{exp},\phi, \bar E,\bar \theta, \bar{Fil}_{exp},\bar \phi)
$$
is a one-periodic Higgs-de Rham flow over $X_n$. Clearly, the functor $\mathcal{GR}_n$ lifts $\mathcal{GR}_{n-1}$.\\

{\itshape From one-periodic Higgs-de Rham flow to Fontaine module}:  From a given object $(E,\theta,Fil,\phi,\bar E,\bar \theta,\bar{Fil},\bar \phi)\in HDF_{p-2,1}(X_n)$, one derives immediately the de Rham module
$$
(H,\nabla,Fil):=(C_n^{-1}(E,\theta),Fil).
$$
In order to complete it into a Fontaine module, it remains to put a relative Frobenius $\Phi$ on it. Let $(\bar H,\bar \nabla)$ be the mod $p^{n-1}$ reduction of $(H,\nabla)$ which is equal to $C_{n-1}^{-1}(\bar E,\bar \theta)$. Set $(E_1,\theta_1)=Gr_{Fil}(H,\nabla)$. Then we have two objects in the category $\sH(X_n)$: $(E,\theta,\bar H,\bar \nabla,\bar{Fil},\bar \phi)$ and $(E_1,\theta_1,\bar H,\bar \nabla,\bar{Fil},Id)$. Then $\phi: (E_1,\theta_1)\cong (E,\theta)$ and the identity map on $(\bar H,\bar \nabla,\bar{Fil})$ give rise to an isomorphism
$$
\phi: (E_1,\theta_1,\bar H,\bar \nabla,\bar{Fil},Id)\cong (E,\theta,\bar H,\bar \nabla,\bar{Fil},\bar \phi).
$$
and therefore an isomorphism $\tilde \phi:=\mathcal{T}_n(\phi)$ of twisted flat modules after Proposition \ref{gluing}:
$$
(\tilde{H},\tilde \nabla):=\mathcal{T}_n(E_1,\theta_1,\bar H,\bar \nabla,\bar{Fil},Id)\cong \mathcal{T}_n(E,\theta,\bar H,\bar \nabla,\bar{Fil},\bar \phi):=(\tilde{H}_{-1},\tilde \nabla_{-1}).
$$
Note the isomorphism $\mathcal{T}_1(\phi)$ for $n=1$ is nothing but the original isomorphism $\phi$ between graded Higgs bundles. Then for an open affine covering of $X_{n+1}$ and the set of Frobenius liftings as given in Lemma \ref{relation}, following the method in the characteristic $p$ case, we define simply an $\sO_{U_i}$-isomorphism by
$$
\Phi_i=F_{i,n}^*(\tilde \phi): F_{i,n}^*\tilde{H}\cong F_{i,n}^*(\tilde{H}_{-1}),
$$
where the latter module is just $H|_{U_{i,n}}=C_n^{-1}(E,\theta)|_{U_{i,n}}$. At this point, the triple $(H|_{U_{i,n}},Fil|_{U_{i,n}},\Phi_i)$ is a local $p$-torsion Fontaine module without connection (i.e. it is an object in the local category $MF(R_{i,n})$ with $R_{i,n}=\Gamma(U_{i,n},\sO_{U_{i,n}})$ as given in Page 31 \cite{Fa1}). By Theorem 2.1 \cite{Fa1} ii), one knows that $H$ is a \emph{locally free} $\sO_{X_n}$-module. Moreover, by the construction of the functor $\sF_n$, $\Phi_i$ is indeed an isomorphism of flat modules:
$$
\Phi_i: (F_{i,n}^{*}\tilde H|_{U_{i,n}},F_{i,n}^*\tilde \nabla|_{U_{i,n}})\cong (H,\nabla)|_{U_{i,n}}.
$$
This gives the horizontal property as required by an evaluation of the relative Frobenius (see Variant 1 \S2). It remains to explain $\Phi_i$s are related via the Taylor formula (\ref{taylor formula}). Let $e$ be a local section of $\tilde H$ over $U_{i,n}$. Then one computes that over $U_{ij,n}:=U_{i,n}\cap U_{j,n}$,
\begin{eqnarray*}
G_{ij}\circ  \Phi_i(e\otimes 1) &=&\sum_{J}\frac{\tilde{\nabla}_{-1}(\partial)^J(\tilde \phi(e))}{J!}\otimes z^{J}  \\
    &=& \sum_{J}\frac{\tilde{\phi}(\tilde{\nabla} (\partial)^J(e))}{J!}\otimes z^{J} \\
    &=&\Phi_j(\sum_{J}\frac{\tilde{\nabla} (\partial)^J(e)}{J!}\otimes z^{J})= \Phi_j\circ \varepsilon_{ij}(e\otimes 1).
\end{eqnarray*}
The second equality follows from the property that $\tilde \phi$ respects the twisted connections, and the last equality follows from the transition over $U_{ij,n}$ of a local section of $\tilde H$ over $U_{i,n}$ to a local section of $\tilde H$ over $U_{j,n}$. So we obtain a relative Frobenius $\Phi$ from local $\Phi_i$s and thus a strict $p^n$-torsion Fontaine module over $X_n$:
$$
\mathcal{IC}_n(E,\theta,Fil,\phi,\bar E,\bar \theta,\bar{Fil},\bar \phi):=(H,\nabla,Fil,\Phi).
$$
That the functor $\mathcal{IC}_n$ lifts $\mathcal{IC}_n$ follows from the lifting property of the inverse Cartier transform and the construction of the relative Frobenius.\\

{\itshape Equivalence of categories}: This is done by induction on $n$ and the constructions of the natural transformations in the proof of Proposition \ref{correspondence in the type (0,1) case}.
\end{proof}
At this point, we have completed our theory on the Higgs correspondence in positive and mixed characteristic (\S3-\S5). In \cite{Katz73} (see also \cite[Remark 5.5]{Katz71}), N. Katz established the following result.
\begin{theorem}[Proposition 4.1.1 \cite{Katz73}]\label{Katz result}
	Let $X_n/W_n$ be an irreducible smooth affine scheme over $W_n$, equipped with an absolute Frobenius lifting $F_{X_n}$. Then there is an equivalence of categories between the category of $W_n(\F_{p^f})$-representations of $\pi_1(X_n)$ and the category of pairs $(E,\phi)$ consisting of a locally free sheaf of finite rank $E$ over $X_n$ together with an isomorphism $\phi:F_{X_n}^{*f}(E)\to E$. 
\end{theorem}
The Katz's correspondence has been further developed in the work \cite{EK} of Emerton-Kisin, where the role of unit $F$-crystals is emphasized. A pair $(E,\phi)$ in Theorem \ref{Katz result} is called a \emph{Frobenius-periodic} vector bundle over $X_n$. This is the prototype of the notion of perodic Higgs-de Rham flow. Indeed, to a Frobenius-perodic vector bundle $(E,\phi)$, one associates the perodic Higgs-de Rham flow $(E,0,Fil_{tr},\cdots,Fil_{tr},\phi)$ over $X_n$ of level zero, and this association is an equivalence of categories. Take an arbitrary $W$-lifting $X=\Spec\ R$ of $X_n$ and $\Gamma$ (resp. $\Gamma^{ur}$) the Galois group of the maximal extension of $R$ \'{e}tale over $R[1/p]$ (resp. over $R$) (Ch. II \cite{Fa1}). Then, one may show that the corresponding $W_n(\F_{p^f})$-representation of $\Gamma$ to $(E,0,Fil_{tr},\cdots,Fil_{tr},\phi)$ over $X_n$ after Theorem \ref{periodic corresponds to FF module over a truncated witt ring} and Theorem 2.6 \cite{Fa1} factors through the natural quotient $\Gamma\to \Gamma^{ur}=\pi_1(X_n)$, and the resulting representation of $\pi_1(X_n)$ coincides with the representation corresponding to $(E,\phi)$ after Theorem \ref{Katz result}. Moreover, our theory extends to the case where only the existence of $W_{n+1}$-lifting of $X_n$ is assumed. Namely, one may show that, given a $W_{n+1}$-lifting $X_{n+1}$ of $X_n$, there is an equivalence of categories between the category of $f$-periodic Higgs-de Rham flow over $X_n$ of level zero and the category of $W_n(\F_{p^f})$-representations of $\pi_1(X_n)$. \\

To demonstrate the usage of the theory in a higher Hodge-Tate weight situation, we provide an immediate construction of a $p$-divisible group over a geometric base over $W$ whose Kodaira-Spencer map is an isomorphism.
\begin{example}\label{p-divisible group}
Let $A_1$ be any ordinary abelian variety defined over $k$ of dimension $g$. By Serre-Tate theory, it has the canonical lifting $A$ over $W(k)$ with the Frobenius lifting $F: A\to A$. Consider the following  Higgs bundle $(E,\theta)$ over $A/W$:
$$
E^{1,0}\oplus E^{0,1}=\Omega_{A}\oplus \mathcal{O}_A, \qquad  \theta^{1,0}=Id:\Omega_{A}\to \mathcal{O}_{A}\otimes \Omega_A.
$$
In the following we show that this Higgs bundle is one-periodic. We set $(E_n,\theta_n):= (E,\theta)\otimes \Z/p^n\Z$, and $F_n= F\otimes \Z/p^n\Z$. First we construct a one-periodic Higgs-de Rham flow on $A_1$: since $A_2$ has the global Frobenius lifting $F_2$, it follows that
$$
C_1^{-1}(E_1,\theta_1):=(H_1,\nabla_1) ,
$$
with
$$
 H_1:= F_1^*E_1 \qquad \textrm{and      } \qquad  \nabla_1=\nabla_{can}+ \frac{dF_2}{p}(F_1^*\theta_1).
$$
A Hodge filtration of level one on $(H_1,\nabla_1)$ is defined by  $Fil_1^1=F_1^*\Omega_{A_1}=\Omega_{A_1}$. Set
$$
  (E'_1,\theta'_1):=Gr_{Fil_1}(H_1,\nabla_1).
$$
Then
$$
  E'_1=\Omega_{A_1}\oplus \mathcal{O}_{A_1}, \qquad \theta'^{1,0}_1= \frac{dF_2}{p}(F_1^*\theta_1).
$$
Because of the ordinariness of $A_1$, the Hasse-Witt map $$\frac{dF_2}{p}: H^0(A_1,\Omega_{A_1})\to  H^0(A_1,\Omega_{A_1})$$ is bijective. So $\theta'^{1,0}_1$ has to be an isomorphism. Then we proceed to show $(E'_1,\theta'_1)$ is isomorphic to $(E_1,\theta_1)$. Indeed, there is one natural choice $\psi_1:  E'_1\to E_1$ of isomorphisms described as follows: its $(0,1)$-component mapping $\mathcal{O}_{A_1}$ to itself is the identity, and its $(1,0)$-component mapping $\Omega_{A_1}$ to itself is the unique isomorphism commuting with the Higgs fields. Therefore, we have obtained a one-periodic flow over $A_1$ as claimed:
$$
  \xymatrix{ &(H_1,\nabla_1)\ar[dr]^{Gr_{Fil_1}}
     \\ (E_1,\theta_1) \ar[ur]^{C_{1}^{-1}}
 &&  (E'_1,\theta'_1) \ar@/^2pc/[ll]^{\stackrel{\psi_1}{\cong}  }}
$$
Next we proceed to the $W_2$-level. Using the fact that $A_3$ has the Frobenius lifting $F_3$, one computes that $\tilde{H}_{-1,2}=\Omega_{A_2}\oplus \mathcal{O}_{A_2}$, and $\tilde{\nabla}_{-1,2}=p\nabla_{can}+ \theta_2$. Using Remark \ref{global existence of Frobenius lifting}, it follows that
$$
C_2^{-1}(E_2,\theta_2)= (H_2,\nabla_2)
$$
with $H_2=F_2^*E_2=\Omega_{A_2}\oplus \mathcal{O}_{A_2}$ and $\nabla_2$ is the connection defined by the formula
$$
\nabla_2(f\otimes e)=df\otimes e+f\cdot(\frac{dF_3}{p}\otimes 1)(1\otimes \tilde\nabla_{-1,2}(e)),
$$
where $f$ (resp. $e$) is a local section of $\sO_{A_2}$ (resp. $\tilde{H}_{-1,2}$) (see Claim \ref{def of connection from p-connection}). Now we take the filtration $Fil_2^1=\Omega_{A_2}$, which lifts $Fil_1^1$. Then the associated graded Higgs bundle is
$$
E'_2= \Omega_{A_2}\oplus \mathcal{O}_{A_2}, \ \theta'^{1,0}_2=\frac{dF_3}{p}(F_2^*\theta_2)
$$
which lifts $(E'_1,\theta'_1)$. Again there is an obvious isomorphism
$$
\psi_2: (E'_2,\theta'_2)\to (E_2,\theta_2)
$$
which lifts $\psi_1$ and whose $(0,1)$-component is the identity map. So we obtain a one-periodic flow over $A_2$:
$$
\xymatrix{ &(H_2,\nabla_2)\ar[dr]^{Gr_{Fil_2}}
     \\ (E_2,\theta_2) \ar[ur]^{C_{2}^{-1}}
 &&  (E'_2,\theta'_2)\ar@/^2pc/[ll]^{\stackrel{\psi_2}{\cong}  }}
$$
Then one continues and constructs inductively a one-periodic flow over $A_n, n\geq 1$:
$$
\xymatrix{ &(H_n,\nabla_n)\ar[dr]^{Gr_{Fil_n}}
     \\ (E_n,\theta_n) \ar[ur]^{C_{n}^{-1}}
 &&  (E'_n,\theta'_n)\ar@/^2pc/[ll]^{\stackrel{\psi_n}{\cong}  }}
$$
Passing to the limit, one obtains therefore a one-periodic Higgs-de Rham flow over $A/W$, and hence a rank $g+1$ crystalline $\Z_p$-representation of Hodge-Tate weight one of the generic fiber $A^0$ of $A/W$ which by \cite[Theorem 7.1]{Fa1} corresponds to a $p$-divisible group over $A/W$.
\end{example}

\section{Strongly semistable Higgs modules}
 Let $X/k$ be a smooth projective variety over $k$, equipped with an ample divisor $Z$. The semistability in this section means the $\mu_Z$-slope semistability. Recall that a vector bundle over $X$ is said to be \emph{strongly semistable} if the bundle as well as its pullback under any power of Frobenius are semistable. The relation between strongly semistable bundles with trivial Chern classes and representations of the algebraic fundamental group was firstly revealed by Lange-Stuhler in the curve case (see \cite[\S1]{LS}). It asserts that a semistable vector bundle $E$ of degree zero over a curve can be trivialized after a finite morphism if and only if it is strongly semistable (\cite[Satz 1.9]{LS}). Note that $E$ being strongly semistable of degree zero is equivalent to the condition that there is a pair $(e,f)$ of integers with nonnegative $e$ and positive $f$ such that
$$
F_{X}^{*e+f}(E)\cong F_{X}^{*e}(E),
$$
where $F_{X}:X\to X$ denotes the absolute Frobenius morphism as usual. It corresponds to a representation of $\pi_1(X)$ into $\GL(k)$ if and only if $e$ in the above isomorphism can be taken to be zero (\cite[Proposition 1.2, Satz 1.4]{LS}). The result of Lange-Stuhler has been generalized to a singular curve by Deninger-Werner (see \cite[Theorem 18]{DW}) and this generalization played a key role in their partial $p$-adic analogue of the Narasimhan-Seshadri theory. Besides the intimate relation with the representation of $\pi_1$, the notion of strongly semistability is also useful in other situation, for example, in Langer's proof of boundedness of semistable sheaves and Bogomolov's inequality in positive characteristic \cite{Langer0}. Therefore, it is a natural question to generalize this notion to Higgs modules. Interestingly enough, it turns out that our generalization (especially Theorem \ref{semistable bundles of small ranks are strongly semistable} below) has played a key role in the very recent result, due to A. Langer \cite{Langer2}, on the Bogomolov-Gieseker inequality for semistable Higgs bundles and Miyaoka-Yau inequality for surfaces in positive characteristic. \\

The key of the generalization is to replace the Frobenius pullback with the inverse Cartier transform of Ogus-Vologodsky \cite{OV}, as seen in the following
\begin{definition}
Let $X$ be a smooth projective variety over $k$ together with a fixed $W_2$-lifting of $X$. A Higgs module $(E,\theta)$ is called strongly semistable if
it appears in the initial term of a semistable Higgs-de Rham flow, that is, all
Higgs terms $(E_i,\theta_i)$s in the flow are semistable and defined over a common finite subfield of $k$.\footnote{This definition corrects the error in \cite[Definition 2.1]{LSZ2012}.}
\end{definition}
A torsion Higgs module is by definition automatically semistable, which is however uninteresting to study in the current setting. Therefore, a semistable Higgs module is taciturnly assumed to be torsion free in this paper. Clearly, a strongly semistable vector bundle $E$ is strongly Higgs semistable: one takes simply the following Higgs-de Rham flow

\begin{adjustbox}{scale=1.1}
$$
\xymatrix{
                &  (F_{X}^*E,\nabla_{can})\ar[dr]^{Gr_{Fil_{tr}}}       &&  (F_{X}^{*2}E,\nabla_{can})\ar[dr]^{Gr_{Fil_{tr}}}    \\
 (E,0) \ar[ur]^{F_X^*}  & &     (F_{X}^*E,0) \ar[ur]^{F_X^*}&&\ldots       }
$$
\end{adjustbox}

where $\nabla_{can}$ is the canonical connection in the theorem of Cartier descent and $Fil_{tr}$ stands for the trivial filtration as before. \\

The first result on the Higgs semistability of the graded Higgs module associated to a strict $p$-torsion Fontaine module is due to Ogus-Vologodsky \cite[Proposition 4.19]{OV}.
\begin{proposition}[Ogus-Vologodsky]
Let $X/k$ be a smooth projective curve of genus $g$. Let $(H,\nabla,Fil,\Phi)$ be a strict $p$-torsion Fontaine module (with respect to some $W_2$-lifting of $X$). Assume that
$$
n(\rk H-1)\max\{2g-2,1\}<p-1.
$$
Then $Gr_{Fil}(H,\nabla)$ is a semistable Higgs bundle.
\end{proposition}

Their result can be generalized as follows (see \cite[Proposition 0.2]{SXZ} for a generalization in the geometric case as given in Example \ref{geometric situation} and see also \cite[Proposition 3.7]{SZPeriodic}):
\begin{proposition}\label{geometric case}
Let $X/k$ be a smooth projective variety. Let $(H,\nabla,Fil,\Phi)$ be a strict $p$-torsion Fontaine module (with respect to some $W_2$-lifting of $X$). Then the graded Higgs bundle $(E,\theta):=Gr_{Fil}(H,\nabla)$ is Higgs semistable. Moreover, any Higgs subsheaf $(G,\theta)\subset (E,\theta)$ of slope zero is strongly semistable.
\end{proposition}
\begin{proof}
Let us first recall that we have proven that there is a natural isomorphism
$$
\tilde \Phi: C_1^{-1}(E,\theta)\cong (H,\nabla).
$$
See \cite[Proposition 1.4]{LSZ}. Therefore, one obtains a natural isomorphism $Gr_{Fil}\circ C_1^{-1}(E,\theta)\cong (E,\theta)$. It implies that $p\mu(E)=\mu(E)$, and thus $\mu(E)=0$. For the first statement, we prove by contradiction. Let $(F,\theta)$ be the maximal destabilizing Higgs subsheaf of $(E,\theta)$ with positive slope $\mu(F)$. Then $Gr_{Fil}\circ C_1^{-1}(F,\theta)$ is naturally a Higgs subsheaf of slope $p\mu(F)>\mu(F)$. A contradiction. For the second statement, let us first notice that a Higgs subsheaf $(G,\theta)$ of slope zero is automatically semistable, as firstly observed by C. Seshadri. Now that $Gr_{Fil}\circ C_1^{-1}(G,\theta)$ is naturally a Higgs subsheaf of $(E,\theta)$ which is again of slope zero and hence semistable, $(G,\theta)$ is strongly semistable.
\end{proof}
It is surprising to have the following
\begin{proposition}\label{rank two semistable implies strongly
semistable} Notation as above. Any rank two nilpotent semistable Higgs module is strongly semistable.
\end{proposition}
\begin{proof}
Let $(E,\theta)$ be a rank two nilpotent semistable Higgs module over $X$.
For the reason of rank, $\theta$ is nilpotent of exponent $\leq 1$. Denote $(H,\nabla)$ for $C_1^{-1}(E,\theta)$, and $HN$ the Harder-Narasimhan filtration on $H$. We need to show that the graded Higgs module
$Gr_{HN}(H,\nabla)$ is again semistable. If $H$ is semistable, there is
nothing to prove: in this case, the $HN$ is trivial and hence the
induced Higgs field is zero, and $Gr_{HN}(H,\nabla)=(H,0)$ is
Higgs semistable. Denote otherwise $L_1\subset H$ for the invertible subsheaf of maximal slope and $L_2=H/L_1$ the quotient sheaf. Then $L_1$ cannot be $\nabla$-invariant. Indeed, $(H,\nabla)$ is $\nabla$-semistable as a general property: say $M\subset H$ any $\nabla$-invariant subsheaf. Then $(F,\theta|_{F}):=C_1(M,\nabla|_{M})\subset C_1(H,\nabla)=(E,\theta)$ is a Higgs subsheaf of slope $\mu(F)=\mu(M)/p$, where $C_1$ is the Cartier transform of Ogus-Vologodsky (see also \cite{LSZ}). As $(E,\theta)$ is Higgs semistable, it follows that $$\mu(M)=p\mu(F)\leq p\mu(E)=\mu(H).$$  So the natural map
$$
\theta'=Gr_{HN}\nabla: L_1\to L_2\otimes \Omega_{X_1}
$$
is nonzero. A nontrivial proper Higgs subsheaf $L\subset Gr_{HN}(H,\nabla)$ is simply an invertible subsheaf with $\theta'(L)=0$.  So $L\subset L_2$ and $$\mu(L)\leq \mu(L_2)<\mu(H)=\mu(Gr_{HN}H).$$ In this case, $Gr_{HN}(H,\nabla)$ is actually Higgs stable.
\end{proof}
Motivated by Proposition \ref{rank two semistable implies strongly semistable}, we proposed a conjecture on strongly semistability of a nilpotent semistable Higgs module of higher rank in the first version of the paper (see \cite[Conjecture 2.8]{LSZ2012}). The conjecture was proven by A. Langer in \cite{Langer1} in the case of small rank, which is quite crucial to his algebraic proof of Bogomolov-Gieseker inequality and Miyaoka-Yau inequality.  In Appendix A, we shall provide an independent proof of the conjecture in the case of small rank.
\begin{theorem}[Theorem \ref{strsst}, {\cite[Theorem 5.1]{Langer1}}]\label{semistable bundles of small ranks are strongly semistable}
Notation as above. Any nilpotent semistable Higgs module of rank $\leq p$ over $X$ is strongly semistable.
\end{theorem}
After establishing the notion of strongly semistability and exhibiting ample examples, we shall concentrate on those strongly semistable Higgs modules with trivial Chern classes, as guided by the theorem of Lange-Stuhler in the beginning. It turns out they are quite close to be periodic.
\begin{theorem}\label{quasiperiodic equivalent to strongly semistable} 
Let $X/k$ be a smooth projective variety together with a fixed $W_2$-lifting of $X$. A preperiodic Higgs module is strongly semistable with trivial   Chern classes. Conversely, a strongly semistable Higgs module with trivial Chern classes is preperiodic.
\end{theorem}
\begin{proof}
For a Higgs module $(E,\theta)\in HIG_{p-1}(X)$, let $(H,\nabla)=C_1^{-1}(E,\theta)$ be the corresponding flat module. It follows from the proof of {\cite[Theorem 4.17]{OV}} that
$$
c_l(H)=p^lc_l(E), l\geq 0.
$$
Since for any Griffiths transverse filtration $Fil$ on $(H,\nabla)$ the associated graded Higgs module $(E',\theta')=Gr_{Fil}\circ C_1^{-1}(E,\theta)$ has the same Chern classes as $H$, it follows that
$$
c_l(E')=p^lc_l(E), l\geq 0.
$$
Therefore, in a Higgs-de Rham flow, one has for $i\geq 0$
$$
c_l(E_{i+1})=p^lc_l(E_{i}), l\geq 0.
$$
This forces the Chern classes of a preperiodic Higgs module to be trivial. Also, a slope $\lambda$ Higgs subsheaf in $(E_i,\theta_i)$ gives rise to a
slope $p\lambda$ Higgs subsheaf in $(E_{i+1},\theta_{i+1})$. This
implies that, in a preperiodic Higgs-de Rham flow, each Higgs term $(E_i,\theta_i)$ contains no Higgs subsheaf of positive degree. So each $(E_i,\theta_i)$ is semistable. This shows the first statement. Conversely, let $(E,\theta)$ be a strongly semistable Higgs module with trivial Chern classes, and let $(E_i,\theta_i)_{i\geq 0}$ be the Higgs terms appearing in a semistable Higgs-de Rham flow with the initial term $(E,\theta)$. As discussed above, each $(E_i,\theta_i)$ has trivial Chern classes (and the same rank as $E$). By \cite[Lemma 5]{Langer2} and \cite[Theorem 4.4]{Langer1}, the moduli space of semistable Higgs modules with trivial Chern classes over $X/k$ is bounded and defined over $k$. Let $k'\subset k$ be the common finite subfield for the infinite sequence $\{(E_i,\theta_i)\}_{i\geq 0}$. Since any scheme of finite type over $k$ has only finitely many $k'$-rational points, there must exist a pair of integers $(e,f)\in \Z_{\geq 0}\times \N$ such that there is an isomorphism (over $k$) of Higgs modules $(E_{e+f},\theta_{e+f})\cong (E_e,\theta_e)$. Certainly, one can make the above isomorphism into an isomorphism of graded Higgs modules by adjusting one of their gradings. Therefore, $(E,\theta)$ is preperiodic.
\end{proof}

Using the notion of strongly semistable Higgs module as bridge, we can produce crystalline $k$-representations (up to isomorphism) from semistable nilpotent Higgs bundles (of small rank) with trivial Chern classes.
\begin{theorem}\label{rank two semistable bundle corresponds to rep}
Let $X/W$ be a smooth projective scheme. Then for any rank $r\leq p-1$ semistable nilpotent Higgs bundle $(E,\theta)$ over $X_k$ with trivial Chern classes, one associates a unique $r$-dimensional crystalline $k$-representation of $\pi_1(X_{K})$ up to isomorphism. 
\end{theorem}
The proof relies on the next result which follows directly from Theorem \ref{semistable bundles of small ranks are strongly semistable} and Theorem \ref{quasiperiodic 
equivalent to strongly semistable}. 
\begin{corollary}\label{existence of flow}
Notation as Theorem \ref{rank two semistable bundle corresponds to rep}. Then any rank $r\leq p$ semistable nilpotent Higgs bundle over $X_k$ with trivial Chern classes is preperiodic.
\end{corollary} 
It is a nontrivial problem to decide when a small rank semistable graded Higgs bundle with trivial Chern classes is periodic (and its period when it is indeed the case), albeit always preperiodic by the corollary. We shall discuss this problem as well as the basic properties of representations constructed in Theorem \ref{rank two semistable bundle corresponds to rep} on a later occasion. \\ 

More preparations are needed before we can prove Theorem \ref{rank two semistable bundle corresponds to rep}. A (pre)peridoic Higgs bundle may lead more than one (pre)periodic Higgs-de Rham flows. However, in the setting of Corollary \ref{existence of flow}, there is a natural choice of filtrations in a preperiodic Higgs-de Rham flow. 
\begin{proposition}[Theorem \ref{goodfil}, Remark \ref{canonical};{  \cite[Theorem 5.5]{Langer1}}]\label{Simpson filtration}
Let $(H,\nabla)$ be a $\nabla$-semistable flat bundle over $X_k$. Then there exists a uniquely defined reduced gr-semistable filtration on $(H,\nabla)$ which is preserved by any automorphism of $(H,\nabla)$.
\end{proposition}
The reader is referred to the appendix for the definition of a gr-semistable (resp. reduced) filtration. We shall call the filtration in the above result the \emph{Simpson filtration} and denote it by $Fil_S$. Note that the Simpson filtration in the rank two case is nothing but the Harder-Narasimhan filtration on $H$, and it may well differ from the Harder-Narasimhan filtration for higher ranks. A periodic Higgs-de Rham flow whose each de Rham term is equipped with the Simpson filtration is unique up to lengthening, as proven below. The unicity is shared more generally by a periodic Higgs-de Rham flow with the following  
\begin{assumption}\label{assumption on filtration}
Let $(E,\theta,Fil_0,\cdots,Fil_{f-1},\varphi)$ be a periodic Higgs-de Rham flow over $X_k$. Assume that, for each $0\leq  i\leq f-1$, the filtration $Fil_i$ on $H_i$ is
preserved by any automorphism of $(H_i,\nabla_i)$.
\end{assumption}
Recall in the definition of the lengthening, one has the following isomorphisms of graded Higgs modules induced by $\varphi:(E_f,\theta_f)\cong (E_0,\theta_0)$:
$$
(GrC_1^{-1})^{nf}(\varphi):(E_{(n+1)f},\theta_{(n+1)f})\cong
(E_{nf},\theta_{nf}), \quad n\in \N.
$$ 
Set, for $-1\leq i< j$, 
$$ 
\varphi_{j,i}=(GrC_1^{-1})^{(i+1)f}(\varphi)\circ\cdots\circ(GrC_1^{-1})^{jf}(\varphi):(E_{(j+1)f},\theta_{(j+1)f})\cong (E_{(i+1)f},\theta_{(i+1)f}).
$$
Note $\varphi_{j,-1}$ is just the isomorphism $\varphi_j$ in the $j$-th lengthening of the starting periodic flow.  
\begin{lemma}\label{finiteness implies periodic}
Let $\phi: (E_f,\theta_f)\cong(E_0,\theta_0)$ be another isomorphism of graded Higgs modules. Then there exists a pair $(i,j)$ with
$0\leq i<j$ such that $\phi_{j,i}\circ \varphi_{j,i}^{-1}=Id$.
\end{lemma}
\begin{proof}
If we denote $\tau_s=\phi_{s}\circ \varphi_{s}^{-1}$, then $\tau_s$
is an automorphism of $(E_0,\theta_0)$. Moreover, each element in
the set $\{\tau_s\}_{s\in \N}$ is defined over the same finite field
in $k$. As this is a finite set, there are $j>i\geq 0$ such that
$\tau_j=\tau_i$. So the lemma follows.
\end{proof}
\begin{proposition}\label{phi plays no role}
Assume Assumption \ref{assumption on filtration}. Let $(i,j)$ be a pair given
by Lemma \ref{finiteness implies periodic} for two given
isomorphisms $\varphi, \phi : (E_f,\theta_f)\cong(E_0,\theta_0)$.
Then there is an isomorphism in $HDF_{n,(j-i)f}(X_2/W_2)$:
$$(E,\theta,Fil_0,\cdots,Fil_{(j-i)f-1},\varphi_{j-i-1})\cong (E,\theta,Fil_0,\cdots,Fil_{(j-i)f-1},\phi_{j-i-1}),$$
where both sides of the isomorphism are obtained by $j-i-1$-th lengthening.
\end{proposition}
\begin{proof}
Put $\beta=\phi_{i}\circ \varphi_{i}^{-1}: (E_0,\theta_0)\cong
(E_0,\theta_0)$. We shall check that it induces an isomorphism in
$HDF_{n,(j-i)f}(X_2/W_2)$. By Assumption \ref{assumption on
filtration}, $C_1^{-1}(GrC_1^{-1})^m(\beta)$ for $m\geq 0$ always
respects the filtrations. We need only to check that $\beta$ is
compatible with $\phi_{j-i-1}$ as well as $\varphi_{j-i-1}$. So it
suffices to show that the following diagram is commutative:
\begin{diagram}
E_{(j-i)f}&\rTo{\varphi_{j-i-1}}&E_0\\
\dTo{\varphi_{j,j-i-1}^{-1}}& & \dTo{\varphi_{i}^{-1}}\\
E_{(j+1)f} & &  E_{(i+1)f}\\
\dTo{\phi_{j,j-i-1}} & &\dTo{\phi_{i}} \\
E_{(j-i)f}& \rTo{\phi_{j-i-1}}&E_0.\\
\end{diagram}
Clearly, it is equivalent to show the commutativity of the following diagram:
\begin{diagram}
E_{(j-i)f}&\lTo{\varphi_{j-i-1}^{-1}}&E_0\\
\dTo{\varphi_{j,j-i-1}^{-1}}& & \dTo{\varphi_{i}^{-1}}\\
E_{(j+1)f} & &  E_{(i+1)f}\\
\dTo{\phi_{j,j-i-1}} & &\dTo{\phi_{i}} \\
E_{(j-i)f}& \rTo{\phi_{j-i-1}}&E_0.\\
\end{diagram}
In the above diagram, the anti-clockwise direction is
$$\phi_{j-i-1}\circ\phi_{j,j-i-1}\circ\varphi_{j,j-i-1}^{-1}\circ\varphi_{j-i-1}^{-1}
=\phi_j\circ\varphi_j^{-1}=\phi_i\circ(\phi_{j,i}\circ\varphi_{j,i}^{-1})\circ\varphi_i.$$
By the requirement for $(i,j)$, we have
$\phi_{j,i}\circ\varphi_{j,i}^{-1}=Id$, so the anti-clockwise
direction is $\phi_i\circ\varphi_i$, which is exactly the clockwise
direction. So $\beta$ is shown to be compatible with $\phi_{j-i-1}$
and $\varphi_{j-i-1}$.
\end{proof}
Now we proceed to the proof of Theorem \ref{rank two semistable bundle corresponds to rep}.
\begin{proof}
It is to assemble the previous results. First, by Theorem \ref{existence of flow}, there exists a preperiodic flow with the initial term $(E,\theta)$. But there are several choices. In order to make a unique choice, we shall apply Proposition \ref{Simpson filtration} at each step, that is, we use the Simpson filtration $Fil_S$ on each flat bundle and then we obtain the uniqueness on the filtrations in the flow. Now let $e\in \N_0$ be the minimal number such that $(Gr_{Fil_S}\circ C_1^{-1})^e(E,\theta)$ is periodic and let $f\in \N$ be its period. Thus from $(E,\theta)$ we have obtained an
periodic Higgs-de Rham flow
$$
((Gr_{Fil_S}\circ C_1^{-1})^e(E,\theta),Fil_0=Fil_S,\cdots,Fil_{f-1}=Fil_S,\phi),
$$
unique up to the choice of $\phi$. Fix one choice of $\phi$ and let $\rho$ be the corresponding representation after Corollary \ref{correspondence from crystalline represenations and HB_(0,f)}. As $Fil_S$s satisfy Assumption \ref{assumption on filtration}, it follows from Proposition \ref{phi plays no role} and Corollary \ref{operations on Higgs-de Rham sequences} (ii) that the isomorphism class of $\rho\otimes k$ is independent of the choice of $\phi$. Finally, the rank condition $\leq p-1$ allows us to apply Theorem \ref{Faltings thm}, to conclude the theorem.
\end{proof}

\section{Rigidity theorem for Fontaine modules}
Let $X/W$ be a smooth and projective scheme, equipped with a $W$-ample divisor $Z$. To a ($p$-torsion) Fontaine module $(H,\nabla,Fil,\Phi)$, one associates naturally the graded Higgs bundle $(E,\theta):=Gr_{Fil}(H,\nabla)$ by taking the grading of $(H,\nabla)$ with respect to the filtration $Fil$. In this section, we show that the gr-functor is faithful over those mod $p$-stable objects (with respect to the $\mu_{Z_1}$-slope). \\ 

For a Griffiths transverse filtration $Fil$ of level $w$, we extend $Fil^{i}=Fil^0$ for $-i\in \N$ and $Fil^{w+j}=0$ for $j\in \N$.   
\begin{lemma}\label{unique fil}
Let $Y$ be a smooth projective variety over an algebraically closed field $k$ and let $(V,\nabla)$ be a flat bundle over $Y$. If there exists a Griffiths transverse filtration $Fil$ on $(V,\nabla)$ such that the associated graded Higgs module $(E,\theta)$ is stable, then it is unique up to a shift of index.
\end{lemma}
\begin{proof}
Suppose on the contrary that there is another gr-semistable filtration $\bar{Fil}$ on $V$ which differs from $Fil$ after arbitrary index shifting. Set $(\bar E,\bar \theta):=Gr_{\bar{Fil}}(V,\nabla)$.\\

Case 1: Suppose that there exists an integer $N$, such that for every $i$, $Fil^{i}\subset \bar{Fil}^{i+N}$. Also, we can assume $N$ is so chosen that for some $i_0$, $Fil^{i_0}\subsetneq \bar{Fil}^{i_0+N+1}$. Then the inclusion induces a natural  morphism of Higgs bundles $$f:(E,\theta)\to (\bar{E},\bar{\theta}).$$ Clearly, the morphism $f$ cannot be injective at each closed point. Otherwise, it implies that $Fil^{i}=\bar{Fil}^{i+N}$ for all $i$ which contradicts the assumption that the two previous filtrations are nonequal. So $f$ is not injective. Neither is $f$ zero, since, otherwise, it would imply that $Fil^i\subset \bar{Fil}^{i+N+1}$ for all $i$ which contradicts the assumption on $N$. Therefore, on the one hand, as ${\rm Im}(f)$ is a quotient of $(E,\theta)$, $\mu({\rm Im}(f))>\mu(E)=\mu(V)$; on the other hand, ${\rm Im}(f)$ is also a subobject of $(\bar{E},\bar{\theta})$, $\mu({\rm Im}(f))\leq \mu(\bar E)=\mu(V)$. A contradiction.\\

Case  2: Otherwise, let $a$ be the largest integer such that $Fil^a$ is not contained in $\bar{Fil}^a$, and $b$ the largest integer such that $Fil^{a-i}$ is contained in $\bar{Fil}^{b-i}$ for  all $i\geq 0$. Certainly $a>b$.  Then we define a morphism of Higgs bundles $f:(E,\theta)\to (\bar{E},\bar {\theta})$ as follows: for $i>0$, $f|_{E^{a+i}}=0$, and for $j\geq 0$, $f|_{E^{a-j}}$ is given by
$$
 E^{a-j}=Fil^{a-j}/Fil^{a-j+1}\to \bar{Fil}^{b-j}/\bar{Fil}^{b-j+1}.
$$
This is well defined as $Fil^{a+1}\subset \bar{Fil}^{a+1}\subset \bar{Fil}^{b+1}$ and $Fil^{a-j}\subset \bar{Fil}^{b-j}$ for $j\geq 0$. Clearly $f$ cannot be injective. Also, $f$ cannot be zero. Otherwise, one would get the relation $Fil^{a-j}\subset \bar{Fil}^{b+1-j}$ for all $j\geq 0$, which contradicts the maximality of $b$. The remaining argument goes exactly as Case 1.
\end{proof}
The above statement, true for any characteristic, has the following nice consequence in the current setting.
\begin{proposition}\label{unique filtration}
Let $(E,\theta)$ be a Higgs bundle over $X_1$. Suppose $(E,\theta)$ is stable and one-periodic. Then there is a unique one-periodic Higgs-de Rham flow with the initial term $(E,\theta)$ up to isomorphism.
\end{proposition}
\begin{proof}
Let $(E,\theta, Fil,\phi)$ be a one-periodic Higgs-de Rham flow. First, because of the stability assumption on $(E,\theta)$,  it follows from Lemma \ref{unique fil} that the Hodge filtration $Fil$ on $C_1^{-1}(E,\theta)$ is unique up to a possible shift of index. However, $\phi$ is an isomorphism of \emph{graded} Higgs modules, $Fil$ has to be unique.  Second, for any other choice $\varphi$ making $(E,\theta,Fil,\varphi)$ periodic,  the composite $\varphi\circ \phi^{-1}$ is an automorphism of
$(E,\theta)$. Since it is stable, one must have $\varphi=\lambda\phi$
for a nonzero constant $\lambda\in k$. And there is an obvious isomorphism in the category $HDF$:
$$(E,\theta,Fil,\phi)\cong (E,\theta,Fil,\lambda\phi),$$ the proposition follows.
\end{proof}
An $\F_{p}$-representation $\rho$ of $\pi_1(X_K)$ is said to be absolutely irreducible if the $k$-representation $\rho\otimes k$ is irreducible. A direct consequence of the previous proposition is the following
\begin{corollary}\label{stable corresponds to irreducible}
Notation as above. There is a natural equivalence of categories between the category of absolutely irreducible crystalline $\F_{p}$-representations of $\pi_1(X_{K})$ with Hodge-Tate weights $\leq p-2$ and the category of one-periodic stable Higgs bundles in $HIG_{p-2}(X_1)$.
\end{corollary}
\begin{proof}
Proposition \ref{unique filtration} embeds the category of one-periodic stable Higgs bundles over $X_1$ into the category of one-periodic Higgs-de Rham flows over $X_1$ as a full subcategory. Under the Higgs correspondence and Fontaine-Laffaille-Faltings correspondence (Proposition \ref{correspondence in the type (0,1) case} and Theorem \ref{Faltings thm}), the category of one-periodic stable Higgs bundles in $HIG_{p-2}(X_1)$ corresponds to a full subcategory of crystalline $\F_{p}$-representations of $\pi_1(X_K)$. The image is characterized by the absolutely irreducibility of the associated $\F_p$-representations, as shown by \cite[Theorem 1.3]{SZPeriodic}.
\end{proof}
\begin{remark}\label{reconstruction}
Over $\C$, the gr-functor is an equivalence of categories between the category of irreducible complex polarized variations of Hodge structure and the category of stable graded Higgs bundles. The above corollary is a characteristic $p$-analogue of this equivalence. However, compared with the transcendental nature of the quasi-inverse functor in the complex case, the one in the characteristic $p$ case is much more constructive: indeed, it has already been noticed by Ogus-Vologodsky (see \cite[Definition 4.16]{OV} and the paragraph thereafter) that the flat bundle $(H,\nabla)$ of a strict $p$-torsion Fontaine module is reconstructed by $C_1^{-1}(E,\theta)$, no matter whether $(E,\theta)$ stable or not (but it is at least semistable as explained in Proposition \ref{geometric case}). This point has also been explicitly emphasized in \S4 \cite{LSZ}. Furthermore, in Theorem \ref{goodfil} Appendix A, we show how to construct a gr-semistable filtration on $(H,\nabla)$. By the proof of Lemma \ref{unique fil}, it has to coincide with the Hodge filtration $Fil$ by a possible shift of index, which is also uniquely determined by the gradings in $E$. The construction of the relative Frobenius from an isomorphism $Gr_{Fil}(H,\nabla)\cong (E,\theta)$ is the major content of the functor $\sC^{-1}$ in the proof of Proposition \ref{correspondence in the type (0,1) case}. 
\end{remark}

The above results may be lifted to $W_n$ for $n$ arbitrary.
\begin{proposition}\label{uniqueness for n}
Let $(E,\theta)$ be the initial term of a one-periodic Higgs-de Rham flow over $X_n$. If the mod $p$ reduction of $(E,\theta)$ is stable over $X_1$, then there is a unique one-periodic Higgs-de Rham flow for $(E,\theta)$ up to isomorphism.
\end{proposition}
 \begin{proof}
We prove by induction on $n$. The $n=1$ case is Proposition \ref{unique filtration}. We show make the induction hypothesis as follows: let $(\bar E, \bar \theta)$ be the mod $p^{n-1}$-reduction of $(E,\theta)$, and $(\bar E,\bar \theta, \bar{Fil},\bar\psi)$ be a one-periodic Higgs-de Rham flow over $X_{n-1}$. Then $\bar{Fil}$ is the unique lifting of the Hodge filtration in the characteristic $p$ and $\bar \psi$ is the unique lifting of the isomorphism in characteristic $p$ up to a scalar in $W_{n-1}$. Now let $(E,\theta,Fil_i,\psi_i), i=1,2$ be two one-periodic Higgs-de Rham flow over $X_n$ lifting $(\bar E,\bar \theta,\bar{Fil},\bar \psi)$ over $X_{n-1}$. It suffices to show the following claims:
 \begin{itemize}
 \item [(i)] $Fil_1=Fil_2$;
 \item[(ii)] $\psi_1=\lambda\psi_2, \ \lambda\in W_{n}$.
 \end{itemize}
Without loss of generality, we may assume that $\bar{Fil}$ is reduced. Assume the contrary of the statement (i). Let $a$ be the largest integer such that $Fil_1^a$ differs from $Fil^a_2$, and then $b$ be the largest integer such that $Fil_1^{a-i} \subseteq Fil_2^{b-i}$ for each $i\geq 0$. \\

Case 1. $b>a$. Then as $$Fil^{a}_1\subset Fil^{b}_2\subset Fil_2^{a+1}=Fil_1^{a+1}$$ and $Fil_1^{a+1}\subset Fil_1^{a}$, we get $Fil^{a+1}_1=Fil^{a}_1$, and in particular, $\bar{Fil}^{a+1}=\bar{Fil}^a$. A contradiction. \\

Case 2. $b=a$. As $Fil_1^{a}\subset Fil_2^{a}$ and $Fil_1^{a}=Fil_2^{a}\mod p^{n-1}$, it follows that $Fil_1^{a}=Fil_2^{a}$. A contradiction.\\

Case 3. $b<a$. Let $(H,\nabla)=C_n^{-1}(E,\theta)$. We define a morphism
$$
f: Gr_{Fil_1}(H,\nabla)\to Gr_{Fil_2}(H,\nabla)
$$
as follows: for $i>0$, $f$ on the factor $Fil^{a+i}_1/Fil_1^{a+i+1}$ is simply zero; for $i\leq 0$, $f: Fil^{a+i}_1/Fil^{a+i+1}_1\to Fil^{b+i}_2/Fil^{b+i+1}_2$ is the natural morphism. It is easy to check that this gives a morphism of Higgs bundles. Because of the choice of $b$, $f$ is nonzero. As its mod $p^{n-1}$ reduction is clearly the zero map, we obtain a nonzero morphism $\frac{f}{[p^{n-1}]}$ on the mod $p$-reductions of both sides of the morphism $f$, which are isomorphic to the stable Higgs bundle $(E,\theta)_1$ in characteristic $p$. Clearly, it is neither an isomorphism. A contradiction. Therefore, $Fil_1=Fil_2$, and $Gr_{Fil_1}(H,\nabla)=Gr_{Fil_2}(H,\nabla)$. We continue to show the statement (ii). For this, we consider the composite
$$
\phi:=\psi_1\circ\psi_2^{-1}: (E,\theta)\cong (E,\theta).
$$
By the induction hypothesis, $\phi\mod p^{n-1}=\bar \lambda \in W_{n-1}$. Take any lifting $\lambda\in W_{n}$ of $\bar \lambda$ and consider the endomorphism $\phi\rq{}:=\phi-\lambda$ of $(E,\eta)$. As $\phi\rq{}\mod p^{n-1}=0$, we get an endomorphism of the stable Higgs bundle in characteristic $p$:
$$
\frac{\phi\rq{}}{[p^{n-1}]}: (E,\theta)_1\to (E,\theta)_1,
$$
which has to be a scalar $\mu\in k$. So we get $\phi=\lambda+p^{n-1}\mu\in W_{n}$, and the statement (ii) follows.
\end{proof}
The following rigidity theorem for Fontaine modules follows immediately from the previous proposition and the Higgs correspondence.
\begin{corollary}\label{fully faithful}
Let $(H_i,\nabla_i,Fil_i,\Phi_i), i=1,2$ be two Fontaine modules (torsion free or not) over $X/W$, and $(E_i,\theta_i)$ the associated graded Higgs bundles. If $(E_1,\theta_1)$ is isomorphic to $(E_2,\theta_2)$ and mod $p$ Higgs stable, then 
$(H_i,\nabla_i,Fil_i,\Phi_i), i=1,2$ are isomorphic.
 \end{corollary}
Combining Proposition \ref{uniqueness for n} with Theorem \ref{periodic corresponds to FF module over a truncated witt ring} and Corollary \ref{stable corresponds to irreducible}, we obtain the following
\begin{corollary}\label{irreducible over W}
There is a natural equivalence of categories between the category of crystalline $\Z_{p}$ (resp. $W_n(\F_p)$) representations of $\pi_1(X_K)$ with Hodge-Tate weight $\leq p-2$ whose mod $p$ reduction is absolutely irreducible and the category of one-periodic Higgs bundles over $X/W$ (resp. $X_n/W_n$) whose exponent is $\leq p-2$ and mod $p$ reduction is stable.
 \end{corollary}

\appendix
\section{Semistable Higgs bundles of small ranks are strongly Higgs semistable}
\begin{center}
Guitang Lan, Mao Sheng, Yanhong Yang and Kang Zuo
\end{center}

In this appendix, we shall prove the following result.

\begin{theorem}\label{strsst}
Let $(E,\theta)$ be a nilpotent semistable Higgs module over a smooth projective variety $X$ defined over $\bar{\F}_p$. If $\textrm{rank} \ E\leq p$, then $(E,\theta)$ is strongly semistable.
\end{theorem}
This same result has also been obtained by A. Langer independently. See \cite[Theorem 5.12]{Langer1}. In the early version of the manuscript \cite{LSZ2012}, we have proposed the following conjecture which inspired the above theorem.
\begin{conjecture}\label{conjecture}  \cite[Conjecture 2.8]{LSZ2012}
A nilpotent semistable Higgs module of exponent $\leq p$ is strongly Higgs semistable.
\end{conjecture}
We had shown the rank two case in \cite{LSZ2012}. See \cite[Theorem 2.6]{LSZ2012} which is just Proposition \ref{rank two semistable implies strongly semistable} in the current version. Shortly after the appearance of \cite{LSZ2012}, Lingguang Li \cite{Li} has verified the conjecture in the rank three case. The conjecture requires modification in order to be true in the higher rank case.\\

The key step in the proof is Theorem \ref{goodfil}, a positive characteristic generalization of Simpson's result \cite[Theorem 2.5]{Simpson}, which states that over a complex smooth projective curve, every vector bundle with an integrable holomorphic connection admits a Griffiths transverse filtration, such that the associated-graded Higgs bundle is semistable. This generalization is proved similarly as  in \cite[Theorem 2.5]{Simpson}.\\

Throughout the appendix, we assume that $Y$ is a smooth projective variety over an algebraically closed field $k$, $H$ is an ample divisor of $Y$, Higgs modules as well as flat modules are torsion free, and the semistability is referred to the $\mu=\mu_H$-semistability.
\begin{definition}
A flat module $(V, \nabla)$ is called $\nabla$-semistable if
for every submodule $V_1\subset V$ with
$\nabla(V_1)\subset V_1\otimes_{\sO_Y}\Omega_Y$, $\mu(V_1)\leq \mu(V)$ holds.
\end{definition}
The goal of this section is to prove the following theorem.
\begin{theorem}\label{goodfil}
Let $(V, \nabla)$ be a $\nabla$-semistable flat module over a smooth projective variety $Y$ over an algebraically closed field $k$. Then there exists a Griffiths transverse filtration $Fil$ such that the associated-graded Higgs module to $(V, \nabla, Fil)$ is semistable.
\end{theorem}
A Griffiths transverse filtration on $(H,\nabla)$ with semistable graded Higgs module is said to be \emph{gr-semistable}. 
\begin{remark}
By \cite{Weil}, every holomorphic vector bundle  that  admits a connection  is of degree $0$,  thus in this case every flat bundle $(V,\nabla)$ is automatically $\nabla$-semistable. So the above result generalizes \cite[Theorem 2.5]{Simpson}. On the other hand, the $\nabla$-semistability condition in the statement is indeed necessary for its truth over a general field. Let $k$ be a field of characteristic $p\geq 3$ and $V=\sO\oplus \sO(p)$ be the rank two vector bundle over $\P^1_k$, equipped with the canonical connection $\nabla_{can}$ of the Cartier descent theorem. Then $(V,\nabla_{can})$ admits \emph{no} gr-semistable filtration.
 \end{remark}

To prove the above theorem, we need some lemmas.

\begin{lemma}\label{maxsubbundle}
Let $(E, \theta)$ be a graded Higgs module on $Y$. If $(E, \theta)$ is unstable as a Higgs module, then its maximal destabilizing subsheaf $I\subset E$ is saturated, and it is a sub graded Higgs module, that is $I=\oplus_{i=0}^n I^i$ with $I^i:=I\cap E^{i}$.
\end{lemma}
\begin{proof}
The saturated property of the Higgs subsheaf $I\subset E$ follows from the maximality. To show the second property, one chooses $t\in k$ such that $t^i\neq 1 $ for $0<i\leq n$. Note that there is an isomorphism $f: (E, \theta)\to (E, \tfrac{1}{t}\theta)$ given by $f|_{E^i}=t^i\text{Id}$. Because of the uniqueness of the maximal destabilizing subobject, we see that $f(I)=I$. Let $s$ be any local section of $I$. Write  $s$ as  $\sum_{i=0}^ns^i$, where $s^i$ is a local section of $E^i$. Then for $j\geq 0$,
\begin{align*}
f^{j} (s)=\sum_{i=0}^nt^{ji}s^i \in I.
\end{align*}
Consider
 $$
 \begin{pmatrix}
s\\
f(s)\\
\vdots\\
f^n(s)
 \end{pmatrix} =
 \begin{pmatrix}
   1 & 1 &1 &\cdots &1 \\
   1 & t &t^2&\cdots &t^n\\
   \vdots & \vdots& \vdots& \ddots & \vdots \\
   1 & t^n& t^{2n}& \cdots &t^{n^2}
 \end{pmatrix} \cdot
 \begin{pmatrix}
s^0\\
s^1\\
\vdots\\
s^n
 \end{pmatrix}.
$$
By assumption on $t$, the coefficient matrix is invertible; thus all $s^i$'s are local sections of $I$ and $I =\oplus_{i=0}^n I\cap E^{i}$.
\end{proof}

Let $(V, \nabla)$ be a flat module over $Y$. Start with an arbitrary Griffiths transverse filtration $Fil$ of level $n$ and consider the associated Higgs module $(\text{Gr}_{Fil}(V), \theta)$. If  $(\text{Gr}_{Fil}(V), \theta)$  is unstable, let $I_{Fil}$  be  its  maximal destabilizing subobject. By Lemma \ref{maxsubbundle},
\begin{align}\label{Eq120}
I_{Fil}=\oplus_{i=0}^nI_{Fil}^{i}, \quad  I_{Fil}^i\subset Fil^i/Fil^{i+1} \subset
V/Fil^{i+1}.
\end{align}
Following the construction of Simpson (see \S3 \cite{Simpson}), we define an operation $\xi$ on the set of Griffiths transverse filtrations. The new filtration $\xi(Fil)$ of $V$ is given by
\begin{align}\label{Eq122}
\xi(Fil)^{i+1}:=\text{Ker}(V\to  \frac{ V/Fil^{i+1}}{I_{Fil}^i} ), \text{
for }   0\leq i\leq n; \quad \xi(Fil)^0=V.
\end{align}
Note that $Fil^i\supset \xi(Fil)^{i+1}\supset Fil^{i+1}$, and there is a short exact sequence:
\begin{align}\label{Eq124}
0\to \text{Gr}^i_{Fil}(V)/I_{Fil}^i\to \text{Gr}^i_{\xi(Fil)}(V) \to I_{Fil}^{i-1}
\to 0, \quad \text{for }  0\leq i\leq n+1.
\end{align}
Adding altogether, we obtain a short exact sequence of graded Higgs modules:
\begin{align}\label{exactseq}
0\to \text{Gr}_{Fil}(V)/I_{Fil}\to \text{Gr}_{\xi(Fil)}(V) \overset{h}\to I_{Fil}^{[1]} \to 0,
\end{align}
where $E^{[k]}$ denotes for the graded Higgs module $E$ with index shifted so that $(E^{[k]})^i=E^{i-k}$. If $(E, \theta)$ is unstable, let $\mu_\text{max}(E)$ and $r_\text{max}(E)$ denote respectively the slope and rank of the maximal destabilizing subobject of $E$; otherwise, let $\mu_\text{max}(E)=\mu(E)$ and $r_\text{max}(E)=\text{rk}(E)$.  By (\ref{exactseq}), we have
\begin{lemma}\label{invariants}
The following statements hold:
\begin{enumerate}
\item $\mu_\text{max}(\text{Gr}_{\xi(Fil)}(V)) \leq  \mu_\text{max}(\text{Gr}_{Fil}(V)) $.

\item  If $\mu_\text{max}(\text{Gr}_{\xi(Fil)}(V)) =  \mu_\text{max}(\text{Gr}_{Fil}(V)) $, then $$r_\text{max}(\text{Gr}_{\xi(Fil)}(V)) \leq r_\text{max}(\text{Gr}_{Fil}(V)) .$$

\item  Furthermore, if $r_\text{max}(\text{Gr}_{\xi(Fil)}(V)) = r_\text{max}(\text{Gr}_{Fil}(V))$, then the composite map $ I_{\xi(Fil)} \to \text{Gr}_{\xi(Fil)}(V)\to I_{Fil}^{[1]}$ is injective and is an isomorphism outside a codimension two closed subset of $Y$, where $ I_{\xi(Fil)}$  is  the maximal destabilizing subobject of $\text{Gr}_{\xi(Fil)}(V)$.
\end{enumerate}
\end{lemma}

\begin{proof}
The proof follows from the following exact sequence:
\begin{align}\label{exactseq1}
0\to \text{Ker}(h) \cap  I_{\xi(Fil)} \to  I_{\xi(Fil)}\to h( I_{\xi(Fil)}) \to 0,
\end{align}
 which is induced from (\ref{exactseq}), and where $h( I_{\xi(Fil)})$ is a subsheaf of $ I_{Fil}^{[1]}$.
\end{proof}

The following lemma is the key to our theorem.

\begin{lemma}\label{terminal}
Let $Fil$ be a Griffiths transverse filtration of level $n$ on a flat module $(V,\nabla)$. Assume $(V,\nabla)$ to be $\nabla$-semistable. Then if the associated-graded Higgs module $(\text{Gr}_{Fil}(V), \theta)$ is unstable, then at least one of the following two strict inequalities holds:
\begin{align*}
\mu_\text{max}(\text{Gr}_{\xi^{n+1}(Fil)}(V))<  \mu_\text{max}(\text{Gr}_{Fil}(V)); \quad   r_\text{max}(\text{Gr}_{\xi^{n+1}(Fil)}(V)) < r_\text{max}(\text{Gr}_{Fil}(V)).
\end{align*}
\end{lemma}

\begin{proof}
Put $\mu_k=\mu_\text{max}(\text{Gr}_{\xi^k(Fil)}(V))$ and $r_k=r_\text{max}(\text{Gr}_{\xi^k(Fil)}(V))$ for $k\geq 0$. By Lemma \ref{invariants} (1)-(2), $(\mu_k, r_k)$ decreases in the lexicographic ordering when $k$ grows. Argue by contradiction. Suppose on the contrary that $\mu_{n+1}=\mu_0$ and $r_{n+1}=r_0$. Then, for $0\leq k\leq n$,  $\mu_{k+1}=\mu_{k}$ and $r_{k+1}=r_k$, and, by Lemma \ref{invariants} (3), $I_{\xi^{k+1}(Fil)}\subseteq I_{\xi^{k}(Fil)}^{[1]}\subseteq I_{Fil}^{[k+1]}$, which are locally free and coincide with each other away from a codimension two closed subset $Z\subset Y$. Hence they have the same slope. \\

A direct calculation on the short exact sequences (\ref{exactseq}) for $Gr_{\xi^{k+1}(Fil)}(V)$, for $k$ running from $0$ to $n$, reveals the following fact: the short exact sequence (\ref{exactseq}) of \emph{graded} Higgs modules for $Gr_{\xi^{n+1}(Fil)}(V)$ takes a special form:
\begin{align} \label{seq4}
0\to \oplus_{i=0}^{n}E^i \to \oplus_{i=0}^{2n+1}E^i \to \oplus_{i=n+1}^{2n+1} E^{i} \to 0,
\end{align}
where
$$
Gr_{\xi^{n+1}(Fil)}(V)=Gr_{\xi^{n+1}(Fil)}(V)=\oplus_{i=0}^{2n+1}E^i,
$$
and
$Gr_{\xi^n(Fil)}(V)/I_{\xi^n(Fil)}=\oplus_{i=0}^{n}E^i$ and $I_{\xi^{n}(Fil)}^{[1]}=\oplus_{i=n+1}^{2n+1}E^{i}$. Since over the open subset $U=Y-Z$, $$I_{\xi^{n}(Fil)}^{[1]}|_{U}=I_{\xi^{n+1}(Fil)}|_{U}\subset Gr_{\xi^{n+1}(Fil)}(V)|_{U},$$ it follows that the sequence (\ref{exactseq}) splits over $U$ as graded Higgs modules. Thus, one has
$$
\theta|_{U}(E^{n+1}|_{U})=0.
$$
Hence, $\theta$ is indeed zero on $E^{n+1}$, which means nothing but $V':=(\xi^{n+1}(Fil))^{n+1}$ is $\nabla$-invariant. On the other hand, one has the relation that
$$
\mu(V')=\mu(Gr_{\xi^{n+1}(Fil)}(V'))=\mu(I_{\xi^{n}(Fil)}^{[1]})=
\mu(I_{Fil}^{[n+1]})=\mu(I_{Fil})>\mu(V).
$$
The strict inequality contradicts the $\nabla$-semistability of $(V,\nabla)$. This completes the lemma.
\end{proof}

\begin{proof}[Proof of Theorem \ref{goodfil}]
One takes an arbitrary Griffiths transverse filtration $Fil$ of $(V, \nabla)$ (e.g. the trivial filtration), and then applies consecutively the operator $\xi$ on $Fil$. The meanings of $\mu_k$ and $r_k$ for $k\geq 0$ are the same as those in the previous lemma. Lemma \ref{invariants} says that pairs $(\mu_k,r_k), k\geq 0$ decrease in the lexicographic ordering as $k$ grows. So for certain $k_0\geq 0$, the sequence $\{(\mu_k, r_k)\}_{k\geq k_0}$ becomes constant. Then Lemma \ref{terminal} asserts that $\xi^{k_0}(Fil)$ has to be a gr-semistable filtration of $(V,\nabla)$.
\end{proof}

\begin{remark}\label{canonical}
In the proof of Theorem \ref{goodfil}, if we start with the trivial filtration $V=Fil^0\supset Fil^1=0$, the resulting gr-semistable filtration has the extra property that it is invariant under any automorphism and hence has the same definition field as $(V,\nabla)$. In a Griffiths transverse filtration $Fil: V=Fil^0\supseteq Fil^1\supseteq\cdots$ on $(V,\nabla)$, we call a term $Fil^i, i\geq 1$ \emph{redundant} if $Fil^{i-1}=Fil^{i}$. One can remove all redundant terms from the filtration and shift the indices correspondingly so that the resulting filtration is a strictly decreasing filtration of form $Fil_{red}: V=Fil^0\supsetneq Fil^1\supsetneq Fil^2\supsetneq\cdots$. We call this operation the \emph{reduction} of a filtration, and a filtration \emph{reduced} if it is equal to its reduction. It is not hard to observe that  $Gr_{Fil}(V,\nabla)$ and $Gr_{Fil_{red}}(V,\nabla)$ are isomorphic as Higgs modules. Thus, the reduction of a gr-semistable filtration on $(V,\nabla)$ is again gr-semistable.
\end{remark}

\begin{proof}[Proof of Theorem \ref{strsst}]
Note first that the inverse Cartier transform $(V,\nabla)$ of a Higgs module $(E,\theta)$ is $\nabla$-semistable if and only if $E$ is $\theta$-semistable. This is a direct consequence of the equivalence theorem of the (inverse) Cartier transform of Ogus-Vologodsky \cite{OV}. So for a
nilpotent semistable Higgs module of rank $\leq p$, its inverse Cartier transform is a $\nabla$-semistable flat module of rank $\leq p$. By Theorem \ref{goodfil} and Remark \ref{canonical}, there exists a Griffiths transverse filtration $Fil$ on $(V,\nabla)$ (of level $\leq p-1$ for the reason of rank) such that $Gr_{Fil}(V,\nabla)$ is nilpotent semistable of the same rank $\leq p$ and defined over the same ground field of $(E,\theta)$. Therefore, we can obtain a semistable Higgs-de Rham flow with the leading term $(E,\theta)$ by applying the previous construction inductively. This completes the proof.
\end{proof}

{\bf Acknowledgements:} Christopher Deninger has drawn our attention to the work \cite{Langer}, Adrian Langer has helped us understanding \cite{Langer} and Carlos Simpson has explained us the proof of \cite[Proposition 4.3]{Simpson}. We also appreciate several e-mail communications with Mark Kisin. We thank them heartily. We would like to thank especially an anonymous referee who has pointed out several drawbacks in the original draft and also provided us a very constructive suggestion in the construction of the lifting of the inverse Cartier transform over a truncated Witt ring. The second approach in the construction of the functor $\sT_n$ in \S4 is due to the referee.

\end{document}